\documentclass[12pt]{amsart}
\usepackage{amssymb}
\usepackage{amsmath}
\usepackage{mathrsfs}
\usepackage{tikz}
\usetikzlibrary{decorations.fractals}
\usepackage{esint}
\usepackage{amsfonts}
\usepackage{xcolor}
\usepackage{enumerate}
\usepackage{amssymb,amsmath}
 \usepackage[colorlinks, linkcolor=black, citecolor=blue, pagebackref, hypertexnames=false]{hyperref}
\def\C {{\mathbb C}}
\def\Z {{\mathbb Z}}

\def\R {\mathbb{R}}
\def\N {\mathbb{N}}

\def\d{{\rm d}}

\newtheorem{proposition}{Proposition}[section]
\newtheorem{theorem}[proposition]{Theorem}
\newtheorem{corollary}[proposition]{Corollary}
\newtheorem{lemma}[proposition]{Lemma}

\theoremstyle{definition}
\newtheorem{definition}[proposition]{Definition}
\newtheorem{remark}[proposition]{Remark}
\newtheorem{example}[proposition]{Example}
\numberwithin{equation}{section}

\def \no#1#2#3 {{\bf #1} (#3), #2.}
\def \eds#1#2#3 {#1, #2, #3.}

\title[observability inequality]
{Observability inequality, log-type Hausdorff content and heat equations}

\author[S. Huang]{Shanlin Huang}

\address{Shanlin Huang
\newline\indent
School of Mathematics and Statistics, Hubei Key Laboratory of Engineering Modeling and Scientific Computing, Huazhong University of Science and Technology,  Wuhan,  430074,  P.R. China}
\email{shanlin\_huang@hust.edu.cn}

\author[G. Wang]{Gengsheng Wang}

\address{Gengsheng Wang
\newline\indent
Center for Applied Mathematics, Tianjin University, Tianjin 300072, P.R. China}
\email{wanggs62@yeah.net}

\author[M. Wang ]
{ Ming Wang}

\address{Ming Wang
\newline\indent
School of Mathematics and Statistics, HNP-LAMA, Central South University, Chang-sha, Hunan 410083, P.R. China
}
\email{m.wang@csu.edu.cn}

\subjclass[2000]{}
\keywords{}

\begin{document}

\begin{abstract}
This paper studies observability inequalities for heat equations on both bounded domains and the whole space $\R^d$. The observation sets are measured by log-type Hausdorff contents, which are induced by certain log-type gauge functions closely related to the heat kernel. On a bounded domain, we derive the observability inequality for observation sets of positive log-type Hausdorff content. Notably, the  aforementioned inequality holds not only for all sets with Hausdorff dimension $s$ for any $s\in (d-1,d]$, but also for certain sets
of Hausdorff dimension $d-1$.
   On the whole space $\R^d$, we establish the observability inequality for observation sets that are thick at the scale of the log-type Hausdorff content.  Furthermore, we prove that for the 1-dimensional heat equation on an interval, the Hausdorff content we have chosen is an optimal scale for the observability inequality.

To obtain these observability inequalities, we use the adapted Lebeau-Robiano strategy from \cite{Duyckaerts2012resolvent}. For this purpose, we prove the following results at scale of the log-type Hausdorff content, the former being derived from the latter:
 We establish
 a spectral inequality/a Logvinenko-Sereda uncertainty principle; we set up a quantitative propagation of smallness of analytic functions;
 we build up a Remez' inequality;  and more
 fundamentally,  we provide
   an upper  bound for
  the log-type Hausdorff content of
  a set where a monic  polynomial is small,  based on an estimate in Lubinsky \cite{Lubinsky1997small}, which
is ultimately traced back to  the classical Cartan Lemma.
 In addition, we set up a capacity-based slicing lemma (related to the log-type gauge functions) and establish a quantitative relationship between Hausdorff contents and capacities. These tools are crucial in the studies of the aforementioned propagation of smallness in high-dimensional situations.

\end{abstract}
\maketitle


\tableofcontents

\section{Introduction}
The observability inequality of the heat equation represents a form of quantitative unique continuation, which
 is widely applicable in  control theory and inverse problems.

 We begin by considering  the following heat equation:
 \begin{align}\label{equ-heat-high-bd}
\left\{
\begin{array}{l}
  \partial_tu(t,x)=\Delta u(t,x),\quad\, t>0,x\in \Omega, \\
   u(t,x)=0,\quad\quad \quad \quad \quad  t\geq 0,x\in \partial\Omega,\\
  u(0,\cdot)\in L^2(\Omega),
\end{array}
\right.
\end{align}
where   $\Omega\subset \mathbb{R}^d$ (with $d\geq 1$) is a bounded domain with a smooth boundary $\partial\Omega$.
We observe a solution of the above equation over $[0,T]\times E$, where $T>0$ and $E\subset\Omega$.
  The standard observability inequality is stated as follows: For each $T>0$, there exists  an {\it observability  constant
$C_{obs}>0$},  depending only on $T$, $d$, $\Omega$ and $E$, such that any solution $u$ to \eqref{equ-heat-high-bd} satisfies
\begin{align}\label{equ-ob-intro}
   \|u(T,\cdot)\|_{L^2(\Omega)}\leq C_{obs}\left(\int_0^T\int_E|u(t,x)|^2\d x\d t \right)^{1/2}.
\end{align}
When the above  holds,  we call $E$ {\it an {\it observable set} of \eqref{equ-heat-high-bd}.}
 The same can be said about  the heat equation over $\mathbb{R}^d$.

There is a vast body of literature on the observability inequality of  heat equations. Here, we only recall some of them, which are most pertinent to the goal of the current paper.
When $E\subset\Omega$ is open,   the inequality \eqref{equ-ob-intro} was derived, in general, by a global Carleman estimate
\cite{FurIma, FuLvZhang}, or the spectral inequality method \cite{Lebeau1995contr}, or the frequency function method \cite{Phung2013anob}.
When $E\subset\Omega$ is a subset of positive $d$-dim Lebesgue measure,
 \eqref{equ-ob-intro} was first proved in \cite{ApraizNullcontrol2013} (see also \cite{Apraiz2014obser}), where the propagation of smallness from measurable sets for analytic functions plays a key role.
 For the heat equation over $\mathbb{R}^d$, where $E$ has a positive $d$-dim Lebesgue measure, \cite{Egidisharp2018,Wang2019observable}
  proved that
  the observability  inequality holds if and only if $E$ is a thick set.

 Recently, Burq and Moyano \cite{Burq2023propagation} established a new observability inequality for the heat equation on a $W^{2,\infty}$ compact manifold $(\Omega, g)$ of dimension $d$, where the observation set $E\subset\Omega$ is allowed to have zero $d$-dim Lebesgue measure.
 More precisely, it was shown that   if  $E$ is a subset of positive $(d-1+\delta)$-Hausdorff content for some $\delta\in(0,1)$ close to $1$,  then for each $T>0$, there is a positive constant $C_{obs}$, depending only on $T$, $d$, $\Omega$, $\delta$ and $E$, such that any solution $u$ to \eqref{equ-heat-high-bd}, with the  operator $-\Delta$ replaced by
 a uniform elliptic operator $-\mbox{div}(A(x)\nabla\cdot)$,
 satisfies
 \begin{align}\label{equ-ob-intro-2}
   \|u(T,\cdot)\|_{L^2(\Omega)}\leq C_{obs}\int_0^T\sup_{x\in E}|u(t,x)| \d t.
\end{align}
The result does not cover the scenario where $\delta$ is near $0$ for some technical reasons.
Indeed, their proof is based on the following spectral inequality (built up in  \cite{Burq2023propagation}):
\begin{align}\label{equ-ob-intro-4}
  \sup_{\Omega}|\phi|\leq Ce^{C\sqrt{\lambda}}\sup_E|\phi|\;\;\mbox{for each}\;\; \phi\in \mathcal{E}_\lambda,\,\,\lambda>0,
\end{align}
where $E$ and $\Omega$ are the same as above and $\mathcal{E}_\lambda$ denotes the space spanned by the eigenfunctions
of $-\Delta$ (with the zero Dirichelet boundary condition) corresponding to the eigenvalues
 not exceeding $\lambda$.
A crucial point  in proving \eqref{equ-ob-intro-4}, as detailed in
\cite{Burq2023propagation},  is the propagation of smallness for the gradient of solutions to $\mbox{div}(A(x)\nabla u)=0$, obtained in   \cite{Logunov2018quantitative}, which shows that
\begin{align}\label{equ-ob-intro-3}
\sup_{B_1}|\nabla u|\leq (\sup_{E}|\nabla u|)^\beta (\sup_{B_2}|\nabla u|)^{1-\beta},
\end{align}
where $\beta\in(0,1)$, and $E$ is a set of positive $(d-2+\delta)$-Hausdorff content for some $\delta\in(0,1)$ close to $1$. However, as mentioned in \cite{Logunov2018quantitative}, it remains unclear  whether \eqref{equ-ob-intro-3} holds for $\delta$ close to $0$.
Moreover, Le Balc'h and Martin \cite{LeBalch2403} proved that the observability inequality \eqref{equ-ob-intro-2} still holds  for parabolic equation $\partial_tu+Hu=0$ on bounded domains, where $H$ is the sum of a divergence elliptic operator and a bounded potential. They also obtained a corresponding result on the whole space $\R^d$.

 More recently, Green et al. \cite{walton2024observability}  refined the result in  \cite{Burq2023propagation,LeBalch2403} for the  heat equation (with constant coefficients) on $\Omega$.  In particular, they proved that
 any $E \subset \Omega$ with positive $(d-1+\delta)$-Hausdorff content (where $\delta>0$ can be arbitrarily small) is an observable set of \eqref{equ-heat-high-bd}.
 The proof also relies on   \eqref{equ-ob-intro-4}, while the main tool used to obtain it is the Bang-type estimate for analytic functions in terms of the transfinite diameter of a set.
As mentioned in \cite{walton2024observability}, \eqref{equ-ob-intro-4} can also be proved by  the propagation of smallness for solutions
to generalized Cauchy-Riemann systems in  \cite{Malinnikova2004propaga}. Moreover, the result in \cite{walton2024observability} is sharp in the sense that there exists a subset $E\subset\Omega$ of positive $(d-1)$-Hausdorff content such that \eqref{equ-ob-intro-2} fails.  Indeed, it can  be chosen as follows:
\begin{align}\label{equ-nodal}
    E:=\{x\in \Omega: \phi_\lambda(x)=0\},
\end{align}
where  $\phi_\lambda$ is an eigenfunction of the Laplacian corresponding to the eigenvalue $\lambda>0$.
It is well known that the $(d-1)$- Hausdorff content of the above $E$ is positive (see  \cite[Theorem 6.2.5]{HL}) or
 \cite{LMNN21}).

 From the main result in \cite{walton2024observability} mentioned above, we see that
 any $E \subset \Omega$
with Hausdorff dimension greater than $d-1$ is an observable set.\footnote{If a subset $E$ has  positive $(d-1+\delta)$-Hausdorff content, then it  has  positive $(d-1+\delta)$-Hausdorff measure, hence $\dim E \geq d-1+\delta$. Conversely, if $\dim E \geq d-1+\delta$, then the $(d-1+\frac{\delta}{2})$-Hausdorff measure of $E$ is infinite, and thus  the $(d-1+\frac{\delta}{2})$-Hausdorff content of $E$ is positive. When $E$ is a set with dimension less than $d$, we call it an observable set if \eqref{equ-ob-intro-2} holds. }
It deserves noting that the sets of Hausdorff dimension $d-1$ are in a critical case which is quite subtle.
 On one hand, there exist sets with Hausdorff dimension $d-1$ that are not observable, such as the set given by \eqref{equ-nodal}. On the other hand, some specific observable subsets (in $\mathbb{R}^d$) with Hausdorff dimension $d-1$ or even $d-2+\delta$ have been given in  \cite{Dolecki1973, Samb2015, walton2024observability}, however, these observable sets are not characterized  by any generalized Hausdorff content.\footnote{Here and in what follows, when we say  that observable sets are characterized by a generalized Hausdorff content $c_F$ (with $F$ a gauge function), we mean that if  $c_F(E)>0$, then $E$ is an observable set. Hence,
$c_F(E)>0$ serves as a sufficient condition for  $E$ to be an observable set, and $c_F$ plays an important role in this characterization.}

Thus, the following question arises naturally: \textit{Can we find a generalized Hausdorff content $c_F$, induced by an appropriate gauge function $F$, that characterizes the observable sets of the equation \eqref{equ-heat-high-bd}? Furthermore, whether this family of observable sets, characterized by $c_F$, includes some  sets with Hausdorff dimension $d-1$?}

In this paper,
we address this question by introducing
the
 Hausdorff content $c_{F_{\alpha,\beta}}$, induced by the following  log-type gauge function:
 \begin{align}\label{equ-intro-F}
F_{\alpha,\beta}(t)=t^{d-1}\Big(\log \frac{1}{t}\Big)^{-\beta}\Big(\log\log\frac{1}{t}\Big)^{-\alpha},\;\;t\in (0,e^{-3}],
\end{align}
where $\alpha\geq 0$ and $\beta\geq 0$.
The reasons for choosing  the above  gauge function are explained as follows:

First, $c_{F_{\alpha,\beta}}$ characterizes observable sets of the equation \eqref{equ-heat-high-bd}.
 In fact,
we
have shown that when $E\subset\Omega$  has a positive $F_{\alpha,\beta}$-Hausdorff content with $\beta=\frac{3}{2}$ and $\alpha>\frac{5}{2}$, it is an observable set of the equation \eqref{equ-heat-high-bd}  (see Theorem~\ref{thm-bound-high}).
 Moreover, the family of observable sets characterized by $c_{F_{\alpha,\beta}}$ contains not only all sets $E \subset \Omega$ with Hausdorff dimension $s$ for any $s\in (d-1,d]$, but also certain sets with Hausdorff dimension exactly $d-1$, which, notably, have apparently not been identified in previous research (see Remarks \ref{rmk-high-d} and \ref{rem-d-1}).

Second, in the one-dimensional case, the Hausdorff  content $c_{F_{\alpha,\beta}}$
serves as a  sharp scale  for observable sets of the heat equation \eqref{equ-heat-high-bd} in the following sense. On one hand, we show that $E\subset \Omega$ is an observable set if  $E$ has a positive $F_{\alpha,\beta}$-Hausdorff content, where $\beta = \frac{1}{2}$ and $\alpha > \frac{3}{2}$ (see Theorem \ref{thm-ob-bound} (i)). On the other hand, for each $\varepsilon>0$, we construct  a subset $E_{\infty}$   possessing a positive $F_{0,\frac{1}{2+\varepsilon}}$-Hausdorff content, which is not an observable set of \eqref{equ-heat-high-bd} (see Theorem \ref{thm-ob-bound} (ii) or Proposition \ref{thm-1017-1.1}). This means that $F_{0,\frac{1}{2}}(t) = \left(\log \frac{1}{t}\right)^{-\frac{1}{2}}$ is a critical gauge function. It is interesting to note that the inverse function of $F_{0, \frac{1}{2}}$ satisfies
$$
F^{[-1]}_{0,\frac{1}{2}}\left(\frac{1}{x}\right) = e^{-x^2},
$$
which aligns to the heat kernel.

Third, we are motivated by the study of some classical problems. For example,
Brownian motion trajectories in $\R^d$ have positive  $h$-Hausdorff measures with  gauge functions
$$
h(t)=
\left\{
\begin{array}{ll}
t^2\log\log \frac{1}{t},     &  \mbox{ if } d\geq 3, \\
t^2\log \frac{1}{t}\log\log\log \frac{1}{t},    &
 \mbox{ if } d=2
\end{array}
\right.
$$
for $t$ small, see \cite{C-Taylor1962} and  \cite[p.60]{Mattila1995geometry}.
Besides, Besicovitch sets in $\R^2$
  also exhibit a generalized Hausdorff dimension lies between $t^2\log\frac{1}{t}$ and $t^2(\log \frac{1}{t})(\log\log \frac{1}{t})^{2+\varepsilon}$ for any   $\varepsilon>0$, see  \cite{Keich1999,ChenYanZhong}.

\section{Main results}\label{mainresult-w-10-29}
This section presents the main results of this paper. The results for the $d$-dim case ($d\geq 2$)  are slightly weaker than those for the $1$-dim case, due to some technical  reasons which will be explained in the note $(ii)$ of Remark \ref{remark1.11-w-9-22}. Thus, we discuss them separately.

\subsection{The case of $1$-dim}\

 We first consider the  heat equation on the interval $(0,L)$ with $L>0$:
\begin{align}\label{equ-heat-b}
\left\{
\begin{array}{l}
  \partial_tu(t,x)=\partial_x^2 u(t,x),\quad t>0,x\in(0,L),  \\
  u(t,0)=u(t,L)=0, \quad t\geq 0,\\
  u(0,\cdot)\in L^2(0,L).
\end{array}
\right.
\end{align}
For each $\alpha\geq 0$, we define the following gauge function (see Definition \ref{def-1}):
\begin{align}\label{1.7w9-21}
h_{\alpha}(t):=(\log \frac{1}{t})^{-\frac{1}{2}}(\log\log\frac{1}{t})^{-\alpha}, \quad 0<t\leq e^{-3}.
\end{align}
We denote by  $c_{h_{\alpha}}$ the corresponding  Hausdorff content, with further details provided  in Section \ref{section 2}.
\begin{theorem}\label{thm-ob-bound}
$(i)$ Let $\alpha>\frac{3}{2}$ and let $E\subset (0,L)$ be a subset with  $c_{h_{\alpha}}(E)>0$. Then for each $T>0$, there is
$C_{obs}>0$, depending only on $T$, $L$ and $c_{h_{\alpha}}(E)$, such that any solution $u$ to \eqref{equ-heat-b} satisfies
\begin{equation}\label{thm-ob-1024h-1}
   \|u(T,\cdot)\|_{L^2(0,L)}\leq C_{obs}\int_0^T\sup_{x\in E}|u(t,x)|\d t.
\end{equation}

$(ii)$ For every $\varepsilon>0$, there is a subset $E_{\infty}\subset (0, L)$ of positive $(\log{\frac{1}{t}})^{-\frac{1}{2+\varepsilon}}$-Hausdorff content such that  \eqref{thm-ob-1024h-1}, with $E$ replaced by $E_{\infty}$, fails.
\end{theorem}

\begin{remark}\label{remark1.2w}
$(i)$
The first conclusion in
Theorem \ref{thm-ob-bound}
says that by measuring a solution  over $(0,T)\times E$, one can recover its value
in $L^2(0,L)$ at time $T$, and then
recover its initial datum by the backward uniqueness of the heat equation.
It yields the null controllability for the heat equation \eqref{equ-heat-b}, which will be precised in Section
\ref{section-contr}.

$(ii)$
 The second conclusion in Theorem \ref{thm-ob-bound}
 exhibits
that  the Hausdorff content induced by the
 $\log$-type gauge function $h_0(t)=(\log{\frac{1}{t}})^{-\frac{1}{2}}$
 is the optimal scale for measuring the observation set
of the equation \eqref{equ-heat-b}.  
It is also critical from the perspective of spectral inequality. Indeed,
selecting $(\log{\frac{1}{t}})^{-\frac{1}{2}}$ as the gauge function leads to a spectral inequality of the form
\begin{equation*}
	\sup_{\Omega}|\phi|\leq Ce^{C\lambda}\sup_E|\phi|\;\;\mbox{for each}\;\; \phi\in \mathcal{E}_\lambda,
\end{equation*}
where $E$ is a subset of positive $(\log{\frac{1}{t}})^{-\frac{1}{2}}$-Hausdorff content.  (See \eqref{spec-ine-1-dim-9-23-w} and \eqref{eq-1019h-01}, with $\alpha=0$.)
Note that the growth rate $e^{C\lambda}$ is exactly balanced with the dissipative effect of the heat equation, thus the well-known Lebeau-Robbiano strategy  fails in this critical case. It is this reason that  we have incorporated a $\log\log$ factor into the gauge function, thereby obtaining a spectral inequality/an uncertainty principle of the forms \eqref{spec-ine-1-dim-9-23-w}/\eqref{eq-1019h-01}, which
enables us to apply the Lebeau-Robbiano strategy.

Our construction of $E_\infty$ strongly relies on the explicit expression of the heat kernel and the intimate connection between the gauge function $F_{0,\frac{1}{2}}(t) = \left(\log \frac{1}{t}\right)^{-\frac{1}{2}}$ and the heat kernel, as mentioned in the introduction.


\end{remark}

Next, we consider the following 1-dim heat equation over $\R$:

\begin{align}\label{equ-heat-1-dim-w}
\left\{
\begin{array}{l}
  \partial_t u(t,x)=\partial_x^2 u(t,x), \quad t>0,x\in \R,\\
   u(0,\cdot)\in L^2(\R).
\end{array}
\right.
\end{align}
To state the corresponding observability inequality, we need the following definition:
\begin{definition}\label{thich-set-9-21-w}
Let $f$ be a gauge function and $c_f$ be the Hausdorff content defined by $f$.
 We say that a subset $E\subset\R$ is  $f$-$(\gamma, L)$ thick, if there are positive constants $\gamma$ and $ L$ such that
$$
\inf_{x\in \R}c_{f}(E\cap[x,x+ L])\geq \gamma  L.
$$
\end{definition}

\begin{remark}\label{remark-content-thik-set}
$(i)$ Definition \ref{thich-set-9-21-w} broadens the concept of thick set from sets with positive 1-dim Lebesgue measure to
those with  positive $c_f$-Hausdorff content. In particular, when $f(t)=t,\;t>0$, an $f$-$(\gamma, L)$ thick set coincides with a classical $(\gamma, L)$
thick set, see \cite{Wang2019observable}.

$(ii)$  Definition \ref{thich-set-9-21-w} can be extended to sets in $\mathbb{R}^d$ ($d\geq 2$)  as follows:
A subset $E\subset \R^d$ is said to be $f$-$(\gamma, L)$ thick, if
$$
\inf_{x\in \R^d}{c_{f}(E\cap Q_{ L}(x))}\geq \gamma|Q_{ L}(0)|,
$$
where $Q_{ L}(x)$ denotes the closed cube in $\R^d$ centered at $x$ with side length $L$, and $|Q_{ L}(0)|$
denotes the volume of $Q_{ L}(0)$.
\end{remark}
\begin{theorem}\label{thm-ob-R}
Let $\alpha>\frac{3}{2}$ and let $E$ be an $h_\alpha$-$(\gamma, L)$ thick set for some constants $\gamma,  L>0$. Then for each $T>0$, there is
$C_{obs}>0$, depending only on $T$, $\alpha$, $E$,  $\gamma$ and $L$, such that any solution $u$ to \eqref{equ-heat-1-dim-w}
satisfies
\begin{align}\label{2.5-w-10-26}
\|u(T,\cdot)\|_{L^2(\R)}\leq C_{obs}\int_0^T\Big\|\big\{\sup_{x\in E\cap [k,k+ L]}|u(t,x)|\big\}_{k\in \Z}\Big\|_{l^2(\Z)}\d t.
\end{align}
\end{theorem}

\begin{remark}\label{remark-strategy-thm1} (Strategy for proving Theorem \ref{thm-ob-bound})
We outline our proof strategy in the following steps.
\vskip 5pt

 \noindent {\it Step 1: An upper bound on level set.} We establish the following  upper  bound for
  the log-type Hausdorff content of
  a set where a monic  polynomial is small (i.e.,
Lemma \ref{lem-ch}):
If $P$ is a moinc polynomial of degree $n$,  then for each $\alpha>0$,
\begin{align}\label{fundamental-estimate-w-9-23}
  c_{h_\alpha}(P;\delta)\lesssim (\log \frac{1}{4\delta})^{-\frac{1}{2}}  n^{\frac12} (\log (n+e))^{-\alpha/3},\;\;\mbox{when}\;\;0<\delta \ll 1,
\end{align}
where
$ c_{h_\alpha}(P;\delta):=c_{h_\alpha}(\{z\in \C: |P(z)|\leq \delta^n\})$.
 The proof of \eqref{fundamental-estimate-w-9-23} is based on an estimate (in Lemma \ref{lem-abs}) about the generalized Hausdorff content of a set where a monic polynomial is small. The latter
is ultimately traced back to  the classical Cartan Lemma.
\vskip 5pt

\noindent {\it Step 2: Log-type Remez's inequality}. Based on \eqref{fundamental-estimate-w-9-23}, we derive  the following log-type Remez's inequality (i.e.,  Lemma \ref{lem-remez}): If $\alpha>0$ and $E\subset \R$ is a subset with
$c_{h_\alpha}(E)>0$, then there is $C>0$, depending only on $\alpha$, such that for all polynomials $P$ of degree $n$,
$$
\max_{|z|\leq 1}|P(z)|\leq  12^n\exp\Big\{\frac{C n^2(\log (n+e))^{-\frac{2\alpha}{3}}}{c^2_{h_\alpha}(E)}\Big\}\max_{E}|P(z)|.
$$
\vskip 5pt
\noindent {\it Step 3: Quantitative propagation of smallness.} We apply the inequality obtained in {\it Step 2 } to establish the quantitative propagation of smallness for analytic function
 at the scale of $c_{h_\alpha}$ (i.e., Lemma \ref{lem-ana}):
  If $\alpha>0$ and $E\subset I$ (where $I$ is a unit-length
  interval containing $0$) is a subset with
$c_{h_\alpha}(E)>0$,  then there is $C>0$, depending only on $\alpha$, such that
any analytic function $\phi:D_5(0)\mapsto \C$ with $|\phi(0)|\geq 1$ satisfies
\begin{align}\label{equ-ob-intro-6}
 \sup_{x\in I}|\phi(x)|\leq M^9\exp\Big\{\frac{C (\log M)^2(\log (\frac{\log M}{\log 2}+e))^{-\frac{2\alpha}{3}}}{c^2_{h_\alpha}(E)}\Big\}\sup_{x\in E}|\phi(x)|,
\end{align}
where $M=\sup_{z\in D_4(0)}|\phi(z)|$.

\vskip 5pt

\noindent {\it Step 4: Spectral inequality.} We use \eqref{equ-ob-intro-6} to build up the following  spectral inequality at the scale of $c_{h_\alpha}$ (i.e., Lemma \ref{lem-spec}):
If $\alpha>0$ and $E\subset (0,L)$ is a subset with $c_{h_\alpha}(E)>0$, then there is a numerical constant $C>0$ such that
when  $\lambda >0$ and $\phi\in \mathcal {E}_\lambda$,
\begin{align}\label{spec-ine-1-dim-9-23-w}
\sup_{x\in [0,L]}|\phi(x)|\leq e^{C(L+1)^2\sqrt{\lambda}}\exp\{\frac{C(L+1)^2\lambda(\log  (\lambda+e) )^{-\frac{2\alpha}{3}}}{c^2_{h_\alpha}(E)}\}\sup_{x\in E}|\phi(x)|.
\end{align}

\vskip 5pt
\noindent  {\it Step 5: Proof of Theorem \ref{thm-ob-bound}}.
We  employ the   adapted Lebeau-Robiano strategy from \cite{Duyckaerts2012resolvent},  together with \eqref{spec-ine-1-dim-9-23-w},
to prove Theorem \ref{thm-ob-bound}. It is  in this step that the condition $\alpha>\frac{3}{2}$ is required.
\end{remark}

\begin{remark}\label{remark-strategy-thm5} (Strategy for proving  Theorem \ref{thm-ob-R})
We employ a similar five-step approach as outlined in Remark \ref{remark-strategy-thm1},
 with significant modifications in Step 4. Below, we explain these adjustments in detail.


 The counterpart of the Lebeau-Robbiano spectral inequality (see
\cite{Lebeau1995contr})  on the whole space is   the Logvinenko-Sereda uncertainty principle (see, e.g. note $(d1)$ in \cite{Wang2019observable}). Thus, our task is to establish an analogous Logvinenko-Sereda uncertainty principle where the Lebesuge measure
is replaced by the Hausdorff content $c_{h_\alpha}$.

When the observation set has zero  Lebesuge measure, we cannot operate in the space $L^2(\R)$, which is the standard working space for  the Logvinenko-Sereda uncertainty principle. Instead, we choose $l^2L^\infty(\R)$, the
amalgams of {$L^\infty$} and {$l^2$}, as our working space in the current case. With the help of the concept of $h_\alpha$-$(\gamma, L)$ thick sets (see
Definition \ref{thich-set-9-21-w} and $(ii)$ in Remark \ref{remark-content-thik-set}), we build up the following uncertainty principle (i.e., Lemma \ref{lem-LS}):
If $\alpha>0$ and $E\subset\R$ is  $h_\alpha$-$(\gamma, L)$ thick, then there is $C>0$,
 depending only on $\alpha$, $ L$ and $\gamma$,
 such that each $v\in l^2L^\infty$,
with $\emph{supp } \widehat{v}\subset [-N,N]$ for some $N>0$, satisfies
\begin{equation}\label{eq-1019h-01}
\|v\|_{l^2L^\infty}\leq (4C( L+1))^{C LN}e^{CN}\exp\Big\{\frac{CN^2(\log
(\frac{CN}{\log 2}+e))^{-\frac{2\alpha}{3}}}{\gamma^2}\Big\}\big\|\big\{\sup_{E\cap [k,k+L]}|v|\big\}_{k\in \Z}\big\|_{l^2(\Z)}.
\end{equation}
\end{remark}

\subsection{The case of $d$-dim, with $d\geq 2$}\



We first consider the heat equation \eqref{equ-heat-high-bd} on a bounded, smooth domain $\Omega$. Let $F_{\alpha,\beta}$ be the gauge function defined in \eqref{equ-intro-F} and let $c_{F_{\alpha,\beta}}$ be the corresponding Hausdorff
content.
\begin{theorem}\label{thm-bound-high}
Let  $\alpha>\frac{5}{2}$ and $\beta= \frac{3}{2}$.
Let $E\subset\Omega$ be a compact subset  with  $c_{F_{\alpha,\beta}}(E)>0$. Then  for each $T>0$, there is $C_{obs}>0$, depending only on $T$, $d$, $\Omega$ and $c_{F_{\alpha,\beta}}(E)$,  such that
any solution $u$ to \eqref{equ-heat-high-bd} satisfies
$$
\|u(T,\cdot)\|_{L^2(\Omega)}\leq C_{obs}\int_0^T\sup_{x\in E}|u(t,x)|\d t.
$$
\end{theorem}

\begin{remark}
Theorem \ref{thm-bound-high}  extends the scope of observation sets beyond those previously dicussed in \cite{walton2024observability}. In fact, any subset with a positive $(d-1+\delta)$-Hausdorff content has a positive $F_{\alpha,\beta}$-Hausdorff content for any $\alpha,\beta>0$.
  Conversely,  there exist  subsets  with   positive $F_{\alpha,\beta}$-Hausdorff content  for any $\alpha,\beta>0$, while the Hausdorff dimension is $d-1$, see Example \ref{rmk-high-d}.   We refer the reader to Remark \ref{rem-d-1} for more comments on the comparison between our results and those in Dolecki \cite{Dolecki1973} and Samb \cite{Samb2015}.
\end{remark}

\begin{remark}\label{remark 1.9w-9-22}
$(i)$ The assumption that $\Omega$ has a smooth boundary in Theorem \ref{thm-bound-high} can be relaxed. See Remark \ref{rem-1126} for more details. Moreover, Theorem \ref{thm-bound-high} also holds when $\alpha>\frac{5}{2}$ and $\beta\geq\frac{3}{2}$,  following directly from properties of the Hausdorff content.

$(ii)$
There are differences between conditions imposed in
Theorem \ref{thm-ob-bound} and Theorem \ref{thm-bound-high}. First, the values of $(\alpha,\beta)$ are different. Second,
 the observation set  $E$ in Theorem \ref{thm-bound-high} is assumed to  be compact, whereas  Theorem \ref{thm-ob-bound} does not impose such assumption.
 The reason  causing the above differences is  that our method creates a gap between the $1$-dim and the $d$-dim approaches (see $(ii)$ of Remark \ref{remark1.11-w-9-22} for details).

 \end{remark}

Next, we consider the following heat equation over $\R^d$ with $d\geq 2$:
\begin{align}\label{equ-heat-high-Rn}
\left\{
\begin{array}{l}
  \partial_tu(t,x)=\Delta u(t,x),\quad t>0,x\in \R^d, \\
  u(0,\cdot)\in L^2(\R^d).
\end{array}
\right.
\end{align}
Throughout this paper, we shall use $\{Q(k)\}_{k\in \Z^d}$ to denote a sequence of closed unit cubes, centered at $k\in \Z^d$. Then we have
$$
\R^d=\bigcup_{k\in \Z^d}Q(k)   \quad \mbox{ and } \quad |Q(k)\cap Q(j)|=0 \mbox{ for all } k\neq j.
$$

\begin{theorem}\label{thm-Rn-high}
Let  $\alpha>\frac{5}{2}$ and $\beta= \frac{3}{2}$.
Let $E\subset \R^d$ be a closed   $F_{\alpha,\beta}$-$(\gamma, L)$ thick set  for some $\gamma>0$ and $L>0$.
Then  for each $T>0$, there is $C_{obs}>0$, depending only on $T$, $d$, $\alpha$, $\gamma$ and $ L$, such that
any solution $u$ to \eqref{equ-heat-high-Rn} satisfies
$$
\|u(T,\cdot)\|_{L^2(\R^d)}\leq C_{obs}\int_0^T\Big\|\Big\{\sup_{x\in E\cap Q(k)}|u(t,x)|\Big\}_{k\in\Z^d}\Big\|_{l^2(\Z^d)}\d t.
$$
\end{theorem}
By a duality argument, we can apply the observabilities in the above theorems to establish null-controllability for heat equations from a class of sets with Hausdorff dimension $d-1$. For further details, we refer to Section \ref{section-contr}.

\begin{remark}\label{remark1.11-w-9-22} (On the strategy to prove Theorem \ref{thm-bound-high} and Theorem \ref{thm-Rn-high})

$(i)$ The methods used in Theorems \ref{thm-bound-high} and \ref{thm-Rn-high}
parallel those for Theorems \ref{thm-ob-bound}
and \ref{thm-ob-R}, with a key difference at {\it Step 3}.
 The necessity for these modifications arises from the lack of a corresponding upper bound
 on level sets for monic polynomials in the  $d$-dim
 case ($d\geq 2$) in general.
 For instance, consider the polynomial $P(z)=(z_1z_2\ldots z_d)^n$,  $z_j\in\C$, $1\le j\le d$. For any  $\delta>0$, the $2d$-dim  Lebesgue measure of
 the level set $\{z\in \C^d: |P(z)|\leq \delta^n\}$ is infinite. This, in particular,  implies that $c_{F_{\alpha,\beta}}(P;\delta)=\infty$ for any $\alpha,\beta>0$, where $F_{\alpha,\beta}$ is given by \eqref{equ-intro-F}.

$(ii)$ We explain the modifications in  {\it Step 3} for higher dimensions as follows:

For the  $d$-dim  case ($d\ge 2$), we extend  the  $1$-dim inequality \eqref{equ-ob-intro-6} by seeking a suitable slicing theorem.
Unfortunately, we are not able to obtain such kind of  slicing theorem at scale of  log-type Hausdorff contents. Instead, we switch to use capacity-based slicing theorem from fractal geometry (see \cite{Mattila1995geometry}). A well-known such slicing theorem states that
if a set $E\subset\R^d$ has a
positive Riesz capacity $C_s(E)$ (where $s=d-1+\delta$, with $\delta\in(0,1]$),
 then there exists a line $l$ passing the origin, such that the set $a\in l^\bot$, with $C_\delta(E\cap(l+a))\geq \kappa$ for some $\kappa>0$, has a positive $(d-1)$-dim Lebesgue measure. For our purpose, we  derived a quantitative version in terms of the
 log-type capacity (see Lemma \ref{lem-slicing}).
 However, \eqref{equ-ob-intro-6} is at scale  of the Hausdorff content, while the above-mentioned slicing theorem is at scale  of the capacity. Thus,
 we  have to set up a quantitative  transference between the Hausdorff content and the capacity (i.e., Lemma \ref{lem-connect}).

  With the above-mentioned  slicing lemma and the
  quantitative transference lemma,
 we can use  \eqref{equ-ob-intro-6} to get a desired quantitative propagation of smallness for analytic functions from a set of positive  log-type Hausdorff content
  in the $d$-dim case (i.e., Proposition \ref{prop-ana-high-2}).

 It deserves mentioning that   Hausdorff contents and  capacities  are very different tools to measure sets. This distinction creates a gap (see Remark \ref{rem-gap}),  leading to slightly weaker results in Theorem \ref{thm-bound-high} and Theorem \ref{thm-Rn-high} compared to  Theorem \ref{thm-ob-bound} and Theorem \ref{thm-ob-R}.

 $(iii)$ Due to the aforementioned modifications, the
   spectral inequality/uncertainty principle in the  $d$-dim case is slightly different from its  $1$-dim counterpart
    (see  Proposition \ref{prop-ana-high-2} and Proposition \ref{prop-LS-high}).

$(iv)$ Two facts are worth mentioned. First, when $E$  has a positive Lebesgue measure, the extension
of the quantitative propagation of smallness from $1$-dim to $d$-dim follows easily from standard techniques involving the use of the polar
coordinates (see \cite{ApraizNullcontrol2013,Kovrijkine2001some,Wang2019observable}).  However, as mentioned in \cite{walton2024observability},
the use of the polar coordinates encounters difficulties, even for the case that $E$  has a positive $s$-Hausdorff measure (with $s<d$).
Second, based on the above-mentioned slicing theorem in fractal geometry, Malinnikova (see \cite{Malinnikova2004propaga}) established  a H\"older type unique continuation inequality for analytic function  from fractal Hausdorff dimension sets for the  $d$-dim (with $d\geq 2$) case. Inspired by  \cite{Malinnikova2004propaga}, we further refined this result by presenting  a propagation of smallness from log-type Hausdorff content sets.
\end{remark}



\begin{remark}\label{rem-new-spec}
As a byproduct, we build up the following spectral inequality: When $E\subset\R^d$
is a subset of positive $F_{\alpha,\beta}$-Hausdorff content (with $\alpha>1$, $\beta=3/2$),
\begin{align}\label{equ-ob-intro-5}
  \sup_{\Omega}|\phi|\leq Ce^{\frac{C\lambda}{(\log \lambda)^\tau}}\sup_E|\phi|,\;\;\;\mbox{for each}\;\; \phi\in \mathcal{E}_\lambda,
\end{align}
with $\tau=\frac{2(\alpha-1-)}{3}$. (Here and below, $a-$ denotes $a-\varepsilon$ for arbitrarily small $\varepsilon>0$.)
The spectral inequality \eqref{equ-ob-intro-5} is crucial to obtain our observability inequality and  has independent interest. Despite the similarity between \eqref{equ-ob-intro-5}
and  \eqref{equ-ob-intro-4}, they are essentially different from two perspectives: First, the tools  used to measure $E$  are different. The former uses  $F_{\alpha,\beta}$-Hausdorff content which is finer than  $g_\delta$-Hausdorff content used in the  latter.  Second, the factor  $e^{\frac{C\lambda}{(\log \lambda)^\tau}}$ in \eqref{equ-ob-intro-5} grows much more quickly  than the factor $e^{C\sqrt{\lambda}}$ in \eqref{equ-ob-intro-4}  as $\lambda\to \infty$.
This aligns with the intuition that as the size of the observation set decrease, the cost constant  increases.
Several key points regarding the aforementioned two factors deserve to be highlighted. Fact one, for open subsets $E\subset\R^d$, the growth order $e^{C\sqrt{\lambda}}$ is optimal and cannot be replaced by $e^{g(\lambda)}$, with  $g(\lambda)/\sqrt{\lambda}\to 0$ as $\lambda\to \infty$ (see \cite{RousseauLebeau2012}, and also see \cite{RousseauMoyano2016} for a similar result  on $\R^d$).
Fact two,  for subsets $E\subset\R^d$ with positive $d$-dim Lebesgue measure or  $(n-1+\delta)$-Hausdorff content,
the growth order remains to be   $e^{C\sqrt{\lambda}}$ (see  \cite{Apraiz2014obser} and \cite{walton2024observability} respectively).
Fact three,
to our best knowledge, except for \eqref{equ-ob-intro-5},
no other spectral inequality exists with a growth order differing from $e^{C\sqrt{\lambda}}$
for the Laplacian on  bounded domains\footnote{For the  Laplacian on the whole space $\R^d$,  the counterpart of  the spectral inequality (which is equivalent to the Logvinenko-Sereda  uncertainty principle) has the same growth rate $e^{C\sqrt{\lambda}}$. It is worth mentioning that for the Schr\"odinger operator $H=-\Delta+V(x)$ with power growth potentials $V$ on $\R^d$, uncertainty principles  with different growth rates  have been established. This is a very active research field, see, for example, \cite{JamingSpectralestimates2021,Dickeuncertainty2023JFAA,Dickespectral2024,wang2024quantitative,zhu2023spectral}. }.
\end{remark}

\section{Remez's inequality and propagation of smallness}\label{section 2}

The main purpose of this section is to  present a quantitative propagation of smallness of analytic functions from sets of positive
log-type Hausdorff contents in the $1$-dim case.

\subsection{Hausdorff contents and examples}\

We start with the following definitions:
\begin{definition}\label{def-1}
$(i)$ If  $f: (0,\infty)\rightarrow [0,\infty)$ is a
strictly increasing and continuous function satisfying
$\lim_{t\to 0^+}\limits f(t)=0$, then it is called a gauge function.

For a gauge function $f$, we use $f^{[-1]}$ to denote its inverse function.\footnote{In this paper, $f^{-1}$ will be used to denote the quantity $1/f$.}

$(ii)$ Let $f$ be a gauge function.  The $f$-Hausdorff content   $c_f$ is a map from
$2^{\mathbb{R}^d}$  to $[0,\infty]$, defined in the manner:
$$
c_f(E):=\inf\Big\{ \sum_{j=1}^\infty f(d(B_j)): E\subset \bigcup_{j=1}^\infty B_j \Big\},\;\;\;E\subset\mathbb{R}^n
$$
where $B_j$, $j=1,2, \dots$,  are open balls (in $\R^d$) with the diameter  $d(B_j)$.
While the $f$-Hausdorff measure $m_f$ is a map from
$2^{\mathbb{R}^d}$ to $[0,\infty]$, defined in the manner:
$$
m_f(E):=\lim_{\delta\to 0^+}\inf\Big\{ \sum_{j=1}^\infty  f(d(B_j)): E\subset \bigcup_{j=1}^\infty B_j, d(B_j)\leq \delta \Big\},\;\;\;
E\subset\mathbb{R}^d.
$$
\end{definition}

\begin{remark}\label{remark2.2-w-9-29}
$(i)$ Only the values of $f$ for small independent variable $t$
 are relevant to $c_f$ and $m_f$.
  Thus, one can define a gauge function $f$ on $(0,K)$ (or $(0,K]$) for some $K>0$, and it should be  interpreted as a gauge function  over $(0,\infty)$ such that $f(t)\sim 1$ as $t\to +\infty$. For instance, the functions
  $F_{\alpha,\beta}$ and
  $h_{\alpha}$ (given by \eqref{equ-intro-F} and
  \eqref{1.7w9-21} respectively) are precisely this type of functions.

 $(ii)$ By Definition \ref{def-1}, for each $E\subset \R^d$,  we have $m_f(E)\geq c_f(E)$
 and  $m_f(E)>0$ if and only if $c_f(E)>0$.

 $(iii) $ For each $\delta>0$, we define the gauge function:
 \begin{align}\label{1.8w-9-21}
 g_\delta(t):=t^\delta,\;t>0.
 \end{align}
 The $g_{\delta}$-Hausdorff measure $m_{g_{\delta}}$ is also called $\delta$-Hausdorff measure and denoted by $\mathcal{H}^\delta$ sometimes. Similarly,  the   $g_{\delta}$-Hausdorff content $c_{g_{\delta}}$ is also called $\delta$-Hausdorff content.
 For further details on the theory of $f$-Hausdorff measure/content, we refer  to \cite{Rogers1998hausdorff, LinYang}.
    \end{remark}


\begin{example}\label{rmk-cantor}
We provide two concrete subsets $E$ in $[0,1]$ such that $c_{h_\alpha}(E)>0$ for each  $\alpha>0$,
  while  $c_{g_\delta}(E)=0$ for  all $\delta\in(0,1]$,
where $h_{\alpha}$ and $g_\delta$ are given by \eqref{1.7w9-21} and \eqref{1.8w-9-21} respectively. These examples show that
Theorem \ref{equ-heat-b} improves the result in  \cite{walton2024observability}.
\vskip 5pt

(i)  Let $\mathcal{L}$ be the set of  all Liouville  numbers in $[0, 1]$. In other words,
\begin{equation}\label{eq-liouv}
  \mathcal{L} := \left\{ x \in [0, 1] \setminus \mathbb{Q}\; :\;  \forall\; n \in \mathbb{N}, \exists\; \frac{p}{q}\in\mathbb{Q} \text{ with } q > 1 \text{ such that } \left| x - \frac{p}{q} \right| < \frac{1}{q^n} \right\}.
\end{equation}
It is well known that the  Hausdorff dimension of $\mathcal{L}$ is zero, thus $c_{g_\delta}(\mathcal{L})=0$  for all $\delta\in(0,1]$. On the other hand, it follows from \cite[Theorem 2]{Olsen2005} that for each $\alpha>0$,
 $m_{h_\alpha}(\mathcal{L})=\infty$. Hence, we have $c_{h_\alpha}(\mathcal{L})>0$. (See  Remark \ref{remark2.2-w-9-29} $(ii)$.)
 This, along with the fact that   $\mathcal{L}$ is compact, yields that
  $0<c_{h_\alpha}(\mathcal{L})<\infty$.

\vskip 5pt
(ii) Let $h_{\alpha}$ be the gauge function given by \eqref{1.7w9-21}. To simplify computations, we first consider the case  $\alpha=0$, namely
\begin{align}\label{3.2-w-10-19}
h_0(t)=(\log \frac{1}{t})^{-\frac{1}{2}}, \;\; 0<t\leq e^{-3}.
\end{align}

Let  $\{c_k\}_{k\ge 1}$ be a sequence of real numbers satisfying
\begin{equation}\label{cant-01}
	h_{0}(c_k)=2^{-k},\,\,\,\,k=1,2,\cdots.
\end{equation}
Then by \eqref{3.2-w-10-19}, we have
\begin{equation}\label{cant-02}
	c_k=e^{-2^{2k}};\;\;\; 2c_{k+1}<c_k\;\;\mbox{for each}\;\;k\geq 1.
\end{equation}

{We now construct a generalized Cantor set in the following manner:

Let $I_{0,1} := [0,1]$. At the first step, we remove the interval $(c_1, 1-c_1)$ from $I_{0,1}$, resulting in two intervals: $I_{1,1} = [0, c_1]$ and $I_{1,2} = [1-c_1, 1]$. At the second step, we remove an open interval from the center of each of $I_{1,1}$ and $I_{1,2}$, yielding four closed intervals of equal length $c_2$, denoted as $I_{2,1}, I_{2,2}, I_{2,3}, I_{2,4}$.
This process continues iteratively. At the $k$-th step, we select two subintervals, each of length $c_k$, from the intervals obtained at the $(k-1)$-th step by removing the middle open interval. The intervals obtained at each step are illustrated in the following figure:}

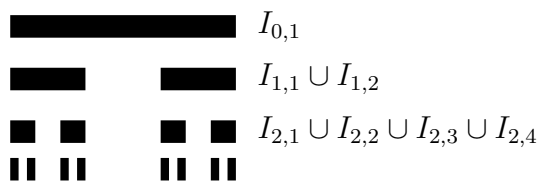
\begin{center}

	\begin{tikzpicture}[decoration=Cantor set,line width=3mm]
		\draw (0,0) -- (3,0) node[right]{$I_{0,1}$};
		\draw decorate{ (0,-.7) coordinate (a_1) -- (3,-.7) } node[right]{$I_{1,1}\cup I_{1,2}$};
		\draw decorate{ decorate{ (0,-1.4) -- (3,-1.4) }}node[right]{$I_{2,1}\cup I_{2,2}\cup I_{2,3}\cup I_{2,4}$};;
		\draw decorate{ decorate{ decorate{ (0,-1.9) -- (3,-1.9) }}};
		%
	\end{tikzpicture}
\end{center}
One can easily see that for each $k\geq 1$, we have that $|I_{k,j}|=c_k$ for all $1\le j\le 2^k$.
Let
\begin{equation}\label{exam-1d}
\mathcal{C}=\bigcap_{k=0}^{\infty}\bigcup_{j=1}^{2^k}I_{k, j}.
\end{equation}
{By \eqref{cant-02}}, we have
$
{dim_{\mathcal{H}}}(\mathcal{C})=\varliminf_{n\to\infty}\limits\frac{\log{2^n}}{-\log{c_n}}=\varliminf_{n\to\infty}\limits\frac{n}{2^{2n}}=0,
$
see e.g. \cite{QuRaosu},
which implies that $c_{g_\delta}(\mathcal{C})=0$ for each $\delta>0$.

On the other hand, since for every $k=1,2,\ldots,$ $\mathcal{C}\subset\bigcup_{j=1}^{2^k}I_{k,j}$, it follows {from \eqref{cant-01}} that
$c_{h_{0}}(\mathcal{C})\le \sum_{j=1}^{2^k}h_0(|I_{k,j}|)= 1$. Furthermore, we can use the way in \cite[p.63]{Mattila1995geometry}
to  verify that
 $\frac14\le m_{h_0}(\mathcal{C})\le1$.

Finally, we can use the same method to handle the situation where $\alpha>0$,
by adjusting the relationships in \eqref{cant-02} accordingly.
\end{example}

\begin{remark}\label{rmk-high-d}

Based on the above examples, we can construct subsets $E\subset \R^d$ such that the Hausdorff dimension of $E$ is $d-1$, while  $c_{F_{\alpha,
\beta}}(E)>0$. Indeed, given
 $\alpha\ge 0$ and $\beta>0$, consider the Cartesian product
$$
E=[0,1]^{d-1}\times \mathcal{C}_{\alpha, \beta},
$$
where $\mathcal{C}_{\alpha, \beta}$ is defined in the same way as \eqref{exam-1d}, except that  $|I_{k,j}|=\tilde{c}_k$, with $\tilde{c}_k$  satisfying
$$
f_{\alpha, \beta}(\tilde{c}_k)=2^{-k},\,\,\,\,k=1,2,\cdots,
$$
where $f_{\alpha, \beta}(t)=(\log \frac{1}{t})^{-\beta}(\log\log\frac{1}{t})^{-\alpha}$.
From which we deduce that $$\exp\{-2^{\frac{k}{\beta}}\}\le\tilde{c}_k\le\exp\{-2^{\frac{k}{2\beta}}\}$$
holds when $k$ is large enough. 
Then  repeat the arguments in (ii) of Example \ref{rmk-cantor}, we have ${dim_H}(\mathcal{C}_{\alpha, \beta})=0$ and $c_{f_{\alpha, \beta}}(\mathcal{C}_{\alpha, \beta})>0$. In summary, we conclude that ${dim_H}(E)=d-1$, while $c_{F_{\alpha, \beta}}(E)>0$.

\end{remark}

\begin{remark}\label{rem-d-1}
The observability inequalities for the equation  \eqref{equ-heat-high-bd} with observation sets of Hausdorff dimension $d-1$
have been studied in   Dolecki \cite{Dolecki1973} and Samb \cite{Samb2015}, based on
the results in the number theory. We summarize their main results  as follows:

In the $1$-dim case where $\Omega=(0,1)$, if
$$
x_0\in\mathcal{A}:=\{x\in (0,1)\;:\; x\;\;\mbox{is an algebraic numbers of order}\; k>1\},
$$
then  for each $T>0$, any solution $u$ to  the equation  \eqref{equ-heat-high-bd} satisfies
\begin{align}\label{equ-ob-d-1}
 \|u(0,\cdot)\|^2_{L^2(0,1)}\leq C_{obs}\int_0^T|u(t,x_0)|^2\d t.
\end{align}

In the case that $d\geq 2$ where $\Omega=(0,1)\times \Omega'$, with $\Omega'$  a bounded domain in $\R^{d-1}$, if  $\omega\subset \Omega'$ is an open subset and  $x_0\in \mathcal{A}$, Then, for each $T>0$, any solution $u$ to  the equation  \eqref{equ-heat-high-bd} satisfies
\begin{align}\label{equ-ob-d-2}
 \|u(0,\cdot)\|^2_{L^2(\Omega)}\leq C_{obs}\int_0^T\int_\omega|u(t,x_0,x')|^2\d x'\d t.
\end{align}
Note that \eqref{equ-ob-d-2} remains true even for  subsets $\omega$ with positive $(d-1+\delta)$-Hausdorff content
for $\delta>0$, as pointed out by Green et.al. \cite{walton2024observability}.

Next, we explain the difference between  the above results and our own as follows: For the $1$-dim case, \eqref{equ-ob-d-1} is exactly the observability inequality \eqref{thm-ob-1024h-1} with $L=1$ and $E=\{x_0\}$.
Thus, for each $x_0\in \mathcal{A}$, $\{x_0\}$ is an observable set.
However, these observable sets are not characterized by any Hausdorff content or any other measuring tool.
Compared with the above result, the advantage of Theorem \ref{thm-ob-bound}
lies in providing  a scale for measuring the observable sets which include a class of sets with Hausdorff dim $d-1$, rather than focusing on how to pick some particular  observable sets of Hausdorff dim $d-1$.

The same can be said about \eqref{equ-ob-d-2} and our Theorem \ref{thm-bound-high}.

\end{remark}

\subsection{Remez' inequality at scale of $c_{h_\alpha}$}\

This subsection aims to present an extension of   Remez's inequality at scale of  $c_{h_\alpha}$ (i.e., Lemma \ref{lem-remez}). A key ingredient
is   an estimate on the ${h_\alpha}$-content of a set where a monic polynomial takes small values (i.e., Lemma \ref{lem-ch}).
We point out that  the essence of the proof relies on the following
 classical Cartan lemma \cite{Cartan28} (see also in \cite[pp.202-204]{Lubinsky1997small}).

\begin{lemma}\label{lem-cartan}
Let $0<r_1<r_2<\cdots<r_n$. Let $P(z)$ (with $z\in \mathbb{C}$) be a monic polynomial of degree $n$. Then there exist positive integers $p\leq n$, $\lambda_1, \lambda_2,\dots, \lambda_p$ and closed balls
 $B_1,B_2,\dots, B_p$ (in $\mathbb{C}$), such that
\begin{itemize}
  \item [(i)] $\lambda_1+\lambda_2+\cdots+\lambda_p=n$.
  \item [(ii)]$d(B_j)=4r_{\lambda_j}$ for all $1\leq j\leq p$.
  \item [(iii)]$$
  \left\{  z\in \mathbb{C} \;:\; |P(z)|\leq \prod_{j=1}^n r_j \right\}\subset \bigcup_{j=1}^p B_j.
  $$
\end{itemize}
\end{lemma}


 Let $P$ be a monic polynomial of degree $n$. For each $\varepsilon>0$, we define the lemniscate:
\begin{align}\label{2.1-w-9-29}
E(P;\varepsilon)=\{z\in \mathbb{C}\;:\;|P(z)|\leq \varepsilon^n\}.
\end{align}
We simply write
\begin{align}\label{2.2-w-9-29}
c_f(P;\varepsilon):=c_f(E(P;\varepsilon)).
\end{align}

{
As an application of
the Cartan Lemma, Lubinsky \cite{Lubinsky1997small} derived the following upper bound on the Hausdorff content of \eqref{2.1-w-9-29}. Since only a sketch of the proof is given in \cite{Lubinsky1997small},
for the reader's convenicence, we provide its proof in detail here.
}

\begin{lemma}\label{lem-abs}
Let $f$ be a gauge function and let   $P$ be a monic polynomial of degree $n$. Then for each  $H\in (0,1]$,
\begin{align}\label{2.3-w-9-29}
c_f(P; ({g_n(H)}/{4})^n)
\leq f(H),
\end{align}
where
\begin{align}\label{2.4-w-9-29}
({g_n(H)}/{4})^n:=
4^{-n}\exp\left\{n\Big( \log H -\frac{1}{f(H)}\int^H_{f^{[-1]}(f(H)/n)} \frac{f(s)}{s}\d s\Big) \right\}.
\end{align}
\end{lemma}
\begin{proof}
Let
\begin{align}\label{2.5-w-9-29}
r_j=\frac{1}{4}f^{[-1]}\Big(\frac{j}{n}f(H)\Big), \quad 1\leq j\leq n.
\end{align}
Let  $\{\lambda_j\}_{j=1}^p$ and $\{B_j\}_{j=1}^p$
be given by
 Lemma \ref{lem-cartan}, corresponding to the above $\{r_j\}_{j=1}^n$.
 Then, by $(i)$ and $(ii)$ of  Lemma \ref{lem-cartan} and \eqref{2.5-w-9-29}, we have
 \begin{align}\label{2.5+-w-9-29}
\sum_{1\leq j\leq p} f(d(B_j))=\sum_{1\leq j\leq p}f(4r_{\lambda_j})
=f(H)\sum_{1\leq j\leq p}\frac{\lambda_j}{n}=f(H).
\end{align}
Meanwhile, by \eqref{2.5-w-9-29} and the monotonicity of both $f$ and $f^{[-1]}$, we see
\begin{align}\label{2.6-w-9-29}
\prod_{j=2}^n r_j=\prod_{j=2}^n\frac{1}{4}f^{[-1]}(\frac{j}{n}f(H))\geq 4^{-(n-1)}\exp(I),
\end{align}
where
\begin{align}\label{2.7-w-9-29}
I:=\int_1^n \log f^{[-1]}(\frac{s}{n}f(H))\d s.
\end{align}
Changing variable $\frac{s}{n}f(H)=f(t)$ and then using integration by parts in \eqref{2.7-w-9-29}, we obtain
\begin{align*}
    I&=\frac{n}{f(H)}\int^H_{f^{[-1]}(f(H)/n)}(\log t)f'(t)\d t\\
    &=\frac{n}{f(H)} \Big((\log H)f(H)-\big(\log f^{[-1]}(f(H)/n)\big)\frac{f(H)}{n}\Big)\\
    &\quad -\frac{n}{f(H)}\int^H_{f^{[-1]}(f(H)/n)}\frac{f(t)}{t}\d t\\
    &=n\log H-\log f^{[-1]}(f(H)/n)-\frac{n}{f(H)}\int^H_{f^{[-1]}(f(H)/n)}\frac{f(t)}{t}\d t.
\end{align*}
Since $r_1=\frac{1}{4}f^{[-1]}(\frac{1}{n}f(H))$, the above, along with \eqref{2.6-w-9-29}, \eqref{2.4-w-9-29}, leads to
$$
\prod_{j=1}^nr_j\geq 4^{-n}\exp\Big(n\Big( \log H -\frac{1}{f(H)}\int^H_{f^{[-1]}(f(H)/n)} \frac{f(s)}{s}\d s\Big) \Big)=({g_n(H)}/{4})^n.
$$
From the above, $(iii)$ of  Lemma \ref{lem-cartan}, \eqref{2.5+-w-9-29}
and the definition of $c_f$, we obtain  \eqref{2.3-w-9-29}. This finishes the proof.
\end{proof}

 In the rest of this section, we will focus on  $c_{h_\alpha}$, with $h_\alpha$ given by \eqref{1.7w9-21}.

\begin{lemma}\label{lem-ch}
Let $\alpha>0$.  {Then there exists a constant $C>0$, depending only on $\alpha$, such that for each $\delta\in(0,1)$ and each monic polynomial $P$ of degree $n$},
\begin{align}\label{equ-47-1}
  c_{h_\alpha}(P;\delta)\leq C (\log \frac{1}{4\delta})^{-\frac{1}{2}}  n^{\frac12} (\log (n+e))^{-\alpha/3},
\end{align}
where $ c_{h_\alpha}(P;\delta)$ is given by \eqref{2.2-w-9-29}.
\end{lemma}
\begin{proof}
We arbitrarily fix $\alpha>0$ and simply write $h$ for $h_\alpha$. Since the proof relies on Lemma \ref{lem-abs}, we also arbitrarily fix
$H\in(0,e^{-3}]$.
The proof is organized by three steps.


\vskip 5pt
\noindent {\it Step 1. We calculate $h^{[-1]}(h(H)/n)$.}
\vskip 5pt
Let
 \begin{align}\label{2.10-w-g-2-29}
 t:=h^{[-1]}(h(H)/n).
 \end{align}
  Then $h(t)=h(H)/n$. Thus,  by the definition of $h$ (see \eqref{1.7w9-21}), we have
$$
(\log \frac{1}{t})^{-\frac{1}{2}}(\log\log\frac{1}{t})^{-\alpha}
=\frac{1}{n}(\log \frac{1}{H})^{-\frac{1}{2}}(\log\log\frac{1}{H})^{-\alpha},
$$
from which, it follows that
\begin{align}\label{equ-ch-1}
 \log\frac{1}{t}=n^2\left(
 \frac{\log\log \frac{1}{H}}{\log\log\frac{1}{t}}\right)^{2\alpha}  \log \frac{1}{H}.
\end{align}
Based on \eqref{equ-ch-1}, we use  iteration to obtain that
\begin{align}\label{equ-ch-2}
 \log\frac{1}{t}&=n^2\left(
 \frac{\log\log \frac{1}{H}}{\log\Big(n^2\left(
 \frac{\log\log \frac{1}{H}}{\log\log\frac{1}{t}}\right)^{2\alpha}  \log \frac{1}{H}\Big)}\right)^{2\alpha}  \log \frac{1}{H}\nonumber\\
 &=n^2\left(
 \frac{\log\log \frac{1}{H}}{\log n^2 +\log\log\frac{1}{H}+{\Xi}}\right)^{2\alpha}  \log \frac{1}{H},
\end{align}
{
where $\Xi=\log\left(
 \frac{\log\log \frac{1}{H}}{\log\log\frac{1}{t}}\right)^{2\alpha}$. By the monotonicity of $h$ and the fact $h(t)=h(H)/n$, we have $t\leq H\leq e^{-3}$. Thus $3\leq\log{\frac{1}{H}}\leq \log{\frac{1}{t}}$. Note that $\varphi(t)=t^{\frac1t}$ ($t>0$) is decreasing when $t>e$, thus we have
 $(\log{\frac{1}{H}})^{(\log{\frac{1}{H}})^{-1}}\geq (\log{\frac{1}{t}})^{(\log{\frac{1}{t}})^{-1}}$, i.e.,
 $$
 \frac{\log\log \frac{1}{H}}{\log\log\frac{1}{t}}\geq \frac{\log \frac{1}{H}}{\log\frac{1}{t}}.
 $$
Inserting this into \eqref{equ-ch-1}, we find that
\begin{equation}\label{eq-11-2-01}
   -\frac{2\alpha}{2\alpha+1}\log{n^2}\le \Xi\leq 0.
\end{equation}
}
 From \eqref{equ-ch-2}, we see that
\begin{align}\label{equ-43-1}
h^{[-1]}(h(H)/n)=t=e^{-n^2\left(
 \frac{\log\log \frac{1}{H}}{\log n^2 +\log\log\frac{1}{H}+\Xi}\right)^{2\alpha}  \log \frac{1}{H}}=H^{n^2\left(
 \frac{\log\log \frac{1}{H}}{\log n^2 +\log\log\frac{1}{H}+\Xi}\right)^{2\alpha}}.
\end{align}

\vskip 5pt
\noindent {\it Step 2.  We  prove the following estimate:}
\begin{align}\label{equ-43-2}
    \log \frac{1}{H}\geq \frac{\log \frac{1}{4\delta}}{2nA^\alpha-1},
\end{align}
where
\begin{align}\label{equ-43-0}
   A:=\frac{\log\log \frac{1}{H}}{\log n^2 +\log\log\frac{1}{H}+\Xi},
   \end{align}
and
\begin{align}\label{2.16-w-g-s-9-29}
\delta^n:=\Big(\frac{g_n(H)}{4} \Big)^n,
   \end{align}
where $g_n(H)$ is given by \eqref{2.4-w-9-29} with $f$ replaced by $h$. By  the definition of $h$ (see \eqref{1.7w9-21}), we have
$$
h(s)\leq (\log \frac{1}{s})^{-\frac{1}{2}}(\log\log \frac{1}{H})^{-\alpha}\;\;\mbox{for each}\;\;s\in [h^{[-1]}(h(H)/n),H].
$$
With \eqref{1.7w9-21}, the above yields
\begin{align}\label{2.16-w-g-9-29}
    \frac{1}{h(H)}\int^H_{h^{[-1]}(h(H)/n)} \frac{h(s)}{s}\d s
   & \leq \frac{1}{h(H)}\int^H_{h^{[-1]}(h(H)/n)} s^{-1}(\log \frac{1}{s})^{-\frac{1}{2}}(\log\log \frac{1}{H})^{-\alpha}\d s\\
   &=(\log \frac{1}{H})^{\frac{1}{2}}\int^H_{h^{[-1]}(h(H)/n)} s^{-1}(\log \frac{1}{s})^{-\frac{1}{2}} \d s\nonumber\\
   &=2(\log \frac{1}{H})^{\frac{1}{2}}\Big( (\log \frac{1}{h^{[-1]}(h(H)/n)})^{\frac{1}{2}}  -(\log \frac{1}{H})^{\frac{1}{2}}\Big)\nonumber\\
   &=2(\log \frac{1}{H})( nA^\alpha-1).\nonumber
\end{align}
In the last step of \eqref{2.16-w-g-9-29}, we used \eqref{equ-43-1} and \eqref{equ-43-0}.
Because of \eqref{2.16-w-g-s-9-29},  it follows from \eqref{2.4-w-9-29} and \eqref{2.16-w-g-9-29} that
$$
\delta^n\geq 4^{-n}\exp\left\{ n\big(\log H - 2(\log \frac{1}{H})( nA^\alpha-1)\big) \right\},
$$
which leads to \eqref{equ-43-2}.
\vskip 5pt
\noindent {\it Step 3. We prove \eqref{equ-47-1}.
}
\vskip 5pt

To this end, we first claim that
\begin{align}\label{equ-43-3}
  h(H)\lesssim  (\log \frac{1}{4\delta})^{-\frac{1}{2}} n^{\frac12} (\log (n+e))^{-\alpha/3}.
\end{align}
  Let $\beta\in(0,1)$ be a parameter to be determined later. We prove \eqref{equ-43-3} by considering two cases.

  {\it Case 1:} When
  \begin{align}\label{2.18-w-g-2-29}
 \log \log \frac{1}{H}>(\log (n+e))^\beta.
\end{align}
    By the definition of $h$ (see \eqref{1.7w9-21}) and \eqref{2.18-w-g-2-29},
  we have
\begin{align}\label{equ-44-0}
    h(H)\leq (\log \frac{1}{H})^{-\frac{1}{2}}(\log (n+e))^{-\alpha \beta}.
\end{align}
Meanwhile, it follows from \eqref{equ-43-0} and \eqref{eq-11-2-01} that  $A\le 1$, which, along with  \eqref{equ-43-2}, implies
$$
    \log \frac{1}{H}\gtrsim \frac{\log \frac{1}{4\delta}}{2n}.
$$
Putting the above into \eqref{equ-44-0} yields
\begin{align}\label{equ-44-1}
 h(H)\lesssim ( \log \frac{1}{4\delta})^{-\frac{1}{2}}n^{\frac{1}{2}}(\log (n+e))^{-\alpha \beta}.
\end{align}

{\it  Case 2:} When
\begin{align}\label{2.21-w-g-9-29}
\log \log \frac{1}{H}\leq (\log (n+e))^\beta.
\end{align}
By \eqref{eq-11-2-01},
 \eqref{equ-43-0} and \eqref{2.21-w-g-9-29}, we see
  that $A\lesssim (\log (n+e))^{-(1-\beta)}$. This, along with \eqref{equ-43-2}, leads to
  $$
\log \frac{1}{H}\gtrsim \frac{\log\frac{1}{4\delta}}{n(\log (n+e))^{-\alpha(1-\beta)}}.
$$
Since $\log\log\frac{1}{H}\gtrsim 1$, the above, along with  the definition of $h$ (see \eqref{1.7w9-21}), shows
\begin{align}\label{equ-44-2}
    h(H)\lesssim (\log \frac{1}{H})^{-\frac{1}{2}}\lesssim (\log \frac{1}{4\delta})^{-\frac{1}{2}}n^{\frac{1}{2}}(\log (n+e))^{-\frac{\alpha(1-\beta)}{2}}.
\end{align}

Now, we choose $\beta=\frac{1}{3}$. Then we have $\alpha \beta=\frac{\alpha(1-\beta)}{2}=\frac{\alpha}{3}$.
Thus,  \eqref{equ-43-3} follows from  \eqref{equ-44-1} and \eqref{equ-44-2} at once.

Finally, by  \eqref{equ-43-3}, we can apply Lemma \ref{lem-abs} to deduce
that
\eqref{equ-47-1} holds for each $0<\delta< 1$.
\end{proof}

Now, we will present an extension of Remez's inequality at the scale of  $c_{h_\alpha}$
\begin{remark}
The classical Remez inequality states the following:
{\it Given an interval $I$, there exists a numerical constant $C > 0$ such that for any polynomial $P$ of degree $n$ and any subset $E \subset I$ with positive Lebesgue measure, the inequality
\begin{align}\label{2.28-w-9-30-revised}
\sup_{x \in I}|P(x)| \leq \left(\frac{C|I|}{|E|}\right)^n \sup_{x \in E}|P(x)|
\end{align}
holds, where $|I|$ and $|E|$ denote the $1$-dim Lebesgue measures of $I$ and $E$ respectively.}
This inequality has been  generalized in \cite{Friedland2017sp} to sets $E$ with positive $\delta$-Hausdorff content ($\delta > 0$). Our extension further refines this by considering sets $E$ that merely possess a positive ${h_\alpha}$-content,  allowing $E$ to be smaller. However, this comes at the cost of the constant in our inequality, analogous to
$C$ in \eqref{2.28-w-9-30-revised}, now depending on $n$ and diverging as $n\to \infty$.
\end{remark}

We will use $D_k$ (with $k>0$)  to denote  the disc $\{z\in \mathbb{C}\; :\; |z|\leq k\}$.

\begin{lemma}[Remez's inequality at scale of $c_{h_\alpha}$]\label{lem-remez}
Let $\alpha>0$. There is a constant  $C>0$, depending only on $\alpha$, such that
for any polynomial $P$  of degree $n$ and any subset $Z\subset D_1$ with $c_{h_\alpha}(Z)>0$,
\begin{align}\label{2.29-w-9-30}
\sup_{z\in D_1}|P(z)|\leq  12^n\exp\Big\{\frac{C n^2(\log (n+e))^{-\frac{2\alpha}{3}}}{c^2_{h_\alpha}(Z)}\Big\}\sup_{z\in Z}|P(z)|.
\end{align}
\end{lemma}

\begin{proof}
We arbitrarily fix $\alpha>0$ and simply write $h$ for $h_\alpha$. Throughout the proof, let $C$ represent various positive constants depending only on $\alpha$, which may change from line to line.

 It is clear that $\sup_{z\in Z}|P(z)|>0$. Thus, without loss of generality, we can assume that
  $|P(z)|\leq 1$ when $z\in Z$. To show \eqref{2.29-w-9-30}, it suffices to prove
\begin{align}\label{equ-228-0}
\max_{z\in D_1}|P(z)| \leq 12^n\exp\Big\{\frac{C n^2(\log (n+e))^{-\frac{2\alpha}{3}}}{c^2_h(Z)}\Big\}.
\end{align}
 We write  $P(z):=A\prod_{j=1}^n(z-z_j)$, with $z_j\in \mathbb{C}$.
  Before proving \eqref{equ-228-0}, we claim
\begin{align}\label{equ-228-1}
 |A|\leq 4^n\exp\Big\{\frac{C n^2(\log (n+e))^{-\frac{2\alpha}{3}}}{c^2_h(Z)}\Big\}.
\end{align}
For this purpose, two facts are given as follows: First, since $|P(z)|\leq 1$ when $z\in Z$, we have
 $Z\subset E(P;1)$ (see \eqref{2.1-w-9-29}). The monotonicity of
 $c_h$ implies
 $c_h(Z)\leq c_h(E(P;1))$. Second, it is clear that  $E(P;1)=E(A^{-1}P;A^{-\frac{1}{n}})$.

 If  $|A|\lesssim 1$,  then \eqref{equ-228-1} follows from the fact that $c_h(Z)\leq c_h(D_1)\lesssim 1$.
 While if
  $|A|\gg 1$, then by  Lemma \ref{lem-ch} where $P$ is replaced by  $A^{-1}P$ (a monic polynomial)
  and the above two facts, we can find
  $C>0$, depending only on $\alpha$, such that  for all $n\geq 1$,
$$
c_h(Z)\leq c_h(E(A^{-1}P;A^{-\frac{1}{n}}))\leq C\Big(\log \frac{A^{\frac{1}{n}}}{4}\Big)^{-\frac{1}{2}}  n^{\frac12}(\log (n+e))^{-\alpha/3}.
$$
Solving the above inequality shows the claim \eqref{equ-228-1}.

Now, we  consider two cases to prove \eqref{equ-228-0}:

{\it Case 1:} If $|z_j|\leq 2$ for all
   $1\leq j\leq n$, then for any $z\in D_1$ and  $j\in\{1,\dots,n\}$, we have $|z-z_j|\leq 3$. This, along with  \eqref{equ-228-1}, implies
$$
\max_{z\in D_1}|P(z)|\leq 3^n|A|\leq 12^n\exp\Big\{\frac{C n^2(\log (n+e))^{-\frac{2\alpha}{3}}}{c^2_h(Z)}\Big\},
$$
thus \eqref{equ-228-0} holds for the first  case.

{\it Case 2:}  If there exists  $1\leq n_1<n$ such that  $|z_j|\leq 2$ for all $1\leq j\leq n_1$ and $|z_j|>2$ for $n_1<j\leq n$,
   we write
   $$
P_1(z):=A\prod_{j=1}^{n_1}(z-z_j), \quad P_2(z):=\prod_{j=n_1+1}^{n}(z-z_j).
$$
   Then we have  $P=P_1P_2$.
Since  any two points $v_1,v_2\in D_1$ satisfy $|P_2(v_1)/P_2(v_2)|\leq 3^{n-n_1}$, we find
\begin{align}\label{2.32-w-9-30}
\frac{\sup_{D_1}|P(z)|}{\sup_Z|P(z)|}\leq 3^{n-n_1}\frac{\sup_{D_1}|P_1(z)|}{\sup_Z|P_1(z)|}.
\end{align}
Given that  all zeros of $P_1$ are contained in $D_2$, by \eqref{2.32-w-9-30} and the result obtained in the first case, we obtain
\begin{align*}
 \frac{\sup_{D_1}|P(z)|}{\sup_Z|P(z)|}&\leq 3^{n-n_1} 12^{n_1}\exp\Big\{\frac{C {n_1}^2(\log (n_1+e))^{-\frac{2\alpha}{3}}}{c^2_h(Z)}\Big\}\\
 &\leq 12^n\exp\Big\{\frac{CC' {n}^2(\log (n+e))^{-\frac{2\alpha}{3}}}{c^2_h(Z)}\Big\},
\end{align*}
where in the last step we used the fact that for some $C'>0$ (depending only on $\alpha$),
 $$
 {n_1}^2(\log (n_1+e))^{-\frac{2\alpha}{3}}\leq C'{n}^2(\log (n+e))^{-\frac{2\alpha}{3}}\;\;\mbox{for all}\;\;n\geq n_1.
 $$
 Thus \eqref{equ-228-0}  holds for the second case.

 Hence, we complete the proof.
\end{proof}

\subsection{Propagation of smallness  from sets of positive ${h_\alpha}$-content in $1$-dim}\

 In this subsection, utilizing the Remez inequalities established previously, we will prove various versions of propagation of smallness inequalities,  from log-type Hausdorff content sets for analytic functions. These inequalities will be essential in deriving subsequent spectral inequalities.

\begin{lemma}\label{lem-ana}
Let $\alpha>0$ and let $I$ be an interval with $|I|=1$ and $0\in I$. Then there is a constant $C>0$, depending only on $\alpha$, such that
when  $\phi:D_5\mapsto \C$ is an analytic function with $|\phi(0)|\geq 1$ and $E\subset I$ is a subset with $c_{h_\alpha}(E)>0$,
$$
\sup_{x\in I}|\phi(x)|\leq M^9\exp\Big\{\frac{C (\log M)^2(\log ( \frac{\log M}{\log 2}+e) )^{-\frac{2\alpha}{3}}}{c^2_{h_\alpha}(E)}\Big\}\sup_{x\in E}|\phi(x)|,
$$
where $M:=\sup_{z\in D_4}|\phi(z)|$.
\end{lemma}

\begin{proof} Our proof  is inspired by \cite{Kovrijkine2001some} and make use of   Lemma \ref{lem-remez}.

We arbitrarily fix $\alpha>0$ and simply write $h$ for $h_\alpha$. Then we arbitrarily fix an analytic function $\phi$ on $D_5$ (with
$|\phi(0)|\geq 1$) and a subset $E\subset I$ (with  $c_h(E)>0$).
Let $a_1,a_2,\cdots,a_m$ be the zeros of $\phi$ in $D_2$.
Since
 $|\phi(0)|\neq 0$, it follows that $a_k\neq 0$ for all $k\in\{1,\dots,m\}$. We define the  scaled Blaschke product:
$$
B(z)= \prod_{1\leq k\leq m}  \frac{2(a_k-z)}{4-\bar{a}_k z}:=\frac{P(z)}{Q(z)},\;\;z\in D_2.
$$
It is straightforward to verify that $B$ is  analytic over $D_2$ and satisfies
\begin{align}\label{equ-remez-0}
|B(z)|=1 \;\; \mbox{ for all }\; |z|=2.
\end{align}
(Here, we used the fact that when $|z|=2$, $|\frac{2(a_k-z)}{4-\bar{a}_k z}|=1$  for all $1\leq k\leq m$.)

Then we define
 $g:D_2\mapsto \C$ by
\begin{align}\label{equ-remez-1}
  g(z)=\phi(z)/B(z), \quad z\in D_2.
\end{align}
Clearly,  $g$ is a zero-free analytic function on $D_2$; $|g(0)|\geq 1$;
 \eqref{equ-remez-1} can be rewritten as
\begin{align}\label{equ-remez-2.5}
    \phi(z)=g(z)\cdot \frac{P(z)}{Q(z)}.
\end{align}

Next, we claim
\begin{align}\label{equ-remez-5}
\frac{\sup_{x\in I}|\phi(x)|}{\sup_{x\in E}|\phi(x)|} \leq M^336^m\exp\Big\{\frac{C m^2(\log m)^{-\frac{2\alpha}{3}}}{c_h(E)}\Big\}.
\end{align}
To this end, several estimates associated with $g$, $P$ and $Q$, will be built in order.

First, by \eqref{equ-remez-0}, \eqref{equ-remez-2.5}
and the definition of $M$, we see
that  $|g(z)|\leq M$, when $|z|=2$. This, along with the maximum modulus principle, yields
\begin{align}\label{equ-remez-2}
|g(z)|\leq M\;\; \mbox{ for all } \;\;z\in D_2.
\end{align}

Second, since $g(z)\neq 0$ for all $z\in D_2$,  it follows from  \eqref{equ-remez-2} that $z\rightarrow \log M-\log |g(z)|$ is
 a non-negative harmonic function on $D_2$. Using  Harnack's inequality and  the fact that
 $|g(0)|\geq 1$, we have
$$
\max_{z\in D_{1}}(\log M-\log |g(z)|)\leq 3 \min_{z\in D_{1}}(\log M-\log |g(z)|) \leq 3\log M,
$$
which gives $\min_{D_{1}} |g(z)|\geq M^{-2}$. This,  together with \eqref{equ-remez-2}
and the fact that $I\subset {D_1}$,
 implies
$$
\frac{\sup_{x\in I}|g(x)|}{\inf_{x\in I}|g(x)|}\leq M^3.
$$
Meanwhile, a direct computation gives
$$
\frac{\sup_{x\in I}|Q(x)|}{\inf_{x\in I}|Q(x)|}\leq \frac{\prod_{1\leq k\leq m}(4+|a_k|)}{\prod_{1\leq k\leq m}(4-|a_k|)}\leq 3^m.
$$
The above two inequalities lead to
\begin{align}\label{equ-remez-3}
\frac{\sup_{x\in I}|g(x)/Q(x)|}{\inf_{x\in I}|g(x)/Q(x)|}\leq \frac{\sup_{x\in I}|g(x)|/\inf_{x\in I}|Q(x)|}{\inf_{x\in I}|g(x)|/\sup_{x\in I}|Q(x)|}\leq 3^mM^3.
\end{align}

Third,  applying Lemma \ref{lem-remez} to the above polynomial $P$, we get
\begin{align}\label{equ-remez-4}
\sup_{I}|P(z)|\leq  12^m\exp\{\frac{C m^2(\log (m+e))^{-\frac{2\alpha}{3}}}{c^2_h(E)}\}\sup_{E}|P(z)|.
\end{align}
Since
\begin{align*}
\frac{\sup_{x\in I}|\phi(x)|}{\sup_{x\in E}|\phi(x)|} \leq  \frac{\sup_{x\in I}|g(x)/Q(x)|}{\inf_{x\in I}|g(x)/Q(x)|} \cdot\frac{\sup_{x\in I}|P(x)|}{\sup_{x\in E}|P(x)|},
\end{align*}
we obtain \eqref{equ-remez-5}  from
\eqref{equ-remez-2.5},  \eqref{equ-remez-3} and \eqref{equ-remez-4} at once.

Finally, we will estimate $m$  in terms of $M$. We denote by  $\mathfrak{n}(r)$ the number of zeros of $\phi$ in the disc $D_r$. Since $\phi$ is analytic in $D_5$, by Jensen formula, we have for all $R<5$
\begin{align}\label{equ-jensen}
\int_0^R \mathfrak{n}(r) \frac{\d r}{r}=\frac{1}{2 \pi} \int_0^{2 \pi} \log \left|\phi\left(R e^{i \theta}\right)\right| \d \theta-\log |\phi(0)| .
\end{align}
Applying \eqref{equ-jensen} with $R=4$, using the definition of $M$ and $|\phi(0)|\geq 1$, we obtain
$$
\log M\geq \int_0^4 \mathfrak{n}(r) \frac{\d r}{r}\geq \int_2^4 \mathfrak{n}(r) \frac{\d r}{r} \geq  \int_2^4 \mathfrak{n}(2) \frac{\d r}{r}= \mathfrak{n}(2) \log 2,
$$
which gives
$m=\mathfrak{n}(2) \leq {\log M}/{\log 2}$.
Thus the right hand side of \eqref{equ-remez-5} is bounded by
$$
M^9\exp\Big\{\frac{C (\log M)^2(\log ( \frac{\log M}{\log 2}+e))^{-\frac{2\alpha}{3}}}{c^2_h(E)}\Big\}.
$$
With the above, \eqref{equ-remez-5} yields the desired inequality.
This completes the proof.
\end{proof}
\begin{remark}\label{rem-Kov}
$(i)$
It is proved in \cite{Kovrijkine2001some} that if $E\subset I$ is subset of  positive Lebesgue measure.
$I$ and $\phi$ are the same as that in Lemma \ref{lem-ana},
then
\begin{align}\label{2.43-w-10-3}
\sup_{x\in I}|\phi(x)|\leq ({C}/{|E|})^{\frac{\log M}{\log 2}}\sup_{x\in E}|\phi(x)|,
\end{align}
where $M=\sup_{z\in D_2}|\phi(z)|$.


$(ii)$ With some minor modifications of the statement, we can remove the condition that $|\phi(0)|\geq 1$ from
 Lemma \ref{lem-ana}.  This will be done in the subsequent  corollary.
\end{remark}

\begin{corollary}\label{cor-ana}
Let $\alpha>0$ and let $I\subset \R$ be  a closed interval with $|I|=1$ and $0\in I$. Then there is a constant $C>0$, depending only on
$\alpha$, such that when  $\phi:D_6\mapsto \C$ is analytic and $E\subset I$ is a subset with $c_{h_\alpha}(E)>0$,
$$
\sup_{x\in I}|\phi(x)|\leq M^9\exp\Big\{\frac{C (\log M)^2(\log ( \frac{\log M}{\log 2}+e) )^{-\frac{2\alpha}{3}}}{c^2_{h_\alpha}(E)}\Big\}\sup_{x\in E}|\phi(x)|,
$$
where $M=\sup_{z\in D_5}|\phi(z)|/\sup_{x\in I}|\phi(x)|$.
\end{corollary}
\begin{proof}

We arbitrarily fix $\alpha>0$
and denote $h_\alpha$ simply as $h$. Without loss of generality, we assume that $\sup_{x\in I}|\phi(x)|\neq 0$, since otherwise $\phi$ would be identically zero, making the desired inequality trivial.

Let $x_0\in I$ be such that  $ |\phi(x_0)|=\max_{x\in I}|\phi(x)|$. Consequently, $|\phi(x_0)|\neq 0$. We then define
  $$
  \psi(z)=\phi(z+x_0)/|\phi(x_0)|, \;\;z\in D_5.
  $$
 It is clear that $\psi$ is analytic  on $D_5$ and $|\psi(0)|=1$. Moreover, we have
 $$
 0\in I-x_0;\;\; |I-x_0|=1;\;\;c_h(E-x_0)=c_h(E)>0.
 $$
   Thus we can apply Lemma \ref{lem-ana} to $\psi$ and obtain
$$
1\leq \sup_{x\in I-x_0}|\psi(x)|\leq M^9\exp\Big\{\frac{C (\log M)^2(\log ( \frac{\log M}{\log 2}+e))^{-\frac{2\alpha}{3}}}{c^2_h(E)}\Big\}\sup_{x\in E-x_0}|\psi(x)|,
$$
where $M=\sup_{z\in D_4}|\psi(z)|\leq \sup_{z\in D_5}|\phi(z)|/\sup_{x\in I}|\phi(x)|$. Rewriting the above inequality in terms of $\phi$  yields the desired estimate, thereby concluding the proof.
\end{proof}

The next corollary is a consequence of Corollary \ref{cor-ana}. It will be used in our studies for the case of $d$-dim.

\begin{corollary}\label{cor-inter-820}
Under the  assumptions of Corollary \ref{cor-ana}, there exists a constant $C>0$, depending on $\alpha$ and $c_{h_\alpha}(E)$, such that for each $\varepsilon\in (0,e^{-3}]$,
$$
\sup_{x\in I}|\phi(x)|\leq \varepsilon \sup_{z\in D_5}|\phi(z)|+ \exp\left\{C\frac{(\log \frac{1}{\varepsilon})^2 }{(\log\log \frac{1}{\varepsilon})^{2\alpha/3}} \right\}\sup_{x\in E}|\phi(x)|.
$$
\end{corollary}
\begin{proof}
We arbitrarily fix $\alpha>0$
and simply write $h$ for $h_\alpha$.  If $$M:=\sup_{z\in D_5}|\phi(z)|/\sup_{x\in I}|\phi(x)|\leq e^3,
$$
then the desired estimate follows directly from Corollary \ref{cor-ana}.

Hence we assume $M>e^3$ in the following. Applying  Corollary \ref{cor-ana}, we have
\begin{align}\label{equ-coro-ana2-1}
  \sup_{x\in I}|\phi(x)|\leq \exp\big\{ {C (\log M)^2(\log  \log M )^{-\frac{2\alpha}{3}}} \big\}  \sup_{x\in E}|\phi(x)|,
\end{align}
where   $C>0$ is a constant depending only on $\alpha$ and $c_h(E)$, and its value may vary  in different places.
We write
$$
X=\sup_{x\in E}|\phi(x)|, \;\;Y=\sup_{x\in I}|\phi(x)|,\;\; Z=\sup_{z\in D_5}|\phi(z)|.
$$
It is clear that  $X\leq Y\leq Z$. Moreover,  \eqref{equ-coro-ana2-1} can be rewritten as
\begin{align}\label{equ-coro-ana2-2}
Y\leq X\exp\left\{\frac{C}{A}\log^2 (\frac{Z}{Y})\right\}\tag{\ref{equ-coro-ana2-1}'}
\end{align}
where $A=(\log\log Z/Y)^{2\alpha/3}$. Taking the logarithm on  both sides of  \eqref{equ-coro-ana2-2} yields
\begin{align*}
    \log Y-\log X\leq \frac{C}{A}(\log Z-\log Y)^2,
\end{align*}
whence it can be deduced that
$$
\log Y\leq \log Z+\frac{A}{2C}-\sqrt{\frac{A}{C}\log \frac{Z}{X}+(\frac{A}{2C})^2}.
$$
Together with the fact  $A\lesssim \log (Z/X)$, the above gives
$$
Y\leq Z\exp\left\{ \frac{A}{2C}-\sqrt{\frac{A}{C}\log\frac{Z}{X}+(\frac{A}{2C})^2} \right\}
\leq Z \exp\{-\varepsilon_0\sqrt{\frac{A}{C}\log\frac{Z}{X}}\},
$$
where $0<\varepsilon_0<1$ is some small but fixed number.
 This can be rewritten as
$$
\frac{Z}{Y}\geq \exp\left\{ \varepsilon_0\sqrt{\frac{1}{C}\log\frac{Z}{X}(\log\log\frac{Z}{Y})^{2\alpha/3}} \right\}.
$$
Using  iteration in the above inequality shows
\begin{align*}
\frac{Z}{Y}&\geq \exp\left\{\varepsilon_0 \sqrt{\frac{1}{C}\log\frac{Z}{X}(\log \sqrt{\frac{1}{C}\log\frac{Z}{X}(\log\log\frac{Z}{Y})^{2\alpha/3}})^{2\alpha/3}}
\right\}\\
&\geq \exp\left\{C \sqrt{ \log\frac{Z}{X}(\log\log \frac{Z}{X} )^{2\alpha/3}}
\right\},
\end{align*}
 which leads to
\begin{align}\label{equ-coro-ana2-3}
Y\leq Z \exp\left\{-C \sqrt{ \log\frac{Z}{X}(\log\log \frac{Z}{X} )^{2\alpha/3}}
\right\}.
\end{align}

We now arbitrarily fix $\varepsilon\in (0,e^{-3}]$. We split the discussion into two cases.

{\it Case 1:}  $\exp\left\{-C \sqrt{ \log\frac{Z}{X}(\log\log \frac{Z}{X} )^{2\alpha/3}}\right\}>\varepsilon$,
we have
\begin{align}\label{2.44-w-10-1}
Z/X<t:=g^{[-1]}(\varepsilon),
\end{align}
where $g$ is the following strictly decreasing function:
 $$
g(s)=\exp\left\{-C \sqrt{  \log s  (\log\log s )^{2\alpha/3}}\right\},\;\;s\in (e,\infty).
$$
Meanwhile, direct computations show that
$$
 \log t = \frac{(\frac{1}{C}\log \frac{1}{\varepsilon})^2 }{(\log\log t)^{2\alpha/3}}=\frac{(\frac{1}{C}\log \frac{1}{\varepsilon})^2 }{(\log \frac{(\frac{1}{C}\log \frac{1}{\varepsilon})^2 }{(\log\log t)^{2\alpha/3}})^{2\alpha/3}}\leq \frac{C(\log \frac{1}{\varepsilon})^2}{(\log\log \frac{1}{\varepsilon})^{\frac{2\alpha}{3}}}.
$$
With \eqref{2.44-w-10-1}, the above yields
\begin{align}\label{equ-coro-ana2-4}
 Z< X\exp\left\{C\frac{(\log \frac{1}{\varepsilon})^2 }{(\log\log \frac{1}{\varepsilon})^{2\alpha/3}} \right\} .
\end{align}
Now, by \eqref{equ-coro-ana2-4} and $Y\leq Z$, we have
\begin{align}\label{equ-coro-ana2-5}
Y\leq \varepsilon Z+ X\exp\left\{C\frac{(\log \frac{1}{\varepsilon})^2 }{(\log\log \frac{1}{\varepsilon})^{2\alpha/3}} \right\}.
\end{align}

{\it Case 2:}   $\exp\left\{-C \sqrt{ \log\frac{Z}{X}(\log\log \frac{Z}{X} )^{2\alpha/3}}\right\}\leq \varepsilon$, \eqref{equ-coro-ana2-5} follows immediately from \eqref{equ-coro-ana2-3}.

Therefore the proof is complete.
\end{proof}

\section{Spectral inequality and uncertainty principle  in $1$-dim}\label{section-thm-1d}
In this section, we  first build up a spectral inequality and a quantitative uncertainty principle
at the scale of $c_{h_\alpha}$ for $1$-dim case,
then  we prove
 Theorem \ref{thm-ob-bound} and Theorem \ref{thm-ob-R}.

\subsection{Spectral inequality in the case of $1$-dim}\label{subsect 4.1-w-10-25}\

Let $L>0$. Define the operator  $A:=-\partial_x^2$ with  domain $D(A)=H^2(0, L)\cap H^1_0(0, L)$.
The eigenvalues and eigenfunctions of  $A$ are given by
 $$
\lambda_k=\Big(\frac{k\pi}{L}\Big)^2;\;\;  \; \phi_k(x)=\sin \frac{k\pi x}{L}, \;x\in[0,L],\;\;\mbox{for all}\;\; k=1,2,\cdots.
$$
For each $\lambda>0$,  we define the  subspace:
$\mathcal {E}_\lambda:=\mbox{Span}\{\phi_k\; :\;  \lambda_k\leq \lambda \}$.
Note that
\begin{align}\label{3.1-w-10-2}
\lambda_k\leq \lambda\;\;\mbox{if and only if}\;\;k\leq \frac{\sqrt{\lambda}L}{\pi}.
\end{align}

The main result of this subsection is the following  spectral inequality.
\begin{lemma}\label{lem-spec}
Let $\alpha>0$.
There is  a  constant $C>0$, depending only on $\alpha$, such that for each subset  $E\subset [0,L]$  with $c_{h_\alpha}(E)>0$ and each $\lambda>0$,
\begin{align}\label{3.2-w-10-2}
\sup_{x\in [0,L]}|\phi(x)|\leq e^{C(L+1)^2\sqrt{\lambda}}\exp\left\{\frac{C(L+1)^2\lambda(\log  ( \lambda+e) )^{-\frac{2\alpha}{3}}}{c^2_{h_\alpha}(E)}\right\}\sup_{x\in E}|\phi(x)|, \;\forall\;\phi\in \mathcal {E}_\lambda.
\end{align}
\end{lemma}
The proof of Lemma \ref{lem-spec} is based on  Lemma \ref{lem-ana} and the following  lemma.
\begin{lemma}[Nazarov-Turan \cite{Nazarov1993local}]\label{lemma3.2-w-10-2}
 Let $p(x)=\sum_{k=1}^nc_ke^{\mu_kx}$ (with $ (c_k\in \C, \mu_k\in \C$) be an exponential polynomial of order $n$. Let $I$ be an interval. Then there is an absolute constant $C>0$ such that
 when  $E\subset I$ is a  subset  with $|E|>0$,
 $$
 \sup_{x\in I}|p(x)|\leq e^{|I|\max| \emph{Re }  \mu_k|} \left(\frac{C|I|}{|E|} \right)^{n-1}
 \sup_{x\in E}|p(x)|.
 $$
\end{lemma}
\vskip 5pt

\begin{proof}[Proof of Lemma \ref{lem-spec}]
We arbitrarily fix $\alpha>0$ and simply write $h$ for $h_\alpha$.
Then we arbitrarily fix $E\subset [0,L]$  with $c_{h_\alpha}(E)>0$.
If $\lambda>0$ is such that $\mathcal {E}_\lambda=\emptyset$, then \eqref{3.2-w-10-2} holds trivially.
So we arbitrarily fix $\lambda>0$ so that  $\mathcal {E}_\lambda\neq\emptyset$.

To  apply  Lemma \ref{lem-ana}, we need to  reduce  $E$ to a set which is  contained in a unit interval  while maintaining positive $h$-Hausdorff content. However,  scaling $E$ by $L^{-1}E$
is not feasible due to the lack of  homogeneity in $c_{h}$. Hence, we prove  \eqref{3.2-w-10-2} by considering two scenarios: $L>1$ and $L\leq 1$.

We start with the first case that  $L\leq 1$.
Let $\phi\in \mathcal {E}_\lambda$. Then there is a sequence  $\{c_k\}$ ($c_k\in\R$ for all $k$) such that
\begin{align}\label{3.3-w-10-2}
    \phi(x)=\sum_{\lambda_k\leq \lambda} c_k\sin\frac{k\pi x}{L},\;x\in[0,L].
\end{align}
Without loss of generality, we can assume
 that $\|\phi\|_{L^2(0,L)}=1$, i.e.,
\begin{align}\label{3.4-w-10-2}
\sum_{\lambda_k\leq \lambda}|c_k|^2=1.
\end{align}
Then there is an $x_0\in(0,L)$ such that
\begin{align}\label{3.5-w-10-2}
|\phi(x_0)|\geq 1.
\end{align}
Meanwhile, direct computations show that
\begin{align*}
\phi^{(m)}(x)=\sum_{\lambda_k\leq \lambda} c_k\Big(\frac{k\pi}{L}\Big)^m\sin\Big(\frac{k\pi x}{L}+\frac{m\pi}{2}\Big),\;\;m=0,1,2,\dots.
\end{align*}
With the Cauchy-Schwarz inequality and \eqref{3.4-w-10-2}, the above leads to
\begin{align}\label{3.6-w-10-2}
\max_{x\in[0,L]}|\phi^{(m)}(x)|\leq \sum_{\lambda_k\leq \lambda}|c_k|\Big(\frac{k\pi}{L}\Big)^m\leq \lambda^{\frac{m}{2}}\sum_{\lambda_k\leq \lambda}|c_k|\leq \sqrt{\frac{L\lambda^{1/2}}{\pi}}\lambda^{\frac{m}{2}},\;\;m=0,1,2,\dots.
\end{align}
 Thus $\phi$ can be extended to an entire function of exponential type. Moreover, it follows from  the Taylor formula, as well as  \eqref{3.6-w-10-2} that  for each $z\in\C$,
\begin{align}\label{3.6-1-w-10-2}
|\phi(z)|\leq \sum_{m=0}^\infty \Big|\frac{\phi^{(m)}(0)}{m!}z^m\Big|\leq \sum_{m=0}^\infty \sqrt{\frac{L\lambda^{1/2}}{\pi}}\lambda^{\frac{m}{2}} \frac{|z|^m}{m!}=\sqrt{\frac{L\lambda^{1/2}}{\pi}}e^{\sqrt{\lambda}|z|}.
\end{align}

Now we let  $\phi_{x_0}(z)=\phi(z+x_0)$, with $z\in\C$. It is clear that  $\phi_{x_0}$ is an entire function with $|\phi_{x_0}(0)|\geq 1$ (see \eqref{3.5-w-10-2}). Since $c_h$ is translation invariant, it follows from Lemma \ref{lem-ana}, as well as \eqref{3.6-1-w-10-2} that
$$
\sup_{x\in [0,1]-x_0}|\phi_{x_0}(x)|\leq M^9\exp\Big\{\frac{C (\log M)^2(\log ( \frac{\log M}{\log 2}+e) )^{-\frac{2\alpha}{3}}}{c^2_h(E)}\Big\}\sup_{x\in E-x_0}|\phi_{x_0}(x)|
$$
where
$$
M=\sup_{z\in D_4}|\phi_{x_0}(z)|\leq \sqrt{\frac{L\lambda^{1/2}}{\pi}}e^{\sqrt{\lambda}(4+L)}<e^{(4+2L)\sqrt{\lambda}}.
$$
Since $L\leq 1$, the above yields
\begin{align*}
&\sup_{x\in [0,L]}|\phi(x)|\leq \sup_{x\in [0,1]}|\phi_{x_0}(x)|\\
&\leq e^{18(L+2)\sqrt{\lambda}}
\exp\Big\{\frac{16C(L+2)^2\lambda(\log (\frac{(4+2L)\sqrt{\lambda}}{\log 2}+e))^{-\frac{2\alpha}{3}}}{c^2_h(E)}\Big\}\sup_{x\in E}|\phi(x)|\\
&\leq e^{54\sqrt{\lambda}}\exp\Big\{\frac{C\lambda(\log  (\lambda+2))^{-\frac{2\alpha}{3}}}{c^2_h(E)}\Big\}\sup_{x\in E}|\phi(x)|
\end{align*}
{with a different constant $C$,}
which leads to \eqref{3.2-w-10-2} for the case that $L\leq 1$.

Next, we consider the second case that  $L>1$. Let $m_0$ be the smallest integer such that $L/m_0\in(0,1]$. Then $L\leq m_0<L+1$.
We divide equally  $[0,L]$  into $m_0$ sub-intervals $I_1,I_2,\dots,I_{m_0}$ (with $|I_j|=L/m_0$ for all $j$).
Then by  the subadditivity of $c_h$, we have
$$
c_h(E)\leq \sum_{j=1}^{m_0}c_h(E\cap I_j).
$$
Thus there is a $j_0\in \{1,2,\cdots,m_0\}$ such that
\begin{align}\label{equ-519-1}
    c_h(E\cap I_{j_0})\geq \frac{c_h(E)}{m_0}>0.
\end{align}
At the same time, by Euler's identity, we obtain from \eqref{3.3-w-10-2} that
$$
\phi(x)=\sum_{\lambda_k\leq \lambda}\frac{c_k}{2}\Big(e^{\frac{ik\pi x}{L}}-e^{\frac{-ik\pi x}{L}}\Big).
$$
This is an exponential polynomial with the pure imaginary part.
By \eqref{3.1-w-10-2}, its order
 $N:=2\#\{k:\lambda_k\leq \lambda\}$ satisfies  $N\leq \frac{2L\sqrt{\lambda}}{\pi}$.
Then, according to Lemma \ref{lemma3.2-w-10-2}, we have
\begin{align}\label{3.9-w-10-2}
\sup_{x\in[0,L]}|\phi(x)|\leq (Cm_0)^{\frac{2L\sqrt{\lambda}}{\pi}}\sup_{x\in I_{j_0}}|\phi(x)|.
\end{align}
Since $|I_{j_0}|\leq 1$,  we can use the result obtained in the first case, as well as \eqref{equ-519-1}, to derive
\begin{align}\label{3.10-w-10-2}
\sup_{x\in I_{j_0}}|\phi(x)| \leq e^{54\sqrt{\lambda}}
\exp\Big\{\frac{Cm^2_0\lambda(\log  (\lambda+e))^{-\frac{2\alpha}{3}}}{c^2_h(E)}\Big\}\sup_{x\in E\cap I_{j_0}}|\phi(x)|.
\end{align}
By \eqref{3.9-w-10-2} and \eqref{3.10-w-10-2},  we obtain
$$
\sup_{x\in[0,L]}|\phi(x)|\leq (Cm_0)^{\frac{2L\sqrt{\lambda}}{\pi}}
e^{54\sqrt{\lambda}}\exp\Big\{\frac{Cm^2_0\lambda(\log  (\lambda+e) )^{-\frac{2\alpha}{3}}}{c^2_h(E)}\Big\}\sup_{x\in E}|\phi(x)|,
$$
which leads to \eqref{3.2-w-10-2} for the case that $L>1$.

Hence, we  complete the proof of Lemma \ref{lem-spec}
\end{proof}

\subsection{Uncertainty principle at scale of $h_\alpha$}\

Typically, in the study of  observability inequalities for the heat equation on the whole space,
the uncertainty principle (attributed to  Logvinenko-Sereda) serves a role analogous to the spectral inequality (see  \cite{Egidisharp2018,Wang2019observable}).
However, at the scale of
$c_{h_\alpha}$, integrals may not  exist. Therefore, to investigate the uncertainty principle at this specific scale,   an alternative working space to $L^2(\R)$ is required. We choose  $l^2L^\infty(\mathbb{R})$,
 defined as follows:
\begin{definition}\label{definition4.1-w-10-4}
Let $l^2L^\infty(\R)$
be the Banach space endowed with the norm
\begin{align}\label{4.12-w-10-21}
\|u\|_{l^2L^\infty(\R)}:=\Big\|\{\|u\|_{L^\infty(k,k+1)}\big\}_{k\in\Z}\Big\|_{l^2(\Z)}.
\end{align}
We abbreviate it to $l^2L^\infty$ when there is no risk to cause any confusion.
\end{definition}
\begin{remark}\label{remark4.2-w-10-4}
$(i)$ Similarly, we can define the space $l^qL^p(\R)$ with $1\leq p, q\leq \infty$, as well as  the space $l^qL^p(\R^d)$.

$(ii)$ For properties of  $l^qL^p$ spaces, we refer readers to \cite{FournierAmalgamsLp1985}.
  \end{remark}

The main result of this subsection is the following uncertainty principle at the scale of $c_{h_\alpha}$. (Here, we recall
Definition \ref{thich-set-9-21-w} for
the concept of $h_\alpha$-$(\gamma,L)$ sets.)
\begin{lemma}\label{lem-LS}
 Let $\alpha>0$. There is a constants $C>0$, depending only on $\alpha$, such that when $E$ is an  $h_\alpha$-$(\gamma, L)$ thick set in $\R$
 for some $L,\gamma>0$,
\begin{align}\label{4.1-w-10-4}
\|u\|_{l^2L^\infty}\leq (4C(  L+1))^{CLN}e^{CN}
\exp\Big\{\frac{CN^2(\log (\frac{CN}{\log 2}+e))^{-\frac{2\alpha}{3}}}{\gamma^2}\Big\}
\Big\|\big\{\sup_{E\cap [k,k+  L]}|u|\big\}_{k\in\Z}\Big\|_{l^2(\Z)}
\end{align}
for each  $u\in l^2L^\infty$ with  $\emph{supp } \widehat{u}\subset [-N,N]$ for some $N>0$. (Here $\widehat{u}$
is the Fourier transform of $u$.)
\end{lemma}
\begin{remark}
    The estimate \eqref{4.1-w-10-4} remains valid when  when $l^2L^\infty$ (on the LHS of \eqref{4.1-w-10-4}) is replaced by $l^qL^\infty$   and $l^2(\Z)$ (on the RHS of \eqref{4.1-w-10-4}) is replaced by $l^q(\Z)$  with $1\leq q\leq\infty$. This can be proved using the same method to the one in the proof of Lemma \ref{lem-LS}.
   \end{remark}

Before proving Lemma \ref{lem-LS}, we establish the following Bernstein-type inequality:

\begin{lemma}\label{lem-bern}
Let $1\leq p,q\leq \infty$. Then there exists a numerical constant $C>0$ such that
when $u\in l^qL^p$ with  $\emph{supp } \widehat{u}\subset [-N,N]$ for some $N>0$,
\begin{align}\label{4.2-w-10-4}
\|\partial_x^mu\|_{l^qL^p}\leq (CN)^m\|u\|_{l^qL^p}\;\;\mbox{for all}\;\;m=0,1,2,\dots.
\end{align}
\end{lemma}
\begin{proof}
We arbitrarily fix  $u\in l^qL^p$  with $\emph{supp } \widehat{u}\subset [-N,N]$ for some $N>0$.
 Let $\varphi$ be a Schwartz function such that
 $$
 \widehat{\varphi}(\xi)=1 \mbox{ for } |\xi|\leq 1, \quad \widehat{\varphi}(\xi)=1 \mbox{ for } |\xi|\geq 2.
 $$
 It is clear that  $\widehat{\varphi}(\frac{\xi}{N})=1$ for  $\xi\in [-N,N]$. Then by Fourier's inversion formula, we see
 $$
 u(x)=(2\pi)^{-1}\int_{\R} N\varphi(N(x-y))u(y)\d y.
 $$
Since  $\varphi$ is a Schwartz function, taking the derivative in the above leads to
\begin{align}\label{equ-bern-1}
    |\partial_xu(x)|\leq (\Psi*|u|)(x) =\int_\R \Psi(x-y)|u(y)|\d y,\; x\in \R,
\end{align}
where $\Psi(x)=CN^2(1+N|x|)^{-2}$ ($x\in\R$) for some numerical constant $C>0$.

Meanwhile, by  Young's inequality in the space  $L^p$, we have that for each  $k\in \Z$,
\begin{align*}
 \|\Psi*|u|\|_{L^p(k,k+1)}&\leq \Big \|\sum_{l\in \Z}\int_{[l,l+1]}\Psi(x-y)|u(y)|\d y\Big\|_{L^p(k,k+1)}\\
 &\leq \sum_{l\in \Z} \|u\|_{L^p(l,l+1)}\|\Psi\|_{L^1(k-l-1,k-l+1)}.
\end{align*}
Then by Young's inequality in  space  $l^q$, we obtain
\begin{align*}
 \|\Psi*|u|\|_{l^qL^p} \leq  2\|u\|_{l^qL^p}\|\Psi\|_{L^1(\R)}\leq 2\pi CN\|u\|_{l^qL^p}.
\end{align*}
Combining this with   \eqref{equ-bern-1}, we derive
\begin{align*}
 \|\partial_xu\|_{l^qL^p}  \leq C'N\|u\|_{l^qL^p},\;\;\mbox{with}\;\;C'=2\pi C.
\end{align*}
Iterating this result leads to  \eqref{4.2-w-10-4}. This completes the proof of Lemma \ref{lem-bern}.
\end{proof}

\begin{proof}[Proof of Lemma \ref {lem-LS}]
We arbitrarily fix $\alpha>0$ and simply write $h$ for $h_\alpha$. Then we arbitrarily fix $E$ and $u$ with conditions in the lemma.

 We adopt the idea in  \cite{Kovrijkine2001some}. To this end, we write
  \begin{align*}
 \R=\bigcup_{k\in \Z} I_k,\;\;\mbox{with}\;\; I_k:=  [k,k+1].
 \end{align*}
 We say that $I_k$ is a good interval, if
 \begin{align}\label{equ-good}
  \|u^{(m)}\|_{L^\infty(I_k)}\leq (AN)^m\|u\|_{L^\infty(I_k)}\;\;\mbox{for all}\;\;m\geq 0,
\end{align}
 with $A:=3^{1/2}C$, where $C$ is the constant given by  Lemma \ref{lem-bern} when $q=2$ and $p=\infty$.
 Conversely,
  $I_k$ is a bad interval, if it is not good, i.e.,
 there is an $ m\geq 1$ such that
\begin{align}\label{equ-bad}
\|u^{( m)}\|_{L^\infty(I_k)}\geq (AN)^{ m}\|u\|_{L^\infty(I_k)}.
\end{align}
  It follows from  \eqref{equ-bad},   Lemma \ref{lem-bern} and the definition  of $A$ that
  \begin{align}\label{eq-1118-h}
\sum_{\mbox{\footnotesize bad}~ I_k} \|u\|^2_{L^\infty(I_k)}
&\leq  \sum_{m\geq 1}\sum_{\mbox{\footnotesize bad}~ I_k} (AN)^{-2m}\|u^{(m)}\|^2_{L^\infty(I_k)}\\
&\leq  \sum_{m\geq 1}(AN)^{-2m}\|u^{(m)}\|^2_{l^2L^\infty} \nonumber\\
&\leq   \sum_{m\geq 1}C^{2m}A^{-2m}\|u\|^2_{l^2L^\infty}
\leq \frac{1}{2}\|u\|^2_{l^2L^\infty}.\nonumber
\end{align}
With the definition $l^2L^\infty$, the above leads to
 \begin{align}\label{4.7-w-10-4}
\sum_{\mbox{\footnotesize good}~ I_k} \|u\|^2_{L^\infty(I_k)} \geq \frac{1}{2}\|u\|^2_{l^2L^\infty}.
\end{align}

The rest of the proof is divided into two steps.
\vskip 5pt

\noindent \emph{Step 1. We prove that when $I_k$ is a good interval,
\begin{align}\label{equ-LS-6}
\sup_{x\in [k,k+1]}|u(x)|
 \leq \Xi \sup_{x\in E\cap [k-2 L,k+2 L]}|u(x)|,
\end{align}
where
\begin{align*}
\Xi= (4C(L+1))^{C(L+1)N}e^{36AN}\exp\Big\{\frac{144A^2CN^2(\log (\frac{2AN}{\log 2}+e))^{-\frac{2\alpha}{3}}}{\gamma^2}\Big\}.
\end{align*}
}
\vskip 5pt
For this purpose, we arbitrarily fix  a good interval $I_k$. Then by \eqref{equ-good}, we see that  $u$ can be extended to an entire function such that
\begin{align}\label{equ-LS-0}
|u(z-x)|\leq e^{AN|z-x|} \quad \mbox{ for all } z\in \C, x\in I_k.
\end{align}
Again, since the Hausdorff content $c_{h}$ lacks homogeneity, we will prove  \eqref{equ-LS-6} by considering
 two cases: $L\geq 1$ and $L<1$.

In the first case that  $L\geq 1$, we let  $m_0=\lceil L \rceil$, i.e., $m_0$ is the smallest integer
such that $m_0>L$. Since $E$ is $h$-$(\gamma,L)$ thick, we find
$$
c_h(E\cap[k,k+m_0])\geq \gamma L.
$$
Meanwhile, by the subadditivity of $c_h$, we have
$$
c_h(E\cap[k,k+m_0])\leq \sum_{0\leq j\leq 2m_0-1}c_h\Big(E\cap\Big[k+\frac{j}{2},k+\frac{j+1}{2}\Big]\Big).
$$
The above two estimates indicate that there is  a $j_0\in \{0,1,\cdots,2m_0-1\}$ such that
\begin{align}\label{equ-LS-1}
c_h\Big(E\cap\Big[k+\frac{j_0}{2},k+\frac{j_0+1}{2}\Big]\Big)    \geq \frac{\gamma L}{2m_0}\geq
\frac{\gamma}{3}.
\end{align}
Without loss of generality, we can assume that $u\neq 0$. Since $u$ is analytic, it does not vanish identically on
$[k+\frac{j_0}{2},k+\frac{j_0+1}{2}]$. By homogeneity, we can assume that
\begin{align}\label{4.11-w-10-4}
\max_{x\in[k+\frac{j_0}{2},k+\frac{j_0+1}{2}]}|u|=|u(x_0)|\geq 1\;\;\mbox{for some}\;\;x_0\in \Big[k+\frac{j_0}{2},k+\frac{j_0+1}{2}\Big].
\end{align}
It is clear that
$\big[k+\frac{j_0}{2},k+\frac{j_0+1}{2}\big]\subset [x_0-\frac{1}{2},x_0+\frac{1}{2}]$. This, along with \eqref{equ-LS-1}, yields
\begin{align}\label{equ-LS-3}
c_h(E\cap[x_0-1/2,x_0+1/2])    \geq\frac{\gamma}{3}.
\end{align}
Thus we  apply Lemma \ref{lem-ana}  with  $\phi(x)=u(x+x_0)$, $I=[x_0-1/2,x_0+1/2]$ and $E$  replaced by $E\cap[x_0-1/2,x_0+1/2]$
to derive
$$
\sup_{x\in [x_0-\frac{1}{2},x_0+\frac{1}{2}]}|u(x)|\leq M^9\exp\Big\{\frac{C (\log M)^2(\log (\frac{\log M}{\log 2}+e))^{-\frac{2\alpha}{3}}}{c^2_h(E\cap [x_0-\frac{1}{2},x_0+\frac{1}{2}])}\Big\}\sup_{x\in E\cap [x_0-\frac{1}{2},x_0+\frac{1}{2}]}|u(x)|,
$$
where
$$
M=\sup_{|z-x_0|\leq 4}|\phi(z)|\leq e^{4AN}.
$$
(Here we used \eqref{equ-LS-0}.) This, together with \eqref{equ-LS-3}, implies that
\begin{align}\label{equ-LS-4}
 \sup_{x\in [x_0-\frac{1}{2},x_0+\frac{1}{2}]}|u(x)|\leq e^{36AN}
 \exp\Big\{\frac{144A^2CN^2(\log \frac{2AN}{\log 2})^{-\frac{2\alpha}{3}}}{\gamma^2}\Big\}\sup_{x\in E\cap [x_0-\frac{1}{2},x_0+\frac{1}{2}]}|u(x)|.
\end{align}
On the other hand,
by \eqref{2.43-w-10-3} in Remark \ref{rem-Kov}, using a scaling argument, we have
\begin{align}\label{equ-LS-5}
\sup_{x\in [k,k+1]}|u(x)|\leq \sup_{x\in[x_0-2L,x_0+2L]}|u(x)|\leq (4CL)^{CLN} \sup_{x\in [x_0-\frac{1}{2},x_0+\frac{1}{2}]}|u(x)|  .
\end{align}
Now,  \eqref{equ-LS-6} in the case that $L\geq 1$  follows from \eqref{equ-LS-4} and \eqref{equ-LS-5} at once.

In the second case where $L<1$, it follows from Definition \ref{thich-set-9-21-w} that $E$ is $c_h$-$(\gamma,1)$ thick. Thus,
we can use  \eqref{equ-LS-6} with $L=1$ to obtain \eqref{equ-LS-6} for the second case. (Of course, the constant $C$ varies.)

Thus, we have proved \eqref{equ-LS-6} for any good interval $I_k$.

\vskip 5pt
\noindent
\emph{Step 2. We prove \eqref{4.1-w-10-4}.}
\vskip 5pt

By  \eqref{equ-LS-6}, we see
\begin{align*}
\sum_{\mbox{\footnotesize good}~ I_k} \|u\|^2_{L^\infty(I_k)}
&\leq \Xi^2 \sum_{\mbox{\footnotesize good}~ I_k}\Big (\sup_{x\in E\cap [k-2L,k+2L]}|u(x)|\Big)^2\\
&\leq \Xi^2 \sum_{k\in \Z}\Big (\sup_{x\in E\cap [k-2L,k+2L]}|u(x)|\Big)^2\\
&\leq \Xi^2 \sum_{k\in \Z} 4\sum_{m\in\{-2,-1,0,1\}}\Big(\sup_{x\in E\cap [k+mL,k+(m+1)L]}|u(x)|\Big)^2\\
&=(4\Xi)^2\sum_{k\in \Z}\Big(\sup_{x\in E\cap [k,k+L]}|u(x)|\Big)^2.
\end{align*}
With \eqref{4.7-w-10-4}, the above leads to
$$
\|u\|_{l^2L^\infty}\leq 2^{1/2}4\Xi\Big(\sum_{k\in \Z}(\sup_{x\in E\cap [k,k+L]}|u(x)|)^2\Big)^{1/2}.
$$

Therefore, we finish the proof of  Lemma \ref {lem-LS}.
\end{proof}


\subsection{Proofs of Theorem \ref{thm-ob-bound} and Theorem \ref{thm-ob-R} }\

This subsection presents the proofs of  Theorem \ref{thm-ob-bound} and Theorem \ref{thm-ob-R}.
The proofs, with the exception of  conclusion $(ii)$ in  Theorem \ref{thm-ob-bound}, can be done by  using the  adapted  Lebeau-Robbiano strategy provided in \cite{Duyckaerts2012resolvent}, together with
 the spectral inequality in Lemma \ref{lem-spec} and the uncertainty principle in Lemma \ref{lem-LS} respectively. Although we cannot directly apply  the results in \cite{Duyckaerts2012resolvent},
 the underlying ideas are  applicable  to our case. For the reader's convenience, we provide
 detailed proofs here.

\begin{proof}[Proof of Theorem \ref{thm-ob-bound}]
We start with proving the conclusion $(i)$. Let $A$ and  $\mathcal {E}_\lambda$  be given at the beginning of
subsection \ref{subsect 4.1-w-10-25}. We arbitrarily fix $\alpha>0$, and then arbitrarily fix
$E\subset (0,L)$ with $c_{h_\alpha}(E)>0$.
The proof is organized by two steps.
\vskip 5pt
\noindent {\it Step 1. We build up  a recurrence inequality.}

\vskip 5pt

First, by Lemma \ref{lem-spec},
there exists a constant $C>0$, depending only on $L,\alpha,c_{h_{\alpha}}(E)$, and a numerical constant $\lambda_0\gg 1$ such that   when $\lambda \geq \lambda_0$,
\begin{align}\label{equ-1022-1}
 \sup_{x\in [0,L]}|\phi(x)|\leq  e^{C\lambda(\log  \lambda )^{-\frac{2\alpha}{3}}} \sup_{x\in E}|\phi(x)|   \;\;\mbox{for each}\;\; \phi\in \mathcal {E}_\lambda.
\end{align}
Meanwhile, since  $\{e^{tA}\}_{t\geq 0}$ is a  contraction $C_0$-semigroup on $L^2(0,L)$, we have
\begin{align*}
\|e^{tA}\phi\|_{L^2(0,L)}\leq \|\phi\|_{L^2(0,L)}\;\;\mbox{for each}\;\; t\geq 0.
\end{align*}
Using  H\"older inequality in the above implies that when $\tau>0$ and $\phi\in {L^2(0,L)}$,
\begin{align}\label{equ-1022-2}
\|e^{\tau A}\phi\|_{L^2(0,L)}\leq \frac{2}{\tau }\int_{\frac{\tau }{2}}^\tau \|e^{t A}\phi\|_{L^2(0,L)}\d t\leq \frac{2\sqrt{L}}{\tau }\int_{\frac{\tau }{2}}^\tau \|e^{t A}\phi\|_{L^\infty(0,L)}\d t.
\end{align}
Since $e^{tA}\mathcal {E}_\lambda\subset \mathcal {E}_\lambda$ when $t\geq 0$,
 it follows from \eqref{equ-1022-1} and \eqref{equ-1022-2} that
\begin{align}\label{equ-1022-3}
\|e^{\tau A}\phi\|_{L^2(0,L)}  \leq \frac{2\sqrt{L}}{\tau }e^{C\lambda(\log  \lambda )^{-\frac{2\alpha}{3}}}\int_{\frac{\tau }{2}}^\tau  \sup_{x\in E}|e^{t A}\phi| \d t\;\;\mbox{for all}\;\; \phi\in \mathcal {E}_\lambda,\;\tau>0.
\end{align}

Next, we define two functions $\varphi$ ad $\psi$ as follows:
\begin{align*}
\psi(t):=t\log t,\;\;t>0;\;\;\;\;\varphi(\lambda):=(\log \lambda)^{\frac{2\alpha}{3}},\;\;\lambda\gg 1.
\end{align*}
Since
\begin{align*}
\lim_{\tau\to 0^+}\tau\psi(\frac{1}{\tau})=+\infty\;\;\mbox{and}\;\;\lim_{\lambda\to+\infty}\frac{C\lambda}{\varphi(\lambda)}=+\infty,
\end{align*}
we can choose $T_0>0$ such that
$$
\frac{1}{\tau }e^{C\lambda(\log  \lambda )^{-\frac{2\alpha}{3}}}=e^{\tau\psi(\frac{1}{\tau})+\frac{C\lambda}{\varphi(\lambda)}}\leq e^{\tau\psi(\frac{1}{\tau})\frac{C\lambda}{\varphi(\lambda)}},\quad\tau\in (0,T_0],\,\,\,\lambda\ge \lambda_0.
$$
With  \eqref{equ-1022-3}, the above yields that when $\tau\in (0,T_0]$,
\begin{align}\label{equ-1022-4}
\|e^{\tau A}\phi\|_{L^2(0,L)}  \leq 2\sqrt{L}e^{C\tau\psi(\frac{1}{\tau})\frac{\lambda}{\varphi(\lambda)}} \int_{\frac{\tau }{2}}^\tau  \sup_{x\in E}|e^{t A}\phi| \d t\;\;\mbox{for each}\;\; \phi\in \mathcal {E}_\lambda.
\end{align}

We now arbitrarily fix  $\phi\in L^2(0,L)$. For every $\lambda>0$, let $\phi_\lambda$ be the orthogonal projection of $\phi$ in $\mathcal {E}_\lambda$ and set $\phi^{\perp}_\lambda:=\phi-\phi_\lambda$. Then $\phi^{\perp}_\lambda\in \mathcal {E}^{\perp}_\lambda$ (the orthogonal complement of $\mathcal {E}_\lambda$ in $L^2(0,L)$).
It follows from \eqref{equ-1022-4} and the triangle inequality that when $\tau\in(0,T_0]$ and $\lambda\geq \lambda_0$,
\begin{align}\label{equ-1023-1}
\|e^{\tau A}\phi\|_{L^2(0,L)}
&\leq \|e^{\tau A}\phi_\lambda\|_{L^2(0,L)}+\|e^{\tau A}\phi^\perp_\lambda\|_{L^2(0,L)} \nonumber\\
&\leq 2\sqrt{L}e^{C\tau\psi(\frac{1}{\tau})\frac{\lambda}{\varphi(\lambda)}}\int_{\frac{\tau }{2}}^\tau  \sup_{x\in E}|e^{tA}\phi_\lambda|   \d t+\|e^{\tau A}\phi^\perp_\lambda\|_{L^2(0,L)}\nonumber\\
&\leq 2\sqrt{L}e^{C\tau\psi(\frac{1}{\tau})\frac{\lambda}{\varphi(\lambda)}}\int_{\frac{\tau }{2}}^\tau  \sup_{x\in E}|e^{tA}\phi|   \d t\nonumber\\
&\quad+2\sqrt{L}e^{C\tau\psi(\frac{1}{\tau})\frac{\lambda}{\varphi(\lambda)}}\int_{\frac{\tau }{2}}^\tau  \sup_{x\in E}|e^{tA}\phi^\perp_\lambda|   \d t+\|e^{\tau A}\phi^\perp_\lambda\|_{L^2(0,L)}.
\end{align}
For the last term on the right hand side of  \eqref{equ-1023-1}, we use facts that
$\|e^{\tau A}\phi\|_{L^2(0,L)}\leq e^{-\lambda \tau }\|\phi\|_{L^2(0,L)}$
 and $\|\phi^\perp_\lambda\|_{L^2(0,L)}\leq\|\phi\|_{L^2(0,L)}$ to get
\begin{align}\label{equ-1023-2}
\|e^{\tau A}\phi^\perp_\lambda\|_{L^2(0,L)}\leq e^{-\lambda \tau}\|\phi\|_{L^2(0,L)}\;\;\mbox{for all}\;\;\tau\in (0,T_0], \;\lambda\geq \lambda_0.
\end{align}
For the second term on the right hand side of  \eqref{equ-1023-1}, we use the $L^2\to L^\infty$ estimate:
\begin{align*}
\|e^{tA}\|_{L^2(0,L)\to L^\infty(0,L)}\leq a_0t^{-\frac{1}{4}}, \;\;t>0,
\end{align*}
for some constant $a_0>0$, depending only on $L$, and use the same argument proving
\eqref{equ-1023-2} to obtain that when $\tau\in(0,T_0]$ and $\lambda\geq \lambda_0$,
\begin{align}\label{equ-1023-3}
\int_{\frac{\tau }{2}}^\tau \sup_{x\in E}|e^{tA}\phi^\perp_\lambda|  \d t&\leq \int_{\frac{\tau }{2}}^\tau a_0(t-\frac{\tau}{3})^{-\frac{1}{4}}\|e^{\frac{\tau}{3}A}\phi^\perp_\lambda\|_{L^2(0,L)}  \d t
\leq  a_0(\frac{6}{\tau})^{\frac{1}{4}}e^{-\frac{\lambda\tau}{3}}\|\phi\|_{L^2(0,L)}.
\end{align}
Here\footnote{ In fact, in the dim-1 case, the interval length $\tau/2$ can absorb the singular term $(\frac{6}{\tau})^{1/4}$ as $\tau\to 0$. But we do not exploit this favor here. The reason is that in the higher dimension, the singular factor $(\frac{6}{\tau})^{1/4}$ will be replaced by $(\frac{6}{\tau})^{d/4}$, which could not be absorbed by $\tau/2$ if $d$ is large. Thus the proof present here also works in all dimension $d\geq 2$.} in the last step we use the fact that the length of the integral interval $[\tau/2,\tau]$ is less than $1$.

Letting $g(\tau,\lambda):=\frac{1}{2\sqrt{L}}e^{-C\tau\psi(\frac{1}{\tau})\frac{\lambda}{\varphi(\lambda)}}$, with $\tau\in(0,T_0]$ and $\lambda\geq\lambda_0$,  plugging \eqref{equ-1023-2} and \eqref{equ-1023-3} into \eqref{equ-1023-1}, we see that when $\tau\in(0,T_0]$ and $\lambda\geq \lambda_0$,
\begin{align}\label{equ-1023-4}
g(\tau,\lambda)\|e^{\tau A}\phi\|_{L^2(0,L)}
&\leq  \int_{\frac{\tau }{2}}^\tau \sup_{x\in E}|e^{tA}\phi|  \d t
+(a_0(\frac{6}{\tau})^{\frac{1}{4}}e^{-\frac{\lambda\tau}{3}}+g(\tau,\lambda)e^{-\lambda
\tau})\|\phi\|_{L^2(0,L)}.
\end{align}

Next we will estimate the last term on the right hand side of \eqref{equ-1023-4}. For this purpose, we
choose  $\tau$  to be a function of $\lambda$ in the following manner:
\begin{align}\label{equ-1023-5}
\frac{1}{\tau(\lambda)}:=\psi^{[-1]}\Big(\frac{\varphi(\lambda)}{4C} \Big), \quad \lambda\geq \lambda_0
\end{align}
and then define another function $f$ by
\begin{align}\label{equ-1023-6}
f(\lambda):=g(\tau(\lambda),\lambda)=\frac{1}{2\sqrt{L}}e^{-\frac{\lambda \tau(\lambda)}{4}},\;\;\lambda\geq \lambda_0.
\end{align}
It is easy to check that $\psi^{[-1]}(\lambda)\sim \frac{\lambda}{\log \lambda}$ for $\lambda\gg 1$. Thus we deduce from \eqref{equ-1023-5} that
\begin{align}\label{equ-1023-7}
  \tau(\lambda)\sim \frac{\log\log \lambda}{(\log \lambda)^{\frac{2\alpha}{3}}}, \quad \lambda\gg 1.
\end{align}
By \eqref{equ-1023-5} and
\eqref{equ-1023-7},
we see that when $\varepsilon>0$ is sufficiently small, there is a constant $C_\varepsilon>0$ depending only on $\varepsilon>0$ such that
\begin{align*}
\tau(\lambda)^{-\frac{1}{4}}e^{-\frac{\lambda\tau(\lambda)}{3}}\leq C_\varepsilon e^{-(\frac{1}{3}-\varepsilon)\lambda \tau(\lambda)},\;\;\lambda\geq\lambda_0.
\end{align*}
This, along with \eqref{equ-1023-6} and the fact that $\tau(\lambda)\ge\tau(\frac{5}{4}\lambda)$ for $\lambda\gg 1$,  yields that there is $\lambda_1>0$, depending only on $a_0$ and $L$, such that
\begin{align}\label{equ-1023-8}
\Big(a_0\Big(\frac{6}{\tau}\Big)^{\frac{1}{4}}e^{-\frac{\lambda\tau(\lambda)}{3}}+g(\tau(\lambda),\lambda)e^{-\lambda
\tau(\lambda)}\Big)\|\phi\|_{L^2(0,L)}\leq f\Big(\frac{5}{4}\lambda\Big)\|\phi\|_{L^2(0,L)},\,\,\,\lambda\geq\lambda_1.
\end{align}
By \eqref{equ-1023-8} and \eqref{equ-1023-4}, we obtain the following recurrence inequality:
\begin{align}\label{equ-1023-9}
f(\lambda)\|e^{\tau(\lambda) A}\phi\|_{L^2(0,L)}
-f(\frac{5}{4}\lambda)\|\phi\|_{L^2(0,L)}\leq  \int_{0}^{\tau(\lambda)} \sup_{x\in E}|e^{tA}\phi|  \d t\;\;\mbox{for each}\;\;\lambda\geq\lambda_1.
\end{align}
\vskip 5pt
\noindent {\it
Step 2. We use \eqref{equ-1023-9} to prove \eqref{thm-ob-1024h-1}.}
\vskip 5pt

First, we define  two sequences as follows:
\begin{align}\label{equ-1023-10}
\lambda_{k+1}=\frac{5}{4}\lambda_k;\; \quad \tau_k=\tau(\lambda_k),\;\;k\geq 1.
\end{align}
Thanks to the assumption that $\alpha>\frac{3}{2}$, the series $\sum_{k}\tau_k$ converges. Indeed,
by \eqref{equ-1023-5}, we see that $\tau(\lambda)$ is positive and decreasing when $\lambda\gg 1$, thus, it follows from
\eqref{equ-1023-10} that
\begin{align}\label{equ-1023-11}
\sum_{k\geq 3}\tau_k=\tau(q^2\lambda_1)+\tau(q^3\lambda_1)+\ldots\leq \int_1^\infty \tau(q^s\lambda_1)\d s,\;\;\mbox{with}\;\;q=5/4.    \end{align}
Since $\alpha>\frac{3}{2}$
and because of \eqref{equ-1023-7}, we see that the integral on the right hand side of \eqref{equ-1023-11} is finite.

Next, let $\{T_n\}_{n\geq 1}$
be such that
\begin{align}\label{equ-1023-12}
 T_n=\sum_{k\geq n}\tau_k, \quad n\geq 1.
\end{align}
Then $T_n\to 0$ as $n\to \infty$.
We arbitrarily fix $T>0$. Then there exists an $N\geq 1$ such that $T_N\leq T$ and $T_N\leq T_0$  (which was introduced in Step 1).  By \eqref{equ-1023-9}, where $\phi$ and $\lambda$ are replaced by $e^{T_{k+1}A}\phi$ and $\lambda_k$ respectively), we find
that when $k\geq N$,
\begin{align}\label{equ-1023-13}
f(\lambda_k)\|e^{T_kA}\phi\|_{L^2(0,L)}-f(\lambda_{k+1})\|e^{T_{k+1}A}\phi\|_{L^2(0,L)}\leq \int_{T_{k+1}}^{{T_k}} \sup_{x\in E}|e^{tA}\phi|  \d t.
\end{align}
Summing above for $k\geq N$,  noting that $f(\lambda_k)\to 0$ as $k\to \infty$ and $\|e^{T_{k+1}A}\phi\|_{L^2(0,L)}$ is uniformly bounded with respect to $k$, we find
\begin{align}\label{equ-1023-14}
f(\lambda_N)\|e^{T_NA}\phi\|_{L^2(0,L)} \leq \int_{0}^{{T_N}} \sup_{x\in E}|e^{tA}\phi|  \d t.
\end{align}
Since $T\geq T_N$, it follows from \eqref{equ-1023-14} that
$$
\|e^{TA}\phi\|_{L^2(0,L)}\leq \|e^{T_NA}\phi\|_{L^2(0,L)} \leq \frac{1}{f(\lambda_N)} \int_{0}^{{T_N}} \sup_{x\in E}|e^{tA}\phi|  \d t\leq \frac{1}{f(\lambda_N)} \int_{0}^{T} \sup_{x\in E}|e^{tA}\phi|  \d t,
$$
which leads to  \eqref{thm-ob-1024h-1}.

This completes the proof of the conclusion $(i)$ in   Theorem \ref{thm-ob-bound}. The conclusion $(ii)$ in   Theorem \ref{thm-ob-bound}  follows from the next Proposition \ref{thm-1017-1.1}.
Therefore the proof of Theorem \ref{thm-ob-bound} is completed.
\end{proof}

\begin{proposition}\label{thm-1017-1.1}
For each $\varepsilon>0$, define the gauge function
\begin{equation}\label{4.47-w-g-10-25}
    f_\varepsilon(t):=(\log(1/t))^{-\frac{1}{2+\varepsilon}}, \;\;t\in(0,e^{-3}].
\end{equation}
 There is a subset $E_\infty\subset (0, L)$ satisfying the following two conclusions:
\begin{enumerate}[(i)]
    \item The observability inequality \eqref{thm-ob-1024h-1} (where  $E$ is replaced by $E_\infty$)
    fails.

    \item
    $0<c_{f_\varepsilon}(E_\infty)<\infty$ but $c_{h_\alpha}(E_\infty)=0$ for each $\alpha>0$,
    where $h_\alpha$ is given by \eqref{1.7w9-21}.

\end{enumerate}
\end{proposition}

\begin{proof}
Without loss of generality, we can assume that    $L=1$. We arbitrarily fix  small $\varepsilon>0$
and $\alpha>0$. The proof is organized by several steps as follows:
\vskip 5pt

\noindent {\it Step 1.  We construct $E_\infty$.}
\vskip 5pt

We take $\varepsilon_2\in (0,\frac{\varepsilon}{3})$ and $\varepsilon_1\in(0,\varepsilon)$ such that
    \begin{equation}\label{eq-1017-01}
       (2+\varepsilon)(1-\varepsilon_2)=2+\varepsilon_1.
    \end{equation}
    We then construct
 a sequence of positive integers $\{q_k\}_{k\ge 1}$ in the following manner. Let  $q_1=4$.
 For each  $k\ge 2$, we take  $q_k$ such that
    \begin{equation}\label{eq-1017-02}
     \exp\left\{\frac{1}{\varepsilon_2}q_{k-1}^{2+\varepsilon_1}\right\}\leq q_k<\exp\left\{Nq_{k-1}^{2+\varepsilon_1}\right\},\,\,\,N=[\frac{1}{\varepsilon_2}]+1.
    \end{equation}
Clearly, such a sequence exists.
Before continuing our proof, we introduce the following notation:  For each $x\in\mathbb{R}$, we write $\Vert x \Vert:=\min\limits_{n\in\mathbb{Z}}|x-n|$, which represents the distance from
     $x$ to the nearest integer.
Now, we define, for each positive integer $k$, the set:
    \begin{equation}\label{eq-1017-03}
      A_k:=\left\{x\in(0,1)\setminus\mathbb{Q},\,\Vert q_kx\Vert<q_k\exp\{-q_k^{2+\varepsilon_1}\}\right\},
    \end{equation}
where  $\mathbb{Q}$ denotes the set of all rational numbers.
Let
\begin{align*}
    E_1=A_1;\;\;\;E_k=\bigcap\limits_{j=1}^k A_j\;\;\mbox{for}\;\;k\geq 2.
\end{align*}
It is clear that
 $E_{k+1}\subset E_k$ for all $k$. We now define the desired set
    \begin{equation}\label{eq-1017-04}
      E_{\infty}=\bigcap\limits_{k=1}^\infty E_k.
    \end{equation}
Note that $E_\infty$ is a subset of all \emph{Liouville numbers} in $[0,1]$.

\vskip 5pt
\noindent {\it Step 2. We prove the conclusion  (i).}
\vskip 5pt

We choose the following function as the initial datum of  \eqref{equ-heat-b} (with $L=1$)
    \begin{equation*}
       u_{0,n}(x)=\sin{n\pi x},\; x\in [0,1],\;\;\mbox{with}\;\;n\in\N^+.
    \end{equation*}
Then the corresponding solution to  \eqref{equ-heat-b} is
    \begin{equation}\label{eq-1017-05}
      u_n(t,x)=e^{-n^2\pi^2t}\sin{n\pi x},\;\;t\geq 0,\; x\in [0,1].
    \end{equation}

By contradiction, we suppose that $E_\infty$ satisfies  \eqref{thm-ob-1024h-1} for some $T>0$, i.e., there is an absolute constant $C_{obs}>0$ such that
$$
\|u(T,\cdot)\|_{L^2(0,L)}\leq C_{obs}\int_0^T\sup_{x\in E_{\infty}}|u(t,x)|\d t.
 $$
Then we have
    \begin{equation}\label{eq-1017-06}
       \frac{1}{C_{obs}}\leq\frac{\int_0^T\sup\limits_{x\in E_\infty}|u_n(t,x)|\d t}{(\int_0^1|u_n(T,x)|^2\d x)^{1/2}}\;\;\mbox{for all } n\in \N^+.
    \end{equation}
However, taking $n=q_k$ leads to
    \begin{equation*}
		\begin{split}
			\sup\limits_{x\in E_\infty}|u_n(t,x)|&=e^{-q_k^2\pi^2t}\sup\limits_{x\in E_\infty}|\sin{q_k\pi x}|\\
			&\leq \pi e^{-q_k^2\pi^2t}\sup\limits_{x\in E_\infty}\Vert q_kx\Vert\\
            &\leq \pi q_ke^{-q_k^2\pi^2t}e^{-q_k^{2+\varepsilon_1}},
		\end{split}
	\end{equation*}
 where in the first equality, we used \eqref{eq-1017-05};  in the first inequality, we used the fact that $|\sin{a\pi}|\leq\pi \Vert  a\Vert$ for $a\in\mathbb{R}$; and in the second inequality, we used \eqref{eq-1017-03}.
Hence,
    \begin{equation}\label{eq-1017-07}
     \int_0^T\sup|u_n(t,x)|\d t\le\frac{1}{\pi q_k}e^{-q_k^{2+\varepsilon_1}}.
    \end{equation}
Meanwhile, taking $n=q_k$ also yields
\begin{equation}\label{eq-1017-08}
     \left(\int_0^1|u_n(T,x)|^2\d x\right)^{1/2}=e^{-q_k^2\pi^2T}\cdot\left(\int_0^1|\sin{q_k\pi x}|^2\d x\right)^{1/2}\gtrsim e^{-q_k^2\pi^2T}.
    \end{equation}
Combining \eqref{eq-1017-07} and \eqref{eq-1017-08} shows
    \begin{equation*}
    	\frac{\int_{0}^{T}\sup\limits_{x\in E_{\infty}}|u_{n}(t,x)|\d t}{(\int_{0}^{1}|u_{n}(T,x)|^{2}\d x)^{1/2}} \lesssim \frac{1}{q_k}e^{-q_{k}^{2}(q_{k}^{\varepsilon_{1}}-\pi^{2}T)}\longrightarrow 0,\,\,\,\mbox{as}\,\,\, k\rightarrow \infty,
    \end{equation*}
which contradicts with \eqref{eq-1017-06} . So the conclusion  (i) is true.

\vskip 5pt
\noindent{\it Step 3. We prove $0<c_{f_\varepsilon}(E_\infty)<\infty$.}
 \vskip 5pt

Since $E_{\infty}$ is bounded, we have $c_{f_{\varepsilon}}(E_{\infty}) < \infty$.
The remainder is to  prove
\begin{align}\label{4.55-w-10-25}
 c_{f_{\varepsilon}}(E_{\infty}) > 0.
\end{align}
This will be done, if we can show  $m_{f_{\varepsilon}}(E_{\infty}) > 0$ (see  Remark \ref{remark2.2-w-9-29}  $(ii)$). To prove the latter, it suffices to show that
 there exists a Radon measure $\mu$ with $\mbox{supp}\; \mu\subset E_{\infty}$ and $0<\mu(\R)<\infty$, such that
    \begin{equation}\label{eq-1017-09}
    	\mu(B(x,r)) \lesssim f_{\varepsilon}(r)=(\log \frac{1}{r})^{-\frac{1}{2+\varepsilon}}\;\;\mbox{for all}\;\;x \in \mathbb{R}, r>0.
    \end{equation}
Indeed, \eqref{eq-1017-09}, along with the mass  distribution principle (see e.g. in \cite[p.112]{Mattila1995geometry} or \cite[p.60]{Falconer90}),
yields  $m_{f_{\varepsilon}}(E_{\infty}) > 0$.

Our proof of \eqref{eq-1017-09} needs the following:

 \underline{Claim 1.} {\it If we define, for each $k\in\N^+$, the intervals
\begin{equation}\label{eq-1022h-18} I_{k,j}=\left(\frac{j}{q_k}-\frac12\exp\{-q_k^{2+\varepsilon_1}\}, \frac{j}{q_k}+\frac12\exp\{-q_k^{2+\varepsilon_1}\}\right), \;\;1\le j\le q_k-1,
\end{equation}
and
\begin{equation}\label{eq-1022h-18-11}
I_{k,0}=(0,\, \frac12e^{-q_k^{2+\varepsilon_1}}),\quad I_{k,q_k}=\left(1-\frac12e^{-q_k^{2+\varepsilon_1}}, \,1\right),
\end{equation}
then for each $k\in\N^+$,
 there exist intervals $I_{k,j}$ such that
  \begin{equation}\label{eq-1021hh-10}
\bigcup_{j'=1}^{J'_{k}}\left(I_{k,j'}\setminus\mathbb{Q}\right) \subset	E_{k}\subset\bigcup_{j=0}^{J_{k}}I_{k,j},
  	 \end{equation}
where $J'_{k}$ and $J_{k}$ satisfy
    \begin{equation}\label{eq-1017-10}
  2^{-k} q_{1}q_{2}\cdots q_{k}\exp\left\{-\sum_{l=1}^{k-1}q_{l}^{2+\varepsilon_{1}}\right\} < J'_{k}\quad \mbox{and}\quad J_{k}\le 2^k q_{1}q_{2}\cdots q_{k}\exp\left\{-\sum_{l=1}^{k-1}q_{l}^{2+\varepsilon_{1}}\right\},
    \end{equation} }
 where we adopt the convention that $\sum_{l=1}^{0}q_{l}^{2+\varepsilon_{1}}=0$.

  We prove \underline{Claim 1}
  by induction with respect to $k$. For $k=1$, we note that $\Vert q_1x\Vert<\frac{q_1}{2}\exp\{-q_1^{2+\varepsilon_1}\}$  if and only if there exists an integer $j\in \{0, 1,\cdots,q_1\}$ such that
  $$
  |x-\frac{j}{q_1}|<\frac12\exp\{-q_1^{2+\varepsilon_1}\}.
  $$
  In addition, note that $E_{1}\subset(0,1)\setminus\mathbb{Q}$. Thus,  we have
  that $E_{1}=\bigcup_{j=0}^{q_{1}}(I_{1,j}\setminus\mathbb{Q})$. This implies \eqref{eq-1021hh-10} with $J_1'=q_1-1$ and $J_1=q_1$. Since $q_1=4$, we have $J_1'\ge \frac{q_1}{2}$ and $J_1\le 2q_1$.
  So \underline{Claim 1} holds for $k=1$. We inductively suppose that
  \eqref{eq-1021hh-10} and \eqref{eq-1017-10} hold for $k=n-1$ ($n\ge 2$).  We aim to show that they hold for  $k=n$. To this end, we observe two facts as follows:

 \noindent  $(i)$ For each $I_{n-1,j}$, $1\le j\le q_{n-1}-1$, we have
 $|I_{n-1,j}|=\exp\{-q_{n-1}^{2+\varepsilon_1}\}$;

\noindent  $(ii)$ The distance between the midpoints of any two adjacent intervals $I_{n, j}$ and $I_{n, j+1}$ is $\frac{1}{q_{n}}$.

  From these facts, we see that  for each $j$,  $I_{n-1,j}$ consists at least
  $$
  q_{n}\exp\{-q_{n-1}^{2+\varepsilon_1}\}-2\ge\frac{q_{n}}{2}\exp\{-q_{n-1}^{2+\varepsilon_1}\},
  $$
  and at most
$$
q_{n}\exp\{-q_{n-1}^{2+\varepsilon_1}\}+2\le 2q_{n}\exp\{-q_{n-1}^{2+\varepsilon_1}\}
$$
  intervals in   $I_{n, j}$. Therefore we have
  $$
  \frac12 J_{n-1}q_{n}\exp\{-q_{n-1}^{2+\varepsilon_1}\}\le J'_{n}\quad \mbox{and}\quad J_{n}\le 2J_{n-1}q_{n}\exp\{-q_{n-1}^{2+\varepsilon_1}\}.
  $$
  These, together with the inductive hypothesis, yield  \eqref{eq-1021hh-10} and \eqref{eq-1017-10}, with $k=n$. Hence, \underline{Claim 1} is true.

  Next, for each $k\in\N^+$ and each  $I_{k, j}\subset E_k$, two facts are clear: First, $I_{k, j}$ is contained in one of intervals $I_{k-1,l}\subset E_{k-1}$
  and contains
  a finite number of intervals $I_{k+1,i}\subset E_{k+1}$; Second, $\max\{|I_{k,l}| :  I_{k,l}\subset E_k\}$ tends to $0$ as $k$ goes to $\infty$.
  From these facts, we can  construct a mass distribution $\mu$ on $E_{\infty}$,
  through  repeated subdivisions so that each of the $k$th level intervals of length $\exp\{-q_k^{2+\varepsilon_1}\}$ in $E_k$ carries the mass $\frac{1}{J_k}$ (see \cite[Proposition 1.7]{Falconer90}). Moreover, by \eqref{eq-1017-10}, we have
  \begin{equation}\label{eq-1022h-11}
  	\mu(I_{k,j})=\frac{1}{J'_k}\le \frac{2^k\exp\{q_{1}^{2+\varepsilon_{1}}+\cdots+q_{k-1}^{2+\varepsilon_{1}}\}}{q_{1}q_{2}\cdots q_{k}}.
  \end{equation}

Now, we arbitrarily fix  $r>0$ and $x\in \R$. Then there is $k$ such that
    \begin{equation}\label{eq-1017-11}
     e^{-q_{k}^{2+\varepsilon_{1}}} \leq r < e^{-q_{k-1}^{2+\varepsilon_{1}}}.
    \end{equation}
    We will prove \eqref{eq-1017-09} by the following two cases.

\emph{Case 1:} $e^{-q_{k}^{2+\varepsilon_{1}}} \leq r < \frac{1}{q_{k}}$.

We recall that the distance between the midpoints of any two adjacent intervals  of $E_k$  is $\frac{1}{q_{k}}$. Thus, $B(x,r)$ can intersect at most $4$ intervals in $E_k$. Consequently, we have
    \begin{equation}\label{eq-1017-12}
        \mu(B(x,r)) \leq 4\mu(I_{k,j})\le 4\frac{2^ke^{q_{1}^{2+\varepsilon_1}+\cdots+q_{k-1}^{2+\varepsilon_1}}}{q_1\cdots q_k} \leq 4(\frac{1}{q_k})^{1-\varepsilon_2},
    \end{equation}
where in the second inequality, we used \eqref{eq-1022h-11}; while in the last inequality, we used the inequality $2e^{q^{2+\varepsilon_1}_{j}}\le q_{j+1}$ for $1\le j\le k-2$ the and the assumption \eqref{eq-1017-02} for the term $e^{q_{k-1}^{2+\varepsilon_1}}$.
However, it is clear that
    \begin{equation}\label{eq-1017-13}
        e^{-q_{k}^{2+\varepsilon_{1}}} \leq r \Longleftrightarrow \frac{1}{q_k}<(\log\frac{1}{r})^{-\frac{1}{2+\varepsilon_1}}.
    \end{equation}
So by \eqref{eq-1017-12}, \eqref{eq-1017-13}, the assumption $e^{-q_{k}^{2+\varepsilon_{1}}} \leq r$, and  \eqref{eq-1017-02}, we obtain \eqref{eq-1017-09}  for \emph{Case 1}.

\emph{Case 2:}  $\frac{1}{q_k} \leq r<e^{-q_{k-1}^{2+\varepsilon_{1}}}$.

In this case,   $B(x,r)$ can intersect  at most
 $\frac{r}{1/q_k}=r q_k$
  intervals in $E_k$.
Thus
   \begin{align}\label{eq-1017-14}
     \mu(B(x,r)) \leq& \min\{\mu(I_{n,j}),\,\;rq_k\cdot \mu(I_{k,j})\} \notag \\
      =&\frac{e^{q_{1}^{2+\varepsilon_1}+\cdots+q_{k-2}^{2+\varepsilon_1}}}{q_1\cdots q_{k-2}} \min\{\frac{1}{q_{k-1}},\frac{re^{q_{k-1}^{2+\varepsilon_1}}}{q_{k-1}}\}\notag \\
      \leq &(\frac{1}{q_{k-1}})^{1-\varepsilon_2}\leq C(N,\varepsilon_1,\varepsilon_2)(\log q_k)^{-\frac{1}{2+\varepsilon}},
    \end{align}
   where  $C(N,\varepsilon_1,\varepsilon_2)=N^{\frac{1-\varepsilon_2}{2+\varepsilon_1}}$. In the  second inequality above, we used the fact $re^{q_{k-1}^{2+\varepsilon_1}}<1$ and \eqref{eq-1017-02}; and in the third inequality, we used \eqref{eq-1017-01} and \eqref{eq-1017-02}.
Meanwhile, it is clear that
    \begin{equation}\label{eq-1017-15}
        \frac{1}{q_k} \leq r\Longleftrightarrow (\log q_k)^{-\frac{1}{2+\varepsilon}}\le (\log \frac{1}{r})^{-\frac{1}{2+\varepsilon}}.
    \end{equation}
By  \eqref{eq-1017-14}, \eqref{eq-1017-15} and the assumption $\frac{1}{q_k} \leq r$, we obtain \eqref{eq-1017-09} for \emph{Case 2}.
 Hence, we have proved that $0<c_{f_\varepsilon}(E_\infty)<\infty$.

\vskip 5pt
\noindent{\it Step 4. We prove that $c_{h_\alpha}(E_{\infty})=0$.}
\vskip 5pt

First, it follows from  \eqref{eq-1022h-18} that for each $k$,
\begin{align}\label{4.66-w-10-25}
f_0(|I_{k,j}|)=q_k^{-(1+\frac{\varepsilon_1}{2})},\;\;1\leq j\leq q_k.
\end{align}
where $f_0$ is given by \eqref{4.47-w-g-10-25} with $\varepsilon=0$.
Second, it follows from \eqref{eq-1021hh-10} and \eqref{eq-1017-10} that
\begin{align}\label{eq-1022h-16}
	c_{f_0}(E_{\infty})\le \sum_{j=0}^{J_k}{f_0(|I_{k,j}|)}&\le
	\left(\frac{2^kq_{1}q_{2}\cdots q_{k}}{\exp\{q_{1}^{2+\varepsilon_{1}}+\cdots+q_{k-1}^{2+\varepsilon_{1}}\}}+1\right)q_k^{-(1+\frac{\varepsilon_1}{2})}\nonumber\\
	&\le 3 q_k^{-\frac{\varepsilon_1}{2}}\longrightarrow 0,\quad \mbox{as}\,\,k\to \infty,
\end{align}
where in the last inequality, we used the fact  $2q_{j}\le e^{q^{2+\varepsilon_1}_{j}}$.
Third, it is clear that
$$
h_\alpha(t)=(\log \frac{1}{t})^{-\frac{1}{2}}(\log\log\frac{1}{t})^{-\alpha}<(\log \frac{1}{t})^{-\frac{1}{2}}= f_0(t),\quad 0<t\leq e^{-1}.
$$
With  \eqref{4.66-w-10-25} and \eqref{eq-1022h-16}, the above implies  $c_{h_\alpha}(E_{\infty})=0$. This  completes the proof of Proposition \ref{thm-1017-1.1}.
\end{proof}

\begin{remark}\label{remark4.10-w-10-21}
By the H\"older inequality,  we have the following  direct consequence of Theorem \ref{thm-ob-bound}:
 \begin{align}\label{3.32-w-g-10-10}
 \|u(T,\cdot)\|^2_{L^\infty(0,L)}\leq c_{obs}\int_0^T\sup_{x\in E}|u(t,x)|^2\d t.
\end{align}
\end{remark}

\begin{proof}[Proof of Theorem \ref{thm-ob-R}]

 Given the similarity to  the proof of  Theorem \ref{thm-ob-bound}, we omit the details and  focus on the key distinctions.

  For each $\lambda>0$, we let
 $$
 \mathcal {E}_\lambda :=\big\{ f\in L^2(\R)\; : \; \mbox{supp}\; \widehat{f}\subset [-\sqrt{\lambda},\sqrt{\lambda}]\big\}.
 $$
By Lemma \ref{lem-spec}, we see that if $E\subset \R$ is $h_\alpha$-$(\gamma,L)$ thick, then there exists a constant $C>0$ depending only on $\alpha,\gamma,L$ such that when  $\lambda\gg 1$,
\begin{align}\label{equ-1024-1}
\|u\|_{l^2L^\infty}\leq e^{C\lambda(\log \lambda)^{-\frac{2\alpha}{3}}}
\Big\|\big\{\sup_{E\cap [k,k+L]}|u|\big\}_{k\in\Z}\Big\|_{l^2(\Z)}
, \quad \forall u\in \mathcal{E}_\lambda.
\end{align}

Let $A=\partial_x^2$ be the operator with domain $D(A)=H^2(\R)$. Then $A$ generates a strongly continuous semigroup in $L^2(\R)$ such that
\begin{align}\label{equ-1024-2}
   \|e^{tA}\phi\|_{L^2(\R)}\leq \|\phi\|_{L^2(\R)}, \quad \forall t\geq 0, \phi\in L^2(\R).
\end{align}
Moreover, we have the $L^2(\R)\to L^\infty(\R)$ estimate:
\begin{align}\label{equ-1024-3}
    \|e^{tA}\|_{L^2(\R)\to L^\infty(\R)}\leq (4\pi t)^{-\frac{1}{4}}, \quad \forall t>0.
\end{align}
Let $\mathcal{E}^{\perp}_\lambda=:\big\{ f\in L^2(\R)\; : \; \mbox{supp}\; \widehat{f}\subset \R \backslash [-\sqrt{\lambda},\sqrt{\lambda}]\big\}$ be the orthogonal complement of $\mathcal{E}^{\perp}_\lambda$ in $L^2(\R)$. Then we have the decay estimate:
\begin{align}\label{equ-1024-4}
 \|e^{tA}\phi\|_{L^2(\R)}\leq e^{-t\lambda}\|\phi\|_{L^2(\R)}, \quad \forall \phi\in \mathcal{E}^{\perp}_\lambda.
\end{align}
Now, using \eqref{equ-1024-1} in place  of \eqref{equ-1022-1} from the proof of Theorem \ref{thm-ob-bound},  along with \eqref{equ-1024-2}, \eqref{equ-1024-3} and \eqref{equ-1024-4}, we obtain \eqref{2.5-w-10-26}. This
completes the proof of Theorem \ref{thm-ob-R}.





\end{proof}

\section{Extension to higher dimensions}\label{section-high-d}
We prove  Theorem \ref{thm-bound-high}/Theorem \ref{thm-Rn-high} through extending the spectral inequality/the uncertain principle to the case of  $d$-dim. This extension   is  grounded in  the generalization of Lemma \ref{lem-ana} to  higher dimensions.
For this purpose, we apply the corresponding $1$-dim results together with the slicing theory, which is analogous to the approach in Fubini's theorem.
The slicing theorem provides information on the size of  intersections of (fractal) sets in $\R^d$ with planes. One of the most powerful tools in addressing the  slicing is the concept of capacity, which is another important  set function in geometric measure theory.

\subsection{Capacities and slicing}\

We start with the following definition.
\begin{definition} ($K$-energy and $K$-capacity)
Let $K$ be a nonnegative upper semicontinuous function on $\R^d\times\R^d$.

$(i)$ Given a Radon measure  $\mu$  on $\R^d$, the $K$-energy of $\mu$ is defined by
    $$
    I_K(\mu):=\int_{\R^d} \int_{\R^d} K(x,y)  \d \mu (x)\d \mu (y).
    $$

$(ii)$ Given a bounded Borel  set $E\subset \R^d$, the $K$-capacity of $E$ is defined by
$$
C_K(E):=\sup\left\{I_K(\mu)^{-1}:\mu \mbox{ is a Radon measure in }\;\R^d\;\mbox{s.t.} \; \mbox{supp}\;\mu\subset E, \mu(\R^d)=1\right\},
$$
with the convention  that $C_K(\emptyset)=0$.
\end{definition}
\begin{remark}\label{remark4.2-w-10-8}
 If $K(x,y)=|x-y|^{-s}, s>0$, then $C_K$ corresponds  to the classical Riesz $s$-capacity.
   We refer readers to \cite{Landkof1972found} for further details on  capacities.
\end{remark}
\begin{definition}\label{definition4.3-w-10-8}(Grassmannian manifold and Radon probability measure)

$(i)$ The Grassmannian manifold $G(d,m)$ (with $1\leq m<d$) consists of all $m$-dimensional linear subspaces of $\R^d$. If $V\subset G(d,m)$, we use  $V^\bot\in G(d,d-m)$ to denote
the orthogonal complement of $V$, and  write $V_a:=V+a$ (with $a\in V^\bot$) for  the $m$-dim plane through $a$ and parallel to $V$.

$(ii)$ Given $V\in G(d,m)$, the Radon probability  measure $\gamma_{d,m}$ on $G(d,m)$ is defined by
$$
\gamma_{d,m}(A):=\theta_d(\{g\in O(d)\;:\;gV\in A\}) \quad \mbox{ for } A \subset G(d,m),
$$
    where $O(d)$ is the orthogonal group consisting all orthogonal transformations on $\R^d$,  $\theta_d$ is
    the Haar measure on  $O(d)$ such that  $\theta_d(O(d))=1$.
\end{definition}
\begin{remark}\label{remark4.4-w-10-8} We give two additional remarks below, see \cite{Mattila1995geometry} for more details.

$(i)$ The measure $\gamma_{d,m}$ is independent of the choice of $V$.

$(ii)$ When $m=1$, the measure $\gamma_{d,1}$ reduces to the surface measure $\sigma^{d-1}$ on $S^{d-1}$ as follows:
$$
\gamma_{d,1}(A)=\sigma^{d-1}\left( \bigcup_{L\in A} L\cap S^{d-1} \right), \quad A\subset G(d,1).
$$
    \end{remark}
The following result on the capacity will be used later.

\begin{proposition}[Mattila \cite{Mattila1081integral}]\label{prop-Mattila}
Let $d>m\geq 1$.
Assume  that for some constant $b>0$,
$$
{K(x,y)}\geq b|x-y|^{m-d},\; (x,y)\in \R^d\times\R^d, |x-y|\leq 1.
$$
Let $H$ be  another lower semicontinuous kernel defined by
\begin{align*}
{H(x,y)}&={K(x,y)}|x-y|^{n-m},\; (x,y)\in \R^d\times\R^d, x\neq y,\\
{H(x,x)}&=\liminf_{(y,z)\to (x,x)\atop
 y\neq z} {H(y,z)},\; x\in\R^d.
\end{align*}
 Then there exists a constant $c>0$, depending only on $d$ and $m$, such that for any compact set $E\subset \R^d$,
 \begin{align}\label{4.1-w-10-8}
 C_K(E)\leq c\int \int_{V^\bot} C_{H}(E\cap V_a) \d \mathcal{H}^{d-m}a\d \gamma_{d,m}V.
 \end{align}
\end{proposition}

\begin{remark}
    $(i)$ In \eqref{4.1-w-10-8}, the notation ``$\d \mathcal{H}^{d-m}a$" means the integral along the variation of $a\in V^\bot$; ``$\d \gamma_{d,m}V$" means the integral along with the variation
    of $V\in G(d,m)$; $E\cap V_a$ can be understood as the slice of $E$ by $V_a$.

    $(ii)$ In the subsequent analysis, we shall focus on a particular class of kernels:
$$
K(x,y)=\hat{K}(|x-y|), \quad x,y\in \R^d,
$$
where $\hat{K} $ is a nonnegative upper semicontinuous function on $\R^+$.
We will  use $C_{\hat {K}}$ to denote $C_K$ when there is no ambiguity.
\end{remark}

We introduce two specific kernels which will be used in our analysis:
\begin{align}\label{4.2-w-g-s-10-8}
    K(x,y)=1/F_{\alpha,\beta}(|x-y|)\;\;\mbox{and}\;\;H(x,y)=1/h_{\alpha,\beta}(|x-y|),
\end{align}\label{4.2-w-g-10-8}
where $F_{\alpha,\beta}$
is given by
\eqref{equ-intro-F}, while
\begin{align}\label{equ-h}
    h_{\alpha,\beta}(t):=(\log \frac{1}{t})^{-\beta}(\log\log\frac{1}{t})^{-\alpha}, \quad 0<t\leq e^{-3}.
\end{align}
We note that the above $H$ and $K$ are defined for $(x,y)\in\R^d\times\R^d$ such that $0<|x-y|\leq e^{-3}$. But they should be treated as nonegative and semicontinuous extensions of $K$ and $H$ over
$\R^d\times\R^d$.

By applying Proposition \ref{prop-Mattila}, with $K$ and $H$ given by \eqref{4.2-w-g-s-10-8}
and $m=1$, we have the following consequence:

\begin{corollary}\label{cor-11-1}
 There is a constant $c>0$, depending only on $d$, such that for any $\alpha>0$ and $\beta>0$, and for each compact set $E\subset\R^d$,
 \begin{align}\label{equ-slicing-1}
  C_{F^{-1}_{\alpha,\beta}}(E)\leq c\int \int_{l^\bot}C_{h^{-1}_{\alpha,\beta}}(E\cap l_a) \d \mathcal{H}^{d-1}a\d \gamma_{d,1}l,
  \end{align}
   where $l$ are lines passing through the origin and $l_a:=l+a$.
\end{corollary}

The next lemma on the slicing plays an important role in our further studies.

\begin{lemma}\label{lem-slicing}
There is a line $l\subset \R^d$ (passing through the origin) and two positive constants $c$ and $k$ (depending only on $d$) such that for any $\alpha, \beta>0$, for each ball  $B_r\subset\R^d$ with the radius
 $r\in(0,2^{-1}e^{-3}]$ and each compact subset  $E\subset B_r$,
\begin{align}\label{4.5-w-g-10-8}
\mathcal{H}^{d-1}\{a\in l^\bot: C_{h^{-1}_{\alpha,\beta}}(E\cap l_a)\geq kr^{-(d-1)}C_{F^{-1}_{\alpha,\beta}}(E)\}\geq \frac{c C_{F^{-1}_{\alpha,\beta}}(E) }{h_{\alpha,\beta}(2r)}.
\end{align}
\end{lemma}
\begin{proof}
Let $\alpha,\beta>0$. We arbitrarily fix
$B_r$ with  $r\in(0,2^{-1}e^{-3}]$ and a compact set $E\subset B_r$.
Several fact are given as follows:

{\it Fact one:} According to Corollary \ref{cor-11-1}, there is a line $l\subset \R^d$ (passing through the origin)
such that
\begin{align}\label{equ-lem-slic-0}
 \int_{l^\bot} C_{h^{-1}_{\alpha,\beta}}(E\cap l_a) \d \mathcal{H}^{d-1}a\geq c_1 C_{F^{-1}_{\alpha,\beta}}(E),
\end{align}
where $c_1>0$ is a constant depending only on $d$.

{\it  Fact two:}  We have
\begin{align}\label{equ-lem-slic-2}
    C_{h^{-1}_{\alpha,\beta}}(E\cap l_a)\leq h_{\alpha,\beta}(2r):=(\log \frac{1}{2r})^{-\beta}(\log\log\frac{1}{2r})^{-\alpha}.
\end{align}
Indeed, to show \eqref{equ-lem-slic-2}, it suffices to prove the following two estimates:
\begin{align}\label{4.8-w-g-s-10-8}
    C_{h^{-1}_{\alpha,\beta}}(E\cap l_a)\leq C_{h^{-1}_{\alpha,\beta}}(B_r);
\end{align}
\begin{align}\label{4.9-w-g-s-10-8}
  C_{h^{-1}_{\alpha,\beta}}(B_r)\leq h_{\alpha,\beta}(2r).
\end{align}
The estimate \eqref{4.8-w-g-s-10-8} follows from the monotonicity of the capacity
and the fact that $E\subset B_r$. We now show \eqref{4.9-w-g-s-10-8}.
Indeed, since $C_{h^{-1}_{\alpha,\beta}}$ has the translation invariance, we only need to show
\eqref{4.9-w-g-s-10-8} for the ball $B_r$
  centered at the origin. Let $\mu$ be a Radon measure such that its supported is in $B_r$ and $\mu(B_r)=1$. Then  the $h^{-1}_{\alpha,\beta}$-engery of $\mu$ is
\begin{align}\label{equ-lem-slic-1}
I_{h^{-1}_{\alpha,\beta}}(\mu)=\int_{B_r}\int_{B_r} \frac{1}{h_{\alpha,\beta}(|x-y|)}\d \mu(x)\d \mu(y).
\end{align}
Since $|x-y|\leq 2r\leq e^{-3}$ when $x,y\in B_r$ and  $h_{\alpha,\beta}$ is a increasing function on $\R^+$, we have
$$
\frac{1}{h_{\alpha,\beta}(|x-y|)}\geq \frac{1}{h_{\alpha,\beta}(2r)}, \quad x,y\in B_r.
$$
With  \eqref{equ-lem-slic-1}, the above yields
$$
I_{h^{-1}_{\alpha,\beta}}(\mu)\geq \frac{1}{h_{\alpha,\beta}(2r)}(\mu(B_r))^2=\frac{1}{h_{\alpha,\beta}(2r)}.
$$
 Since $\mu$  was arbitrarily chosen, the above, along with the definition of the capacity, leads to
\eqref{4.9-w-g-s-10-8}.

{\it  Fact three:} By a similar way showing \eqref{equ-lem-slic-2}, we  obtain
\begin{align}\label{equ-lem-slic-3}
    C_{F^{-1}_{\alpha,\beta}}(E)\leq F_{\alpha,\beta}(2r)= (2r)^{d-1} h_{\alpha,\beta}(2r).
\end{align}

{\it  Fact four:} It is clear that  for some constant $c_2>0$, depending only on $d$,
\begin{align}\label{equ-lem-slic-4}
\mathcal{H}^{d-1}\{a\in l^\bot:E\cap l_a\neq \emptyset\}\leq \mathcal{H}^{d-1}\{a\in l^\bot:B_r\cap l_a\neq \emptyset\}\leq  c_2r^{d-1}.
\end{align}

Now for each $k>0$, we define the set
$$
E_k:=\{a\in l^\bot: C_{h^{-1}_{\alpha,\beta}}(E\cap l_a)\geq kr^{-(d-1)}C_{F^{-1}_{\alpha,\beta}}(E)\}.
$$
Then it follows that
for each $k>0$,
\begin{align}\label{4.13-w-10-8}
c_1 C_{F^{-1}_{\alpha,\beta}}(E)&\leq    \int_{l^\bot} C_{h^{-1}_{\alpha,\beta}}(E\cap l_a) \d \mathcal{H}^{d-1}a\nonumber\\
&\leq \int_{E_k} C_{h^{-1}_{\alpha,\beta}}(E\cap l_a) \d \mathcal{H}^{d-1}a + \int_{l^\bot \backslash E_k} C_{h^{-1}_{\alpha,\beta}}(E\cap l_a) \d \mathcal{H}^{d-1}a\nonumber\\
&\leq \mathcal{H}^{d-1}(E_k)\sup_{a\in l^\bot}C_{h^{-1}_{\alpha,\beta}}(E\cap l_a) + kr^{-(d-1)}C_{F^{-1}_{\alpha,\beta}}(E) \mathcal{H}^{d-1}\{a\in l^\bot\backslash E_k: E\cap l_a\neq \emptyset\}\nonumber\\
&\leq \mathcal{H}^{d-1}(E_k)h_{\alpha,\beta}(2r)+kr^{-(d-1)}C_{F^{-1}_{\alpha,\beta}}(E)\Big(c_2r^{d-1}-\mathcal{H}^{d-1}(E_k)\Big)\nonumber\\
&\leq kc_2C_{F^{-1}_{\alpha,\beta}}(E)+\Big(h_{\alpha,\beta}(2r)- kr^{-(d-1)}C_{F^{-1}_{\alpha,\beta}}(E) \Big)\mathcal{H}^{d-1}(E_k),
\end{align}
where in the first inequality above, we used \eqref{equ-lem-slic-0}, in the third inequality, we used the definition of $E_k$, and in the fourth inequality, we applied facts \eqref{equ-lem-slic-2} and \eqref{equ-lem-slic-4}.
However, it follows from  \eqref{equ-lem-slic-3} that
$$
h_{\alpha,\beta}(2r)- kr^{-(d-1)}C_{F^{-1}_{\alpha,\beta}}(E)\geq h_{\alpha,\beta}(2r)(1-k2^{d-1})\geq \frac{1}{2}h_{\alpha,\beta}(2r)>0,\;\;\mbox{when}\;\;k\leq 2^{-d},
$$
while it is clear that
$$
kc_2C_{F^{-1}_{\alpha,\beta}}(E)\leq \frac{1}{2}c_1 C_{F^{-1}_{\alpha,\beta}}(E)\;\;\mbox{when}\;\;k\leq \frac{c_1}{2c_2}.
$$
Then the above two estimates, along with  \eqref{4.13-w-10-8} where  $k:=\min\{2^{-d}, \frac{c_1}{2c_2}\}$, lead to
\eqref{4.5-w-g-10-8}. This completes the proof.
\end{proof}

\subsection{Quantitative relationship between Hausdorff contents and capacities}\

Our further studies on the propagation of smallness  for analytic functions
in the $d$-dim case are based on Lemma \ref{lem-ana}
and Lemma \ref{lem-slicing}. Lemma \ref{lem-ana} addresses  the $1$-dim case  at the scale of  log-type Hausdorff content, while Lemma \ref{lem-slicing} is a slicing theorem based on the capacity. Thus, we need to build up a quantitative connection between the two scales above. That is the aim of this subsection. It deserves mentioning that a qualitative connection between them has been established (see \cite[p.234]{Kametani1945on} or \cite[p.28]{Carleson1967selected}) and  we  adopt ideas from these references.

We begin by revisiting two established  lemmas whose proofs can be found in \cite[pp.19-20]{Carleson1967selected} and  \cite[pp.7-8]{Carleson1967selected} or  \cite[pp.234-236]{Kametani1945on}, respectively.

\begin{lemma}\label{lem-equi-cap}
 Let $E\subset\R^d$ be a bounded Borel set and $K$ be a nonnegative upper semicontinuous function on $(0,\infty)$. Let
  $\Gamma_E$ be a collection of Radon measures $\mu$, with $\mbox{supp} \mu\subset E$, such that
\begin{align}\label{4.14-w-10-12}
 u_\mu(x):=\int_{\R^d}  {K(|x-y|)}\d \mu(y)\leq 1, \;\;x\in E.
\end{align}
 Then the capacity of $E$ satisfies
 $$
 C_K(E)=\sup_{\mu\in \Gamma_E}\mu(E).
 $$
\end{lemma}

\begin{lemma}\label{lem-equi-content}
Let $h$ be a gauge function. There exist positive constants $A_1$ and $A_2$, depending only on $d$, such that for every compact set $F\subset \R^d$, there is a Radon measure $\mu$ satisfying the following:

($i$)  $\mbox{supp} \mu\subset E$, and
\begin{align*}
\mu(F)\geq A_1c_h(F).
\end{align*}

($ii$)
For every ball $B\subset \R^d$  with diameter $d(B)\leq r$,
$$
\mu(B\cap F)\leq A_2h(r).
$$
\end{lemma}

The main result of this subsection is as follows:

\begin{lemma}\label{lem-car}
The following conclusions are true:
\begin{itemize}
    \item [(i)] Let $1/K$ be a gauge function. Write $h:=1/K$. Then  for every bounded Borel set $E\subset \R^d$,
   \begin{align}\label{4.15-w-10-12}
    c_h(E)\geq \frac{1}{2}C_K(E).
   \end{align}
      \item [(ii)]
      Let $h$ and $1/K$ be two gauge functions. Assume that $h$ and $K$ are absolutely continuous such that
\begin{align}\label{4.16-w-10-12}
 \int^{\hat{d}}_0 K(t)\d h(t)<\infty
  \end{align}
  holds for some $\hat{d}>0$.
      Then there exists a constant $A>0$, depending only on $d$, $\hat d$, $h$ and $K$, such that when  $E\subset\R^d$ is a compact set, with $d(E)\leq\hat{d}$ (where $d(E):=\sup\{|x-y|\;:\;x, y\in E$),
 $$
C_K(E)\geq Ac_h(E).
$$
\end{itemize}
\end{lemma}
\begin{proof}
 (i) We arbitrarily fix a bounded Borel set $E$  in $\R^d$.
 Without loss  of generality, we can assume that
  $C_K(E)>0$.
  According to  Lemma \ref{lem-equi-cap}, there exists a Radon measure $\mu$, with $\mbox{supp} \mu\subset E$, such that
\begin{align}\label{equ-relate-1}
  u_\mu(x):=\int_{\R^d} K(|x-y|)\d \mu(y)\leq 1, \quad x\in E;
\end{align}
\begin{align}\label{equ-relate-2}
\mu(E)\geq \frac{1}{2}C_K(E).
\end{align}
We arbitrarily take open balls  $B_j$, $j=1,2,\dots$ such that  $E\subset \cup B_j$. Then by \eqref{equ-relate-1}, and noting
that $K$ is a decreasing function, we have
\begin{align*}
1\geq \int_{B_j} K(|x-y|)\d \mu(y)\geq  {K(d(B_j))}\mu(B_j),\;\;j=1,2,\dots
\end{align*}
which yields
\begin{align*}
1/{K(d(B_j))}\geq \mu(B_j),\;\;j=1,2,\dots.
\end{align*}
Taking the sum for all $j$ in the above,
 using \eqref{equ-relate-2}, we find
$$
\sum_{j}1/{K(d(B_j))}\geq \sum_j \mu(B_j)\geq \mu(E)\geq \frac{1}{2}C_K(E).
$$
With the definition of $c_h$ (with $h=1/K$), the above leads to
\eqref{4.15-w-10-12}.

(ii) We arbitrarily fix $\hat{d}>0$ and a compact set $E\subset\R^d$, with $d(E)\leq\hat{d}$. Without loss of generality, we can assume that
 $c_h(E)>0$. Then by
 Lemma \ref{lem-equi-content}, we can find  a Radon measure $\mu$, with $\mbox{supp } \mu\subset E$,  such that
 there are two positive constants  $A_1$ and $A_2$, depending only on  $d$, such that
\begin{align}\label{equ-relate-3}
    \mu(E)\geq A_1c_h(E)
\end{align}
and such that
\begin{align}\label{equ-relate-4}
\mu(B\cap E)\leq A_2 h(r)\;\;\mbox{for each ball}\; B\subset\R^d,\;\mbox{with}\; d(B)\leq r.
\end{align}
Let $x_0\in E$. We define
\begin{align}\label{4.21-w-10-12}
\phi(r)=\mu(\{x\in E\;:\;|x-x_0|<r\}).
\end{align}
From \eqref{4.21-w-10-12} and
\eqref{equ-relate-4}, we see
\begin{align}\label{4.22-w-10-12}
\phi(r)\leq A_2h(r).
\end{align}
Meanwhile, two facts are as follows: First,  since $K$ is strictly decreasing, we have
\begin{align}\label{4.23-w-10-12}
K'(r)\le 0
\end{align}
almost everywhere.
Second, we have
\begin{align}\label{5.24-w-f-10-22}
\lim_{r\rightarrow 0^+}\limits\Big((A_2h(r)-\phi(r))K(r)\Big)=0.
\end{align}
Indeed, by \eqref{4.22-w-10-12}, we see that to show  \eqref{5.24-w-f-10-22}, it suffices to prove
\begin{align}\label{equ-10-12-ming}
\lim_{r\rightarrow 0^+}\limits\Big( h(r) K(r)\Big)=0.
\end{align}
However, it follows from \eqref{4.16-w-10-12} that  for each $\varepsilon>0$, there exists $\delta>0$ such that when $\delta'\in(0,\delta)$,
$$
\varepsilon>\int_{\delta'}^\delta K(t)\d h(t)\geq K(\delta)\int_{\delta'}^\delta  \d h(t)= K(\delta)(h(\delta)-h(\delta'))>0.
$$
Letting $\delta'\to 0$ in the above inequality leads to  \eqref{equ-10-12-ming}.

Now, it follows from \eqref{4.21-w-10-12}, \eqref{4.22-w-10-12}, \eqref{4.23-w-10-12},
\eqref{5.24-w-f-10-22} and the absolute continuity of $h, K$ that
\begin{align}\label{4.24-w-10-12}
 u_\mu(x_0)&:=\int_{\R^d} K(|x_0-y|)\d \mu(y)=\int_0^{d(E)}K(r)\d \phi(r)\leq \int_0^{\hat d}K(r)\d \phi(r)     \nonumber\\
 &=K(\hat d)\phi(\hat d)-\lim_{r\rightarrow 0^+}(\phi(r)K(r))-\int_0^{\hat d}\phi(r)\d K(r)\nonumber\\
 &\leq K(\hat d)\phi(\hat d)-\lim_{r\rightarrow 0^+}(\phi(r)K(r))-A_2\int_0^{\hat d}h(r)\d K(r)\nonumber\\
 &=K(\hat d)(\phi(\hat d)-A_2h(\hat d)K(\hat d)
+ \lim_{r\rightarrow 0^+}\Big((A_2h(r)-\phi(r))K(r)\Big)
 +A_2\int_0^{\hat d}K(r)\d h(r)\nonumber\\
&=K(\hat d)(\phi(\hat d)-A_2h(\hat d)K(\hat d)
 +A_2\int_0^{\hat d}K(r)\d h(r):=M.
\end{align}
Clearly, the above constant $M$   depends only on $d,\hat{d}$, $h$ and $K$.

Let $\Tilde{\mu}=M^{-1}\mu$ (which is clearly a Radon measure supported on $E$).
By \eqref{4.24-w-10-12}, we see
$$
u_{\Tilde{\mu}}(x)\leq 1 \quad \mbox{ for all }x\; \in E.
$$
Thus we can apply Lemma \ref{lem-equi-cap} to $\Tilde{\mu}$ and use \eqref{equ-relate-3} to get
$$
C_K(E)\geq \Tilde{\mu}(E)\geq A_1M^{-1}c_h(E).
$$
This completes the proof.
\end{proof}

\begin{remark}
We mention that a qualitative version of Lemma \ref{lem-car} was proved for analytic sets, see \cite{Carleson1967selected}. Note that every Borel set is an analytic set, but
the converse is not true (see \cite{Dasgupta2014set}).
\end{remark}

\begin{remark}\label{rem-gap}
The connection between the $h$-Hausdorff content and the $K$-capacity in Lemma \ref{lem-car} is optimal in the following sense:
There is a bounded closed set $E$ such that $c_h(E)=0$ and $C_K(E)>0$ when $\liminf_{t\to 0^+}h(t)K(t)=0$. Conversely, there is a compact set $E$ such that $c_h(E)>0$ and $C_K(E)=0$, when $\int_0^1K(t)\d h(t)$ diverges (see \cite{Taylor1961on}).
From the above, we see that there is a ‘gap of uncertainty’ in deductions in either direction, and thus
these two scales  measuring sets are fundamentally different.
\end{remark}

The next lemma is a consequence of  Lemma \ref{lem-car} which  connects the Hausdorff contents and the capacity
induced by  the gauge function $F_{\alpha,\beta}$ given by \eqref{equ-intro-F}.

\begin{lemma}\label{lem-connect}
Let $F_{\alpha,\beta}$ (with $\alpha, \beta\geq 0$) given by \eqref{equ-intro-F}.
 Let $E\subset\R^d$ be a compact subset. Then the following assertions are true:
\begin{itemize}
    \item [(i)] If  $C_{F^{-1}_{\alpha,\beta}}(E)>0$, then
    $ c_{F_{\alpha,\beta}}(E)\geq \frac{1}{2}C_{F^{-1}_{\alpha,\beta}}>0.$
    \item [(ii)] Let  $\varepsilon>0$. If  $c_{F_{\alpha,\beta}}(E)>0$ for some  $\beta\geq 1$ and $\alpha\geq 1+\varepsilon$, then there exists a constant $A>0$, depending only on $d,\alpha,\beta,\varepsilon$ and  $d(E)$, such that
   \begin{align}\label{4.25-w-10-12}
    C_{F^{-1}_{\alpha-1-\varepsilon,\beta-1}}(E)\geq Ac_{F_{\alpha,\beta}}(E)>0.
   \end{align}
\end{itemize}
\end{lemma}
\begin{proof}
  The assertion $(i)$  is a direct consequence of
  the conclusion $(i)$ in Lemma \ref{lem-car}, since  $F_{\alpha,\beta}$ is a gauge function.

   We now show the assertion $(ii)$. Let $E\subset\R^d$ be a compact set such that $c_{F_{\alpha,\beta}}(E)>0$ with some $\beta \geq 1, \alpha\geq 1+\varepsilon$. The application of the conclusion $(ii)$ in  Lemma \ref{lem-connect} (with
   $h=F_{\alpha,\beta}$ and $K=1/F_{\alpha-1-\varepsilon,\beta-1}$) leads to \eqref{4.25-w-10-12},
   provided that
 \begin{align}\label{equ-connect-1}
 \int_0^{d(E)} K(t)h'(t)\d t<\infty.
 \end{align}
Thus, it remains to check  \eqref{equ-connect-1}.
Indeed, a  direct computation shows that
\begin{align*}
 h'(t) &=(d-1)t^{d-2}(\log \frac{1}{t})^{-\beta}(\log\log\frac{1}{t})^{-\alpha}+\beta t^{d-2} (\log \frac{1}{t})^{-\beta-1}(\log\log\frac{1}{t})^{-\alpha}\\
  &\quad +\alpha t^{d-2} (\log \frac{1}{t})^{-\beta-1}(\log\log\frac{1}{t})^{-\alpha-1}\\
  &=\mathcal{O}(t^{d-2}(\log \frac{1}{t})^{-\beta}(\log\log\frac{1}{t})^{-\alpha}), \quad 0<t<t_0
\end{align*}
with some $t_0$ small, which yields
\begin{align*}
 \int_0^{d(E)} K(t)h'(t)\d t\leq C+   C\int_0^{\min\{t_0,1/100\}}\frac{1}{t\log\frac{1}{t}(\log\log \frac{1}{t})^{1+\varepsilon}} \d t<\infty,
\end{align*}
i.e., \eqref{equ-connect-1} holds.
 Thus the proof is completed.
\end{proof}

\subsection{Propagation of smallness  for analytic functions in $d$-dim}\

In this subsection, we shall establish a propagation of smallness of analytic functions from sets of positive log-Hausdorff content in the $d$-dim case.
 In the sequel, we use $\mathbb{D}_r(a)$ to denote  the closed  disc in $\C^d$, centered at $a\in \C^d$ with radius $r>0$, and $B_r(a)$ to denote  the open ball in $\R^d$, centered at $a\in\R^d$ with radius $r$.
 When  the center $a$ is not important, we shall write $B_r$ instead of $B_r(a)$ for short.

 The following lemma is a $d$-dim version of  \eqref{2.43-w-10-3} (in Remark \ref{rem-Kov}). It will be used for our purpose.
\begin{lemma}\label{lem-ana-high-1}
Let $r>0$ and $B_{r/2}$ be a ball in $\R^d$ such that $0\in B_{r/2}$. Let $E\subset B_{r/2}$ be a subset of  positive Lebesgue measure $|E|$. Then there is a constant $C>0$, depending only on $d$,  such that for every analytic function $\phi:\mathbb{D}_{6r}(0)\subset\C^d\to \C$,
\begin{align}\label{4.28-w-10-14}
\sup_{B_{r/2}}|\phi|\leq \left(\frac{Cr^d}{|E|}\right)^{\frac{\log M}{\log 2}}\sup_{E}|\phi|,
\end{align}
where
\begin{align}\label{4.29-w-10-15}
M=\sup_{z\in \mathbb{D}_{5r}(0)}|\phi(z)|/\sup_{x\in  B_{r/2}}|\phi(x)|.
\end{align}
\end{lemma}
\begin{proof}
We arbitrarily fix $E\subset B_{r/2}$ with $|E|>0$, and an analytic function $\phi:\mathbb{D}_{6r}(0)\subset\C^d\to \C$.
Without loss of generality, we can assume that $r=1$. (For otherwise, we can replace
 $\phi(x)$ by $\phi(rx)$.)

Let $x_0\in \overline{B_{1/2}}$ be a point such that
\begin{align}\label{4.30-w-10-15}
|\phi(x_0)|=\max_{x\in \overline{B_{1/2}}}|\phi(x)|.
\end{align}

We claim that there is a line segment $l_0$ with $|l_0|=1$, passing through  $x_0$, such that
\begin{align}\label{4.31-w-10-15}
|E\cap l_0|\geq c|E|
\end{align}
 for some   constant $c>0$ depending only on $d$. (Here, $|E\cap l_0|$ and $|l_0|$ denote the $1$-dim Lebesgue measures of $|E\cap l_0|$ and $l_0$.)

 Indeed, using the polar coordinate centered at $x_0$, noting that    $|l\cap B_{1/2}(0)|\leq 1$ for each line  $l$ (in $\R^d$),
 we have
$$
|E|=\int \chi_E(x)\d x =\int^{1}_{0}\int_{|x-x_0|=\rho} \chi_E(x) \d \sigma\d  \rho=\int_0^{1}\rho^{d-1}\int_{\mathbb{S}^{d-1}}\chi_{E-x_0}(\rho\theta) \d \theta\d\rho,
$$
which yields
$$
|E|\leq \int_0^{1}\max_{\theta\in \mathbb{S}^{d-1}}\chi_{E-x_0}(\rho\theta) |\mathbb{S}^{d-1}|\d \rho.
$$
From the above, we can find a $\theta_0\in \mathbb{S}^{d-1}$ such that
$|(E-x_0)\cap\{\rho\theta_0: \rho\in[0,1]\}|\geq c|E|$.
Thus, the segment $l_0:=\{\rho\theta_0+x_0\;:\; \rho\in[0,1]\}$ is what we desired.

Next, by applying \eqref{2.43-w-10-3} (where $I$ and $E$ are replaced by $l_0$ and $E\cap l_0$), using the similar scaling argument to that used in the proof of Corollary \ref{cor-ana}, and noting \eqref{4.30-w-10-15} and \eqref{4.31-w-10-15}, we obtain
\begin{align}\label{4.32-w-10-14}
\sup_{l_0}|\phi|\leq \left(\frac{C}{|E\cap l_0|}\right)^{\frac{\log M_1}{\log 2}}\sup_{E\cap l_0}|\phi|,
\end{align}
where $M_1=\sup_{z\in \mathbb{D}_3(x_0)}|\phi(z)|/\sup_{x\in  l_0}|\phi(x)|$.

Meanwhile, since $0\in B_{1/2}$ and $x_0\in \overline{B_{1/2}}$,  we have $|x_0|\leq 2$, and thus $\mathbb{D}_3(x_0)\subset \mathbb{D}_5(0)$. Moreover, it folows from \eqref{4.30-w-10-15} that  $\sup_{l_0}|\phi|\geq \sup_{B_{1/2}}|\phi|$. Thus,
$$
M_1\leq \sup_{z\in \mathbb{D}_5(0)}|\phi(z)|/\sup_{B_{1/2}}|\phi|.
$$
%
This, along with \eqref{4.32-w-10-14} and the fact $\sup_{E\cap l_0}|\phi|\leq \sup_{E}|\phi|$, implies
\eqref{4.28-w-10-14}.

Thus, we complete the proof.
 \end{proof}

We will now present our result on the propagation of smallness for analytic functions, which can be viewed   as a $d$-dim extension of Corollary
\ref{cor-inter-820}. We recall that for each $a\in\R$, $a-$ denotes $a-\delta$ for $\delta>0$ small.
\begin{proposition}\label{prop-ana-high-2}
Let $B_1$ be a ball with radius $1$ such that $0\in B_1$. Let $E\subset B_1$ be a compact subset  with  $c_{F_{\alpha,\beta}}(E)>0$,  where $F_{\alpha,\beta}$
is given by \eqref{equ-intro-F} with $\alpha>1, \beta=3/2$.
Then there is a positive constant $C$, depending only on $d$, $\alpha$ and $c_{F_{\alpha,\beta}}(E)$, such that any analytic function
 $\phi:\mathbb{D}_{10}(0)\subset\C^d\to \C$ satisfies
\begin{align}\label{4.33-w-10-15}
\sup_{x\in B_{1}}|\phi(x)|\leq \varepsilon \sup_{z\in \mathbb{D}_{10}(0)}|\phi(z)|  +\exp\left\{\frac{C(\log \frac{1}{\varepsilon})^2 }{(\log\log \frac{1}{\varepsilon})^{2(\alpha-1-)/3}} \right\}\sup_{x\in E}|\phi(x)|
\end{align}
for each $\varepsilon>0$ small enough.
\end{proposition}
\begin{proof}
We arbitrarily fix   $r\in(0,2^{-1}e^{-3}]$ and $\alpha>1$. Since
 $c_{F_{\alpha,\beta}}(E)>0$,
 it follows from the definition of the content that  there exists a ball $B_r$ (of radius $r$) such that \footnote{We emphasize that the lower bound of $c_{F_{\alpha,\beta}}(\Tilde{E})$ depends only on $c_{F_{\alpha,\beta}}(E)>0$, and in particular is independent of the shape of $E$. This observation is important in the proof of the uncertainty principle in Proposition \ref{prop-LS-high}.} $c_{F_{\alpha,\beta}}(\Tilde{E})>0$, where $\Tilde{E}:=E\cap \overline{B_r}$. Clearly,  $\Tilde{E}$ is a compact set.

  From conclusion  $(ii)$ in Lemma \ref{lem-connect}, we deduce that
  \begin{equation}\label{eq-1105-hh}
      C_{F^{-1}_{\alpha-1-,\beta-1}}(\Tilde{E})\geq Ac_{F_{\alpha,\beta}}(\Tilde{E})>0,
  \end{equation}
where $A>0$ is a constant, depending only on $d,\alpha$ and $r$.
  Then according to Lemma \ref{lem-slicing}, there exists a line $l$ in $\R^d$ such that
for a positive constant $c$, depending only on $d$ and $h_{\alpha,\beta}$ given by \eqref{equ-h},
\begin{align}\label{equ-ana-high-1}
 \mathcal{H}^{d-1}(\Upsilon)\geq \frac{c C_{F^{-1}_{\alpha-1-,\beta-1}}(\Tilde{E}) }{h_{\alpha-1-,\beta-1}(2r)},
\end{align}
where
\begin{align}\label{4.35-w-10-15}
\Upsilon:=\{a\in l^\bot: C_{h^{-1}_{\alpha-1-,\beta-1}}(\Tilde{E}\cap l_a)\geq kr^{-(d-1)}C_{F^{-1}_{\alpha-1-,\beta-1}}(\Tilde{E})\}.
\end{align}
(Here we recall that $l_a:=l+a$.)

Meanwhile,  by  Lemma \ref{lem-connect} $(i)$, \eqref{4.35-w-10-15} and \eqref{eq-1105-hh}, we see that for each $a\in \Upsilon$,
\begin{align}\label{4.36-w-10-15}
c_{h_{\alpha-1-,\beta-1}}(\Tilde{E}\cap l_a)\geq \frac{1}{2}kr^{-(d-1)}C_{F^{-1}_{\alpha-1-,\beta-1}}(\Tilde{E})\geq  \frac{A}{2}kr^{-(d-1)}c_{F_{\alpha,\beta}}(\Tilde{E})>0.
\end{align}

Now, we arbitrarily fix an analytic function  $\phi:\mathbb{D}_{10}(0)\subset\C^d\to \C$  and $\varepsilon>0$ small enough.
Since  $\beta=3/2$ and $\alpha>1$ and because of  \eqref{4.36-w-10-15}, we can apply Corollary \ref{cor-inter-820} to obtain
a positive constant $C$, depending only on $\alpha$ and $c_{F_{\alpha,\beta}}(\Tilde{E})$, such that for each $a\in \Upsilon$,
\begin{align}\label{equ-111-1}
\sup_{x\in l_a\cap B_r}|\phi(x)|\leq \varepsilon \sup_{z\in D_5(a')}|\phi(z)|+ \exp\left\{C\frac{(\log \frac{1}{\varepsilon})^2 }{(\log\log \frac{1}{\varepsilon})^{2(\alpha-1-)/3}} \right\}\sup_{x\in \Tilde{E}\cap l_a}|\phi(x)|
\end{align}
holds for each $a'\in l_a\cap B_r$. Since $B_r\cap E\neq \emptyset$ and $E\subset B_1$, we find $|a'|\leq 3$, and then $D_5(a')\subset D_{10}(0)$.
This, together with \eqref{equ-111-1}, yields
\begin{align}\label{equ-ana-high-2}
\sup_{x\in \cup_{a\in \Upsilon}l_a\cap B_r}|\phi(x)|\leq \varepsilon \sup_{z\in  D_{10}(0)}|\phi(z)|+ \exp\left\{C\frac{(\log \frac{1}{\varepsilon})^2 }{(\log\log \frac{1}{\varepsilon})^{2(\alpha-1-)/3}} \right\}\sup_{x\in E}|\phi(x)|.
\end{align}

On the other hand, by \eqref{equ-ana-high-1}, we have
$$
|\cup_{a\in \Upsilon}l_a\cap B_r|\gtrsim r\frac{c C_{F^{-1}_{\alpha-1-,\beta-1}}(\Tilde{E}) }{h_{\alpha-1-,\beta-1}(2r)} >0.
$$
Thus by Lemma \ref{lem-ana-high-1}, there is a constant $C_1>0$, depending only on $d, r,\alpha, C_{F^{-1}_{\alpha-1-,\beta-1}}(\Tilde{E})$, such that
$$
\sup_{B_{1}}|\phi|\leq C_1^{\frac{\log M}{\log 2}}\sup_{\cup_{a\in \Upsilon}l_a\cap B_r}|\phi|,
$$
where $M=\sup_{z\in \mathbb{D}_{10}(0)}|\phi(z)|/\sup_{x\in  B_{1}}|\phi(x)|$.
This can be rewritten as
$$
\sup_{B_{1}}|\phi|\leq M^{\delta}\sup_{\cup_{a\in \Upsilon}l_a\cap B_r}|\phi|
$$
with $\delta=\frac{\log C_1}{\log 2}$. It follows that
$$
\sup_{B_{1}}|\phi|\leq  \left(\sup_{z\in \mathbb{D}_{10}(0)}|\phi(z)|\right)^{\frac{\delta}{\delta+1}}\left(\sup_{\cup_{a\in \Upsilon}l_a\cap B_r}|\phi|\right)^{\frac{1}{\delta+1}}.
$$
By Young's inequality, this implies that
\begin{align}\label{equ-ana-high-3}
\sup_{B_{1}}|\phi|\leq \varepsilon \sup_{z\in \mathbb{D}_{10}(0)}|\phi(z)| + \varepsilon^{-\delta}\sup_{\cup_{a\in \Upsilon}l_a\cap B_r}|\phi|
\end{align}
where $\delta=\frac{\log C_1}{\log 2}$.

Combining \eqref{equ-ana-high-2} and \eqref{equ-ana-high-3} leads to
$$
\sup_{B_{1}}|\phi|\leq 2\varepsilon \sup_{z\in \mathbb{D}_{10}(0)}|\phi(z)| + \varepsilon^{-\delta}\exp\left\{\frac{C(\log \frac{1}{\varepsilon^{\delta+1}})^2 }{(\log\log \frac{1}{\varepsilon^{\delta+1}})^{2(\alpha-1-)/3}} \right\}\sup_{x\in E}|\phi(x)|.
$$
Absorbing $\varepsilon^{-\delta}$ by the exponential term
in the above inequality leads to \eqref{4.33-w-10-15}. This completes the proof.
\end{proof}

\subsection{Spectral inequality and uncertainty principle in $d$-dim}

This subsection devotes to presenting the $d$-dim versions of Lemma \ref{lem-spec} and  Lemma \ref{lem-LS}. We start with the
spectral inequality.

Let $\Omega\subset\R^d$ be a bounded  domain with a smooth boundary $\partial\Omega$. Denote by
$\{\lambda_k\}_{k\geq 1}$ and $\{\phi_k\}_{k\geq 1}$ the eigenvalues and the corresponding normalized eigenfunctions of
 $-\Delta$ with the domain $H^2(\Omega)\cap H_0^1(\Omega)$, namely
\begin{align}\label{equ-spec-high-1}
   -\Delta \phi_k = \lambda_k\phi_k, \quad k\geq 1.
\end{align}
Given $\lambda>0$, we define the linear spectral subspace $ \mathcal {E}_\lambda :=\mbox{Span}\{\phi_k \;:\; \lambda_k\leq\lambda\}$.

\begin{proposition}\label{prop-spec-high}
 Let $E\subset\Omega$ be a compact subset with  $c_{F_{\alpha,\beta}}(E)>0$,  where $F_{\alpha,\beta}$
 is given by \eqref{equ-intro-F}, with $\alpha>1$ and $\beta= \frac{3}{2}$.  Then there exists a constant $C>0$, depending only on $d, \Omega, \alpha$ and $c_{F_{\alpha,\beta}}(E)$, such that when $\lambda>0$,
 \begin{align}\label{4.40-w-g-s-10-15}
 \sup_{\Omega}|\phi|\leq C\exp\left\{ C\lambda (\log (\lambda+e))^{-\frac{2(\alpha-1-)}{3}}\right\}\sup_{E}|\phi|\;\;\mbox{for all}\;\;\phi\in \mathcal {E}_\lambda.
 \end{align}
 \end{proposition}

\begin{proof}
We arbitrarily fix $\alpha>0$, $\lambda>0$ and $\phi\in \mathcal {E}_\lambda$.
 We write
 $$
 \phi(x) = \sum_{\lambda_k\leq \lambda} c_k\phi_k(x),\;\; x\in \Omega,
 $$
 where $c_k=\langle \phi,\phi_k \rangle_{L^2(\Omega)}$. Let
\begin{align}\label{equ-1126-1}
 \Phi(x,y)= \sum_{\lambda_k\leq \lambda} c_ke^{\sqrt{\lambda_k}y}\phi_k(x), \quad (x,y)\in \Omega\times \R.
\end{align}
Using \eqref{equ-spec-high-1} and \eqref{equ-1126-1}, one can check that
\begin{align}\label{equ-1126-2}
 \Delta_{x,y}\Phi= 0, \quad (x,y)\in \Omega\times \R.
\end{align}

Since  $E\subset\Omega$ is a compact subset with  $c_{F_{\alpha,\beta}}(E)>0$,  by the sub-additive property of the Hausdorff content, we can find $C_1>0$ and  an open  ball $B_{r}(x_0)$, such that $c_{F_{\alpha,\beta}}(B_r(x_0)\cap E)\geq C_1c_{F_{\alpha,\beta}}(E)>0$ and $B_{4r}(x_0)\subset \Omega$. Without loss of generality, we can assume that $x_0=0$.
(For otherwise, we can use the translation of coordinates.)
In light of \eqref{equ-1126-2}, we see that $\Phi$ is a harmonic function on $B_{4r}(0,0)\subset \Omega\times\R$.
According to the local analyticity of harmonic functions, see e.g. \cite[p.23, Theorem 2.10]{Trudinger}, we have
\begin{align}\label{equ-1126-3}
 \left\|\partial_x^\alpha \partial_y^\beta \Phi\right\|_{L^{\infty}\left(B_{r}(0,0)\right)} \leq
 \left( \frac{d(|\alpha|+\beta)}{r}\right)^{|\alpha|+\beta}\left\|\Phi\right\|_{L^{\infty}\left(B_{2r}(0,0)\right)}, \text { when } \alpha \in \mathbb{N}^d, \beta \geq 0.
\end{align}
  In particular, \eqref{equ-1126-3} implies that for all $\alpha\in \mathbb{N}^d$
\begin{align}\label{equ-1126-4}
 \left\|\partial_x^\alpha \phi\right\|_{L^{\infty}\left(B_{r}(0)\right)} \leq  \left( \frac{d|\alpha|}{r}\right)^{|\alpha|}\|\Phi\|_{L^\infty(B_{2r}(0,0))}\leq C_2\left( \frac{d|\alpha|}{r}\right)^{|\alpha|}\|\Phi\|_{L^2(B_{4r}(0,0))}
\end{align}
for some $C_2=C_2(d,r)>0$. Here in the last step we used the interior estimate of harmonic functions.
Meanwhile, since $B_{4r}(0)\subset \Omega$, it follows from  the orthonormality of $\{\phi_k: k\geq 1\}$ in $L^2(\Omega)$  and \eqref{equ-1126-1} that
\begin{align}\label{equ-1126-5}
\|\Phi\|_{L^2(B_{4r}(0,0))}\leq \|\Phi\|_{L^2(\Omega\times(-4r,4r))}\leq \sqrt{8r}e^{4r\sqrt{\lambda}}\|\phi\|_{L^2(\Omega)}\leq C_3e^{C_3\sqrt{\lambda}}\|\phi\|_{L^\infty(\Omega)}
\end{align}
with $C_3=C_3(r,|\Omega|)>0$. Combining \eqref{equ-1126-4}-\eqref{equ-1126-5}, we find that for some $C_4=C_4(d)>0$
\begin{align}\label{equ-1126-6}
 \left\|\partial_x^\alpha \phi\right\|_{L^{\infty}\left(B_{r}(0)\right)} \leq C_2C_3e^{C_3\sqrt{\lambda}}\frac{\alpha!C_4^{|\alpha|}}{r^{|\alpha|}}\|\phi\|_{L^\infty(\Omega)}, \text { when } \alpha \in \mathbb{N}^d.
\end{align}
Now, we introduce the scaling of $\phi$ as follows:
\begin{align}\label{equ-1126-7}
u(x)=\phi\Big(\frac{x}{K}\Big),\;\;x\in\Omega,
\end{align}
where $K=\max\{1, 20C_4/r\}$. Then by \eqref{equ-1126-7} and   \eqref{equ-1126-6}, we see that
$$
|\partial_x^\alpha u(0)|\leq C_2C_3e^{C_3\sqrt{\lambda}}\frac{\alpha!}{20^{|\alpha|}}\|\phi\|_{L^\infty(\Omega)}, \text { when } \alpha \in \mathbb{N}^d.
$$
With the Taylor formula for multi-variables, the above shows that $u$ is an analytic function on $\mathbb{D}_{10}(0)$ and satisfies
\begin{align}\label{equ-1126-8}
\sup_{z\in D_{10}(0)}|u(z)|\leq \sup_{z\in D_{10}(0)}\sum_{\mu\in \N^d}\frac{|\partial_x^\alpha u(0)||z|^{|\alpha|}}{\alpha!} \leq C_5e^{C_5\sqrt{\lambda}} \|\phi\|_{L^\infty(\Omega)},
\end{align}
where $C_5=C_5(d,r,|\Omega|)>0$.

Next, for each $\tau>0$, we let $\tau E:=\{\tau x: x\in E\}$. Since $K\geq 1$,  noting $\{KB_j\}$ is a covering of $KE\cap KB_R(0)$ if and only if $\{B_j\}$ is a covering of $E\cap B_r(0)$, using the definition of the Hausdorff content and the monotonicity of $F_{\alpha,\beta}$, we have
$$
c_{F_{\alpha,\beta}}(KE\cap KB_r(0))\geq c_{F_{\alpha,\beta}}(E\cap B_r(0))>0.
$$
Thus, we can apply Proposition \ref{prop-ana-high-2} (where $E$ and $\phi$ are replaced by $KE\cap KB_{r}(0)$ and  $u$) to
see  that when  $\varepsilon>0$ is small,
\begin{align}\label{equ-1126-9}
\sup_{x\in B_{1}(0)}|u(x)|\leq \varepsilon \sup_{z\in \mathbb{D}_{10}(0)}|u(z)|  +\exp\left\{\frac{C(\log \frac{1}{\varepsilon})^2 }{(\log\log \frac{1}{\varepsilon})^{2(\alpha-1-)/3}} \right\}\sup_{x\in KE\cap KB_{r}(0)}|u(x)|,
\end{align}
where $C>0$ depends only on $d,\alpha,c_{F_{\alpha,\beta}}(E)$.

Using \eqref{equ-1126-7} and \eqref{equ-1126-8}, we can rewrite \eqref{equ-1126-9} as
\begin{align}\label{equ-1126-10}
\sup_{x\in B_{1/K}(0)}|\phi(x)|&\leq \varepsilon \sup_{z\in \mathbb{D}_{10}(0)}|u(z)|  +\exp\left\{\frac{C(\log \frac{1}{\varepsilon})^2 }{(\log\log \frac{1}{\varepsilon})^{2(\alpha-1-)/3}} \right\}\sup_{x\in E\cap B_{r}(0)}|\phi(x)|\nonumber\\
&\leq \varepsilon C_5e^{C_5\sqrt{\lambda}}\|\phi\|_{L^\infty(\Omega)} +\exp\left\{\frac{C(\log \frac{1}{\varepsilon})^2 }{(\log\log \frac{1}{\varepsilon})^{2(\alpha-1-)/3}} \right\}\sup_{x\in E}|\phi(x)|.
\end{align}

Meanwhile, by the classical spectral inequality, we have
\begin{align}\label{equ-1126-11}
 \sup_{\Omega}|\phi|\leq e^{C_6(\sqrt{\lambda}+1)}    \sup_{B_{1/K}(0)}|\phi|.
\end{align}
Combining \eqref{equ-1126-10} and \eqref{equ-1126-11} leads to
\begin{align}\label{equ-1126-12}
\sup_{\Omega}|\phi|\leq   \varepsilon e^{C_6(\sqrt{\lambda}+1)}   C_5e^{C_5\sqrt{\lambda} } \|\phi\|_{L^\infty(\Omega)}  +e^{C_6(\sqrt{\lambda}+1)}  \exp\left\{\frac{C(\log \frac{1}{\varepsilon})^2 }{(\log\log \frac{1}{\varepsilon})^{2(\alpha-1-)/3}} \right\}\sup_{x\in E}|\phi(x)|.
\end{align}

Finally, by choosing $\varepsilon\sim e^{-C\sqrt{\lambda}}$   such that
$$
\varepsilon e^{C_6(\sqrt{\lambda}+1)}   C_5e^{C_5\sqrt{\lambda} } \|\phi\|_{L^\infty(\Omega)} \leq \frac{1}{2}\sup_{\Omega}|\phi|,
$$
we obtain \eqref{4.40-w-g-s-10-15} from \eqref{equ-1126-12}. This completes the proof.
\end{proof}

 \begin{remark}\label{rem-1126}
$(i)$ The the spectral inequality \eqref{equ-1126-11} still holds if $\Omega$ is assumed to be a Lipschitz  and locally star-shaped domain, see \cite{Apraiz2014obser}. Clearly, Proposition \ref{prop-spec-high} is valid for such domain $\Omega$.

$(ii)$ The proof of Proposition~\ref{prop-spec-high} differs slightly from that of Lemma~\ref{lem-spec}, since the eigenfunctions may not be analytic up to the boundary, unless the boundary $\partial\Omega$ is analytic \cite{Morrey}.
\end{remark}

Next we establish the $d$-dim version of  Lemma \ref{lem-LS}. Similar to Definition \ref{definition4.1-w-10-4}, we can define
the space $l^2L^\infty(\R^d)$, with the norm:
\begin{align*}
 \|u\|_{l^2L^\infty(\R^d)}=\|\{\sup_{x\in Q(k)}|u(x)\}_{k\in\Z^d}\|_{l^2(\Z^d)},\;\; u\in l^2L^\infty(\R^d),
  \end{align*}
where $Q(k)$ is the unit closed cube centered at $k\in \Z^d$.

\begin{proposition}\label{prop-LS-high}
Let $\gamma, L>0$ and $\alpha>1, \beta=\frac{3}{2}$.  Suppose that
 $E\subset\R^d$ is a closed  and $F_{\alpha,\beta}$$-$$(\gamma,L)$ thick set. Then there exists a constant $C$, depending only on $d,\alpha,\gamma,L$, such that when $u\in l^2L^\infty(\R^d)$, with $\emph{supp } \widehat{u}\subset B_N(0)$ for some $N>0$,
\begin{align}\label{4.47-w-w-g-g-10-17}
\|u\|_{l^2L^\infty(\R^d)}\leq C\exp\{CN^2(\log (N+e))^{-\frac{2(\alpha-1-)}
{3}}\} \big\|\{\sup_{E\cap Q(k)}|u|\}_{k\in\Z^d}\big\|_{l^2(\Z^d)}.
\end{align}
\end{proposition}
\begin{proof}
 Since the proof is very similar to that of  Lemma \ref{lem-LS}, we only sketch it  and consider the case $L=1$.
 We arbitrarily fix $N>0$ and  $u\in l^2L^\infty(\R^d)$, with $\mbox{supp } \widehat{u}\subset B_N(0)$.
 We call  cube $Q(k)$ a good cube, if for each multi-index $\mu\in \N^d$,
 \begin{align}\label{equ-LS-high-1}
  \|D^\mu u\|_{L^\infty(Q(k))}\leq (AN)^{|\mu|}\|u\|_{L^\infty(Q(k))}
 \end{align}
 where $A>0$ is a positive constant which will be determined later.
 When  $Q(k)$ is not good, we call it a bad cube.

 Using the Bernstein inequality\footnote{It remains valid in the $d$-dim case, with the proof being nearly identical to that of Lemma \ref{lem-bern}.}  in Lemma \ref{lem-bern}
 and following the same approach to \eqref{eq-1118-h}, one can  show that, when $A$ is sufficiently large, the contribution of bad cubes is small, namely
 $$
 \sum_{\mbox{ \footnotesize bad } Q(k)}(\sup_{Q(k)}|u|)^2
\leq \frac{1}{2}\|u\|^2_{l^2L^\infty(\R^d)}.
 $$
Thus, we have
  \begin{align}\label{equ-LS-high-1.5}
 \sum_{\mbox{ \footnotesize good } Q(k)}\|u\|^2_{L^\infty(Q(k))}\geq \frac{1}{2}\|u\|^2_{l^2L^\infty(\R^d)}.
 \end{align}

 Now we consider a good $Q(k)$. It follows from  \eqref{equ-LS-high-1} that $u$ can be extended to an analytic function on $\C^n$. Moreover,  by the Taylor expansion, we find
 \begin{align}\label{equ-LS-high-2}
  |u(k+z)|\leq e^{C|z|N}\|u\|_{L^\infty(Q(k))}, \;\;z\in \C^d.
 \end{align}
Meanwhile, by our assumptions on $E$, as well as the definition of the $F_{\alpha,\beta}-{(\gamma,L)}$ thickness, we have that  $E\cap Q(k)$ is compact;  $c_{F_{\alpha,\beta}}(E\cap Q(k))$  has a uniform positive lower bound with respect to $k\in \Z^d$. Thus, we can apply  Proposition \ref{prop-ana-high-2} to find a positive constant $C$, depending only on  $d$, $\alpha$, $\gamma,L$, such that when
$\varepsilon>0$ is small,
\begin{align}\label{4.50-w-10-17}
\sup_{x\in Q(k)}|u(x)|\leq \varepsilon \sup_{z\in \mathbb{D}_{10}(0)}|u(k+z)|  +\exp\left\{\frac{C(\log \frac{1}{\varepsilon})^2 }{(\log\log \frac{1}{\varepsilon})^{2(\alpha-1-)/3}} \right\}\sup_{x\in E\cap Q(k)}|u(x)|.
\end{align}
 In light of \eqref{equ-LS-high-2}, we can choose $\varepsilon \sim e^{-CN}$  such that
$$
\varepsilon \sup_{z\in \mathbb{D}_{10}(0)}|u(k+z)|\leq \frac{1}{2}\sup_{x\in Q(k)}|u(x)|.
$$
This, together with \eqref{4.50-w-10-17}, leads to
$$
\sup_{x\in Q(k)}|u(x)|\leq C\exp\left\{ CN^2 (\log (N+e))^{-\frac{2(\alpha-1-)}{3}} \right\}\sup_{x\in E\cap Q(k)}|u(x)|.
$$
Taking the sum over all good $Q(k)$ in the above leads to
$$
\sum_{\mbox{ \footnotesize good }Q(k)}
\|u\|^2_{L^\infty(Q(k))}\leq C^2\exp\left\{ 2CN^2 (\log (N+e))^{-\frac{2(\alpha-1-)}{3}} \right\}\sum_{\mbox{ \footnotesize good }Q(k)}(\sup_{x\in E\cap Q(k)}|u(x)|)^2.
$$
With \eqref{equ-LS-high-1.5}, the above leads to \eqref{4.47-w-w-g-g-10-17}. This completes the proof.
\end{proof}

\begin{remark}
    The result in \eqref{prop-LS-high} remains true if we replace $l^2L^\infty$ by $l^qL^p$ with $1\leq q, p\leq\infty$. The proof is  essentially the same.
\end{remark}

\subsection{Proofs of Theorem \ref{thm-bound-high} and  Theorem \ref{thm-Rn-high}}\

By Proposition \ref{prop-spec-high} and Proposition \ref{prop-LS-high}, we are allowed to employ  the adapted  Lebeau-Robbiano strategy used in proving Theorem  \ref{thm-ob-bound}/Theorem  \ref{thm-ob-R} to establish Theorem \ref{thm-bound-high} and Theorem \ref{thm-Rn-high}.
We point out that the only difference lies in the requirement for the convergence of series appeared in \eqref{equ-1023-11}.
In view of the spectral inequality \eqref{4.40-w-g-s-10-15} and the uncertainty principle \eqref{4.47-w-w-g-g-10-17}, To ensure this convergence we need

$$
\frac{2(\alpha-1-)}{3}>1.
$$
 Since we have assumed that   $\alpha>\frac{5}{2}$, the above holds clearly.


Therefore by repeating the arguments in the proof of Theorem \ref{thm-ob-bound}, we complete the proofs of  Theorem \ref{thm-bound-high} and  Theorem \ref{thm-Rn-high}.

 \section{Application to control}\label{section-contr}

In this section, we  give applications of the main theorems in Section \ref{mainresult-w-10-29} to the null-controllability for heat equations.
We only present the application of Theorem \ref{thm-Rn-high}, since the  analogous results can be obtained for the others by using a similar approach.

Before presenting the application, we introduce several spaces related to a subset $E\subset\R^d$.
First, $C(E)$ and $\mathcal{M}(E)$ denote  the space of all continuous functions on $E$
and the space of all Borel measures on $E$, respectively. Let
\begin{align*}
    l^2C(E)=\{f\in C(E;\R)  :\; \|f\|_{l^2C(E)} := \|\{\sup_{Q_k\cap E}|f|\}_{k\in\Z^d}\|_{l^2(\Z^d)}<\infty \}
\end{align*}
 and its dual space is as follows (see \cite{DualofLplq}):
\begin{align*}
l^2\mathcal{M}(E)=\{\mu\;\;\mbox{is a Borel measure on } E\;:\;
\|\mu\|_{l^2\mathcal{M}(E)} :=
\|\{|\mu|(Q_k\cap E)\}_{k\in\Z^d}\|_{l^2(\Z^d)}<\infty\}.
   \end{align*}

\begin{theorem}[Null-controllablity for the equation \eqref{equ-heat-high-Rn}]\label{thm-control}
Let   $F_{\alpha,\beta}$ be given by \eqref{equ-intro-F} with $\alpha>\frac{5}{2}$, $\beta= \frac{3}{2}$.
Assume that $E\subset \R^d$ is a closed and  $F_{\alpha,\beta}$-$(\gamma, L)$ thick set for some $\gamma>0$ and $\L>0$.
Then for each $T>0$ and each $u_0\in L^2(\R^d)$,  there is a control  $\mu$ in the space:
\begin{align*}
\mathcal{U}_T:=\{\mu\in L^\infty(0,T; l^2\mathcal{M}(\R^d))\;\;:\;\;\mbox{supp} \;\mu\subset [0,T]\times E
\},
   \end{align*}
such that
\begin{align}\label{6.1-w-10-29}
\|\mu\|_{ L^\infty(0,T; l^2\mathcal{M}(\R^d))}\leq C\|u_0\|_{L^2(\R)}
   \end{align}
for some $C>0$ independent of $u_0$, and such that the solution $u$ to the controlled equation
\begin{align}\label{6.2-w-10-29}
\partial_tu(t,x)-\Delta u(t,x)=\mu(t,x),\;\;t>0, x\in\R^d;\;\; u|_{t=0}=u_0
\end{align}
satisfies
\begin{align}\label{6.3-w-10-29}
u|_{t=T}=0\;\;\mbox{over}\;\;\R^d.
\end{align}
\end{theorem}

\begin{remark}
    The proof of the above theorem is based on  the standard duality argument. Some ideas were borrowed from \cite{Burq2023propagation}.
\end{remark}

\begin{proof}[Proof of Theorem \ref{thm-control}]
We arbitrarily fix $T>0$ and $u_0\in L^2(\R^d)$. Given $w_0\in L^2(\R^d)$, $w(t):=e^{(T-t)\Delta}w_0$ is the solution to the following dual equation over $(0,T)$:
\begin{align}\label{equ-backheat}
\partial_tw+\Delta w=0, \quad w|_{t=T}=w_0.
\end{align}
Then according to Theorem \ref{thm-Rn-high}, the solution $w$ satisfies
\begin{align}\label{equ-nullcontro-1}
 \|w|_{t=0}\|_{L^2(\R^d)}\leq C\int_J\|\{\sup_{x\in E\cap Q(k)}e^{(T-t)\Delta}w_0(x)\}_{k\in \Z^d}\|_{l^2(\Z^d)}\d t
\end{align}
for some $C>0$ independent of $w_0$, where $J:=[0,T-\varepsilon_0]$ and  we   fix $\varepsilon_0>0$ small enough.

Next, define the   linear space
$$
X=\{e^{(T-t)\Delta}w_0|_{J\times E}: w_0\in L^2(\R^d)\}.
$$
It is clear that for each $w_0\in L^2(\R^d)$,  the restriction $e^{(T-t)\Delta}w_0|_{J\times E}$ belongs to the space $C(J\times E)$. Meanwhile, there is $\hat{C}>0$ such that for $t\in J$
\begin{align}\label{equ-nullcontro-0} \|e^{(T-t)\Delta}w_0\|_{l^2L^\infty(\R^d)}
&=\|G_{T-t}(x)*w_0\|_{l^2L^\infty(\R^d)}\nonumber\\&\leq \hat{C}\|G_{T-t}(x)\|_{l^2L^\infty(\R^d)}\|w_0\|_{L^2(\R^d)}\leq \hat{C}\|w_0\|_{L^2(\R^d)},\end{align}
where $G_t(x)=(4\pi t)^{-n/4}e^{\frac{-|x|^2}{4t}}$ is the heat kernel.
Thus, $X$ is a
 normed subspace of $L^1(J;l^2C(E))$ with the norm
\begin{align}\label{6.7-w-10-29}
\|w\|_X:=\|w\|_{L^1(J;l^2C(E))}=
\int_J  \|\{\sup_{x\in E\cap Q(k)}|w(t,x)|\}_{k\in\Z^d}\|_{l^2(\Z^d)}  \d t,\;\;w\in X.
\end{align}

Then,  we define a linear functional $\mathcal{F}$ over $X$ by
\begin{align}\label{6.8-w-10-29}
\mathcal{F}(w):=(w|_{t=0}, u_0)_{L^2(\R^d)},\;\;w\in X.
\end{align}
It follows from \eqref{6.8-w-10-29}, \eqref{equ-nullcontro-1} and \eqref{6.7-w-10-29} that
\begin{align}\label{6.9-w-10-29}
\|\mathcal{F}\|_{\mathcal{L}(X;\R) }\leq C\|u_0\|_{L^2(\R^d)},
\end{align}
where $C$ is given by \eqref{equ-nullcontro-1}.
According to  the Hahn-Banach theorem, there exists an $\hat{\mathcal{F}}\in \mathcal{L}(L^1(J;l^2C(E)); \R)$ such that
\begin{align}\label{6.10-w-10-29}
\hat{\mathcal{F}}=\mathcal{F}\;\;\mbox{over}\; X;\;\;\|\hat{\mathcal{F}}\|_{\mathcal{L}(L^1(J;l^2C(E));\R)}=\|\mathcal{F}\|_{\mathcal{L}(X;\R) }.
\end{align}
Then by
 the Riesz representation theorem, there exists $\mu\in L^\infty(J;l^2\mathcal{M}(E))$  such that
 \begin{align}\label{6.11-w-10-29}
\hat{\mathcal{F}}(\varphi)=\int_{J\times E}\varphi(t,x)\d \mu\;\;\mbox{for all}\;\;\varphi\in L^1(J;l^2C(E))
\end{align}
and
 \begin{align}\label{6.12-w-10-29}
\|\hat{\mathcal{F}}\|_{\mathcal{L}(L^1(J;l^2C(E));\R)}
=\|\mu\|_{ L^\infty(J;l^2\mathcal{M}(E))}.
\end{align}

 By the zero extension, we can treat $\mu$ as  an element in $L^\infty((0,T);l^2\mathcal{M}(\R^d))$
 such that $\mbox{supp}\; \mu\subset J\times E$. Thus, by \eqref{6.8-w-10-29},  \eqref{6.10-w-10-29}
 and \eqref{6.11-w-10-29}, we find
 \begin{align}\label{equ-control-mu}
(\varphi|_{t=0},u_0)_{L^2(\R^d)}=\int_0^T \varphi(t,x)\d \mu\;\;\mbox{for all}\;\;\varphi\in X.
\end{align}

Now we claim the solution $u$ to the equation \eqref{6.2-w-10-29} (with the above $\mu$) satisfies
\eqref{6.3-w-10-29}.
For this purpose, we first notice that $\mu\in L^\infty((0,T);l^2H^{-\sigma}(\R^d))$
(with $\sigma=[\frac{d}{2}]+1$), because of
 the Sobolev embedding $H^\sigma(\R^d)\hookrightarrow C(\R^d)$. Next, we arbitrarily fix $w_0\in L^2(\R^d)$ and write $w$ for the solution of \eqref{equ-backheat}. We let, for  every $N>0$,
  $P_N$ be the Fourier projection operator defined by $\widehat{P_Nf}=1_{|\xi|\leq N}\widehat{f}(\xi), \xi\in \R^d$.
  Let $\mu_N:=P_N\mu$, $u_N:=P_N u$ and $w_N:=P_Nw$. Then
  $\mu_N, u_N, w_N\in L^2((0,T);H^\infty(\R^d)).$
  Thus, we can use \eqref{6.2-w-10-29} and \eqref{equ-backheat} to get
\begin{align}\label{equ-nullcontro-3}
0&=\int_0^T((\partial_t+\Delta)w_N, u_N)_{L^2(\R^d)}\d t=(w_N, u_N)\big|_0^T+\int_0^T(w_N,(-\partial_t+\Delta)u
_N)_{L^2(\R^d)}  \d t\nonumber\\
&=(P_Nw_0, u_N|_{t=T})_{L^2(\R^d)}-(w_N(0), P_N u_0))_{L^2(\R^d)}-\int_0^T(w_N,\mu_N)_{L^2(\R^d)}\d t.
\end{align}
Meanwhile, we have that as $N\to\infty$,
\begin{align*}
 &P_Nw_0\to  w_0;\;\; u_N|_{t=T}\to u|_{t=T};\;\; P_N u_0\to u_0 \mbox{ in } L^2(\R^d),\\
 &w_N\to w \mbox{ in }L^2(J; l^2H^\sigma(\R^d)),\\
 &\mu_N\to \mu \mbox{ in } L^2(J; l^2H^{-\sigma}(\R^d)).
\end{align*}
(Here we used fact that $w\in L^2(J; l^2H^\sigma(\R^d))$, which can be proved by the similar way to
that showing \eqref{equ-nullcontro-0}.) Thus we can pass to the limit
for $N\to\infty$ in \eqref{equ-nullcontro-3} to get
\begin{align}\label{equ-1029-1}
0=(w_0,u|_{t=T})_{L^2(\R^d)}-(w|_{t=0}, u_0)_{L^2(\R^d)}  -  \int_0^Tw(t,x)\d \mu.
\end{align}
However, by \eqref{equ-control-mu} (where $\varphi$ is replaced by $w$), we have
\begin{align}\label{equ-1029-2}
\int_0^T w(t,x)\d \mu=(w_0,e^{T\Delta}u_0)_{L^2(\R^d)}.
\end{align}
Combining \eqref{equ-1029-1}-\eqref{equ-1029-2} yields that  for any $w_0\in L^2(\R^d)$,
$$
(w_0,u|_{t=T})_{L^2(\R^d)} =0,
$$
which leads to \eqref{6.3-w-10-29}.

Besides,
\eqref{6.1-w-10-29} follows from  \eqref{6.12-w-10-29}, \eqref{6.9-w-10-29} and \eqref {6.10-w-10-29}
at once. This completes the proof of Theorem \ref{thm-control}.
\end{proof}

\section*{Acknowledgment}
S. Huang, G. Wang and M. Wang are partially supported by the National Natural Science Foundation of China under grants No.12171442, 12371450 and 12171178, respectively.
S. Huang would like to thank Prof. Baowei Wang for very useful discussions on the construction in Proposition \ref{thm-1017-1.1}. We are deeply grateful to Professors A. Walton  Green, K{\'e}vin Le Balc'h,  J{\'e}r{\'e}my Martin and Marcu-Antone Orsoni, who are
the authors of \cite{walton2024observability}, for kindly pointing out an error in the proof of Proposition \ref{prop-spec-high}, enabling us to make timely corrections to the proof.




\end{document}